\documentclass[11pt]{amsart}
\usepackage[T1]{fontenc}
\usepackage{lmodern,microtype,amsmath,amssymb,amsthm,mathtools}
\usepackage[margin=1in]{geometry}
\usepackage[colorlinks=true,allcolors=blue]{hyperref}
\hypersetup{pdftitle={Positive sectional curvature on exotic spheres in infinitely many dimensions},pdfauthor={Fernando Galaz-Garcia}}
\numberwithin{equation}{section}
\newtheorem{theorem}{Theorem}[section]
\newtheorem{proposition}[theorem]{Proposition}
\newtheorem{lemma}[theorem]{Lemma}
\newtheorem{corollary}[theorem]{Corollary}

\newtheorem{theoremAlph}{Theorem}
\newtheorem{corollaryAlph}[theoremAlph]{Corollary} 

\theoremstyle{definition}\newtheorem{definition}[theorem]{Definition}
\theoremstyle{remark}\newtheorem{remark}[theorem]{Remark}

\newcommand{\R}{\mathbb R}\newcommand{\Z}{\mathbb Z}\newcommand{\Q}{\mathbb Q}\newcommand{\HH}{\mathbb H}
\newcommand{\RR}{\mathbb R}\newcommand{\ZZ}{\mathbb Z}\newcommand{\QQ}{\mathbb Q}
\newcommand{\HP}{\mathbb{H}P}\newcommand{\Ahat}{\hat{A}}\newcommand{\Zh}{\widehat{Z}}
\DeclareMathOperator{\sgn}{sign}
\DeclareMathOperator{\SO}{SO}\DeclareMathOperator{\im}{Im}

\newcommand{\id}{\mathrm{id}}
\title[Exotic spheres in infinitely many dimensions]{Positive sectional curvature on exotic spheres in infinitely many dimensions}
\author{Fernando Galaz-Garc\'ia}
\address{Department of Mathematical Sciences, Durham University, Durham, United Kingdom}
\email{fernando.galaz-garcia@durham.ac.uk}
\date{\today}
\begin{document}
\begin{abstract}
We show that there exist exotic spheres with metrics of strictly positive sectional curvature in infinitely many
dimensions: in every dimension $n\equiv3\pmod4$ with $n\geq11$ and $n\not\equiv255\pmod{256}$. Previously, every
known closed simply connected manifold of positive sectional curvature in a dimension above $24$ was diffeomorphic to
a compact rank-one symmetric space. The spheres we consider are certain homotopy spheres arising as quotients of special
$S^3$--$S^3$ bundles introduced by Speran\c{c}a. These homotopy spheres are determined by an orthogonal representation $\rho$ of
$S^3$ and an $S^3$-conjugation-equivariant map $\beta$, and are denoted $\Sigma_{\rho,\beta}$. By recent
constructions, such spheres admit Riemannian metrics of strictly positive sectional curvature. In this paper we consider the
spheres $\Sigma_{\rho_b,\beta}$ of dimension $4b+3$, $b\geq2$, where $\rho_b$ is the representation of $S^3$ on $\R^3\oplus\HH^b$
by left multiplication on the second summand. For these spheres, we show that the Eells--Kuiper invariant is the product of an
integer $k(\beta)$ depending only on $\beta$ and a rational number $r_b$ given by one coefficient of a power series
built from the $\hat A$- and $L$-genera. We then show that, for every $b\geq2$ with $b\not\equiv63\pmod{64}$, there is
a map $\beta$ for which the Eells--Kuiper invariant does not vanish, and thus the corresponding sphere is exotic. A computer calculation also gives the dimensions $n\equiv255\pmod{256}$ below
$2^{16}$.
\end{abstract}

\maketitle
\section{Introduction and results}
The construction of Riemannian metrics with positive sectional curvature is a longstanding problem in global Riemannian geometry and examples have been scarce. Besides the compact rank-one symmetric spaces (spheres, complex and quaternionic projective spaces, and the Cayley plane), also known as CROSSes, the closed simply connected examples known before the summer of 2026 are homogeneous spaces in dimensions $6$, $7$, $12$, $13$ and $24$, found by Berger, Wallach and Aloff--Wallach~\cite{Berger,Wallach,AloffWallach}, biquotients in dimensions $6$, $7$ and $13$, found by Eschenburg and Bazaikin~\cite{Eschenburg,Bazaikin}, and a cohomogeneity-one manifold in dimension $7$, homeomorphic but not diffeomorphic to the unit tangent bundle of $S^4$, found by Grove, Verdiani and Ziller~\cite{GVZ} and by Dearricott~\cite{Dearricott}; see the surveys~\cite[pp.~5--10 and~13]{Ziller} and~\cite[pp.~121 and~125--126]{GroveSurvey}. In particular, in dimensions above $24$, all known simply connected compact positively curved manifolds are diffeomorphic to CROSSes~\cite[p.~26]{WilkingSurvey}; see also \cite[p.~607]{WilkingAnnals}. Recent preprints propose positively curved metrics on products of spheres: Brendle and Hung on $S^2\times S^2$~\cite{BrendleHung}, which would contradict Hopf's conjecture that $S^2\times S^2$ admits no such metric, and Guo, Fang and Lu on $S^2\times S^3$ and on $S^3\times S^3$~\cite{GFL23,GFL33}; since $\chi(S^3\times S^3)=0$, the last would also contradict Hopf's conjecture that positively curved even-dimensional manifolds have positive Euler characteristic. For both conjectures, see~\cite[p.~4]{Ziller}. These manifolds have dimensions at most $6$.

Much of the structure theory of positive curvature assumes symmetry, following the \emph{Grove symmetry programme}, proposed by Grove~\cite[\S2]{GroveSymm}. Hsiang and Kleiner showed that a closed orientable positively curved $4$-manifold with a nontrivial isometric circle action is homeomorphic to $S^4$ or $\mathbb{C}P^2$~\cite[Theorem~1]{HsiangKleiner}. Grove and Searle bounded the rank of an effective isometric torus action on a closed connected positively curved $n$-manifold by $[(n+1)/2]$, and showed that equality implies that the manifold must be diffeomorphic to a sphere, a real or complex projective space, or a lens space~\cite[Theorem, p.~138]{GroveSearle}. Wilking showed that an effective isometric torus action of rank at least $n/4+1$ on a closed simply connected positively curved $n$-manifold, $n\geq10$, implies that the manifold is homeomorphic to a sphere or to $\HP^{n/4}$, or homotopy equivalent to $\mathbb{C}P^{n/2}$~\cite[Theorem~2, p.~260]{WilkingActa}. He also showed that such a manifold is tangentially homotopy equivalent to a CROSS, or homogeneous, when its isometry group is sufficiently large or its cohomogeneity $k\geq1$ is small compared with its dimension~\cite[Theorems~1 and~2 and Corollary~3, p.~608]{WilkingAnnals}; for a homotopy sphere of dimension $n\geq18(k+1)^2$ the conclusion is an equivariant diffeomorphism to a linear sphere or to a Brieskorn sphere, with an action of $\mathrm{Sp}(d)$ induced by a linear one~\cite[Theorem~4, pp.~609--610]{WilkingAnnals}. Fang and Rong considered almost maximal symmetry rank~\cite{FangRong}. Grove, Wilking and Ziller reduced simply connected positively curved cohomogeneity-one manifolds in odd dimensions to an explicit list of examples and candidates~\cite{GWZ}. By Straume's classification of cohomogeneity-one actions on homotopy spheres~\cite[Theorem~C, p.~4]{Straume}, together with work of Back and Hsiang~\cite[Theorem~VI.3, p.~222]{BackHsiang} and, in dimension five, of Searle~\cite[Theorem, p.~491]{Searle}, \cite{SearleCorr}, a positively curved cohomogeneity-one homotopy sphere is equivariantly diffeomorphic to the standard sphere with a linear action; see~\cite[Theorem~2.7]{GWZ}. Further results with torus symmetry concern the Euler characteristic and rational cohomology~\cite{Kennard,KWW}. All these results assume a relatively large amount of symmetry: high torus rank, high symmetry degree or low cohomogeneity.

Exotic spheres are a natural test class for investigating how curvature depends on the smooth structure. In dimensions at least five, a homotopy sphere carries a metric of positive scalar curvature if and only if its $\alpha$-invariant vanishes~\cite{Hitchin}, \cite[Theorem~A]{Stolz}, and every homotopy sphere that bounds a parallelisable manifold carries a metric of positive Ricci curvature~\cite{Wraith}. For sectional curvature, strong pinching puts the smooth structure deep in Gromoll's filtration~\cite[p.~370]{Gromoll}, and strict $1/4$-pinching forces the standard structure, by the differentiable sphere theorem of Brendle and Schoen~\cite{BrendleSchoen}. Nonnegative sectional curvature is known on the Gromoll--Meyer sphere~\cite{GM}, on the Milnor spheres~\cite{GroveZiller} and on every exotic $7$-sphere~\cite[Theorem~A]{GKS}, and Petersen and Wilhelm proposed a positively curved metric on the Gromoll--Meyer sphere~\cite{PW}.

New methods for  constructing positive curvature on exotic spheres have appeared recently. Ouyang~\cite{Ouyang} and Guo, Fang and Lu~\cite{GFL} deform metrics on the Gromoll--Meyer sphere. He, Liu and Yau consider every homotopy $7$-sphere~\cite{HLY}: they build positively curved metrics on two disks that meet the hypotheses of the gluing theorem of Reiser and Wraith~\cite{RW}. Speran\c ca modelled the exotic $8$-sphere and a homotopy $10$-sphere of order three as quotients of special $S^3$--$S^3$ bundles~\cite[Theorem~1, p.~3182]{Sp}, and the compatible-disk method gives the exotic $8$-sphere and the homotopy $10$-spheres of order three metrics of positive sectional curvature~\cite{GG810}. Deng, Hu and Zhang weakened the hypotheses this method needs, and put positive sectional curvature on a family of homotopy $11$-spheres~\cite{DHZ}. Further results cover the homotopy $9$-spheres with vanishing $\alpha$-invariant~\cite{CDNZ} and a family of exotic $15$-spheres~\cite{CDZ}. These examples lie in dimensions $7$, $8$, $9$, $10$, $11$ and $15$.

Apart from the deformation methods and one of the constructions in~\cite{CDNZ}, these constructions pass through quotients of special $S^3$--$S^3$ bundles in the sense of Speran\c ca~\cite{Sp}, modelled on the Gromoll--Meyer bundle; such a bundle has two commuting free $S^3$-actions, one of which is denoted by $\star$~\cite[\S1.1, p.~3183]{Sp}; see also~\cite[\S1.1]{CavenaghiSperanca}. By Speran\c ca's construction~\cite[\S3.1, pp.~3186--3187]{Sp}, when such a bundle has trivialisations over two disks compatible with both actions, its $\star$-quotient is a twisted sphere; every star bundle has such trivialisations~\cite[Proposition~3.1]{GG810}. Explicitly, take an oriented Euclidean space $V$, an orthogonal representation $\rho\colon S^3\to\SO(V)$ and a smooth map $\beta\colon S(V)\to S^3$ with $\beta(\rho(q)y)=q\beta(y)q^{-1}$, which is conjugation equivariance in Speran\c ca's terms~\cite[(3.1), p.~3186]{Sp}. Following Dur\'an and Rigas~\cite[p.~206]{DR}, we call the diffeomorphism $\hat\beta(y)=\rho(\beta(y))y$ of $S(V)$ a \emph{reentrance}. The $\star$-quotient is diffeomorphic to
\[
 \Sigma_{\rho,\beta}=D(V)\cup_{\sigma_\beta}\bigl(-D(V)\bigr),\qquad \sigma_\beta=\hat\beta^{-1},
\]
and we call a homotopy sphere of this form a \emph{star homotopy sphere}, or \emph{star sphere}, after the star bundles of~\cite[Definition~2.2]{GG810}. By the compatible-disk construction of Deng, Hu and Zhang~\cite[Proposition~5.1]{DHZ} and the gluing theorem of Reiser and Wraith~\cite[Theorem~A(i)]{RW}, every star sphere carries a metric of strictly positive sectional curvature; see~\cite[Remark~4.4]{GG810}. Hence, within this framework, finding positively curved exotic spheres reduces to a question of topology: which elements of the group $\Theta_n$ of oriented homotopy $n$-spheres are star spheres? In dimension $7$, the two-disk models $\Sigma^7_n$ of Dur\'an, P\"uttmann and Rigas, used by He, Liu and Yau~\cite[\S2, p.~5]{HLY}, represent $n$ modulo $28$ in $\Theta_7\cong\Z/28$, and $\Sigma^7_1$ is the Gromoll--Meyer sphere~\cite[Theorem~1 and its proof, pp.~270 and~278]{DPR}. These models are star spheres for the representation of $S^3$ on $\operatorname{Im}\HH\oplus\HH$ by conjugation~\cite[\S3, p.~208]{DR}. Hence every homotopy $7$-sphere, the Gromoll--Meyer sphere included, is a star sphere.

In this paper, we show that exotic star spheres exist in infinitely many dimensions. For $b\geq1$ let
\[
 V_b=\R^3\oplus\HH^b,\qquad \rho_b(q)(p,w)=(p,qw),
\]
where $S^3$ acts trivially on $\R^3$ and by left multiplication on each quaternionic coordinate. The action of $S^3$ on $S(V_b)$ is semifree: it fixes the $2$-sphere $S(\R^3)$ and is free elsewhere. Deng, Hu and Zhang studied the case $b=2$~\cite{DHZ}, and Chen, Deng and Zhang the case $b=3$~\cite{CDZ}.

\begin{theoremAlph}
\label{thm:main.thm.A}
Let $b\geq2$ with $b\not\equiv63\pmod{64}$. Then there is an exotic $(4b+3)$-sphere that is a star homotopy sphere for $\rho_b$.
\end{theoremAlph}

Combined with the curvature theorem for star spheres (Theorem~\ref{thm:curvature-general}; cf.~\cite[Remark~4.4]{GG810}), Theorem~\ref{thm:main.thm.A} has the following consequence.

\begin{corollaryAlph}
    \label{cor:main.cor.B}
In every dimension $n\equiv3\pmod4$ with $n\geq11$ and $n\not\equiv255\pmod{256}$, there is an exotic $n$-sphere with a Riemannian metric of strictly positive sectional curvature.
\end{corollaryAlph}

Since an exotic sphere is not diffeomorphic to a CROSS, Corollary~\ref{cor:main.cor.B} gives the first closed simply connected manifolds of positive sectional curvature in dimensions above $24$ that are not diffeomorphic to a CROSS, in infinitely many dimensions. 

The proofs of Theorem~\ref{thm:main.thm.A} and Corollary~\ref{cor:main.cor.B} do not require computer calculation. For $b\equiv63\pmod{64}$, which Theorem~\ref{thm:main.thm.A} excludes, a general argument has eluded us, and an exact computer calculation shows that there is an exotic sphere of positive sectional curvature in every dimension $n\equiv3\pmod4$ with $11\leq n<2^{16}$ (see Corollary~\ref{cor:computation}). For the representations $\rho_b$ we consider, the infinitesimal action fields $K_y$ of $\rho_b$ (Section~\ref{subsec:curvature}) satisfy $\|K_y\|\leq1$ (see Lemma~\ref{lem:action-bound}). Hence, Corollary~\ref{cor:main.cor.B} follows directly from~\cite[Theorem~4.1]{GG810}, whose proof has been formalised in Lean~4~\cite{GGLean}, with the Reiser--Wraith gluing theorem~\cite{RW} as an explicit hypothesis.

The proof of Theorem~\ref{thm:main.thm.A} rests on three results. The first is a formula for the Eells--Kuiper invariant~\cite{EK} of every star sphere arising from the representation $\rho_b$. Before stating it, we define the \emph{slice degree}: a smooth conjugation-equivariant map $\beta$ for $\rho_b$ determines an integer $k(\beta)$ given by the degree of $\beta$ on a slice through a free orbit of the induced $S^3$-action on the invariant $6$-sphere $S(\R^3\oplus\HH w_0)\subset S(V_b)$, where $w_0\in\HH^b$ is a unit vector and $\HH w_0$ is the quaternionic line it spans; the closure of this slice is a $3$-disk whose boundary lies in the fixed $2$-sphere (see Definition~\ref{def:ek-slice-degree} and Lemma~\ref{lem:ek-slice-degree}). The Eells--Kuiper invariant of the star spheres we consider is a multiple of the slice degree. 

\begin{theoremAlph}
    \label{thm:main.thm.C}
Let $b\geq2$, let $m=b+1$, and let $\beta\colon S(V_b)\to S^3$ be a smooth conjugation-equivariant map. Then
\begin{align*}
 \mu(\Sigma_{\rho_b,\beta})&=k(\beta)\,r_b\in\Q/\Z,\\
 r_b&=\frac{1}{a_m}\,[y^{2m}]\left(\frac{y\sqrt{1+y^2/4}}{2\operatorname{arsinh}(y/2)}+\frac{1}{2^{2m+1}(2^{2m-1}-1)}\cdot\frac{y}{(1+y^2)\operatorname{artanh}y}\right),
\end{align*}
where $[y^{2m}]f$ denotes the coefficient of $y^{2m}$ in the Taylor expansion of $f$ at $y=0$, and $a_m=1$ for $m$ even and $a_m=2$ for $m$ odd.
\end{theoremAlph}

It follows that $r_2=1/5952$, recovering the computation of Deng, Hu and Zhang in dimension $11$~\cite{DHZ}, and $r_3=-1/12192$, recovering the value $+k/12192$ of Chen, Deng and Zhang in dimension $15$~\cite{CDZ}, whose $k$ is the Euler-class coefficient of their filling, that is, $-k(\beta)$ in our conventions (see Lemma~\ref{lem:ek-euler} and Remark~\ref{rem:ek-conventions}). The proof generalises their filling to every $b$. 

The spheres of Theorem~\ref{thm:main.thm.A} carry a smooth action of $S^3$ that fixes a $3$-sphere and is free elsewhere. Outside an open invariant tubular neighbourhood of the fixed $3$-sphere, the $S^3$-action on $\Sigma_{\rho_b,\beta}$ is a principal bundle over $D^4\times\HP^{b-1}$, and its mapping cylinder, closed off over the fixed $3$-sphere by $D^4\times\HP^b$, is a spin manifold with boundary $\Sigma_{\rho_b,\beta}$. The Eells--Kuiper invariant of $\Sigma_{\rho_b,\beta}$, computed from this filling, is $k(\beta)$ times a combination of characteristic numbers and the signature. A Lagrange--B\"urmann identity~\cite[Theorem~2.1.1, p.~214]{Gessel}, \cite[Theorem~5.4.2, p.~38]{StanleyEC2} reduces that combination to the single coefficient $r_b$ in Theorem~\ref{thm:main.thm.C} (see Lemma~\ref{lem:lagrange}).

The second result shows that, for every $b$, some slice degree has a small power of $2$. Section~\ref{sec:degrees} extends the obstruction-theoretic argument of Chen, Deng and Zhang~\cite[\S3.1 and Lemma~3.1]{CDZ} to every $b$. The conjugation-equivariant maps for $\rho_b$ correspond to families of automorphisms of the quaternionic Hopf bundle over $\HP^{b-1}$, and the obstructions to realising a slice degree lie in the groups $\pi_{4j+2}(S^3)$, $1\leq j\leq b-1$, which are finite by the work of Serre~\cite[Chapter~V, Proposition~3, p.~494]{Serre}. Recall that the exponent of a finite abelian group is the least common multiple of the orders of its elements. The product $k_0(b)$ of the exponents of these groups is a slice degree (see Proposition~\ref{prop:degree-tower}). The first $p$-torsion in $\pi_*(S^3)$ occurs in $\pi_{2p}(S^3)$~\cite[(13.6)$'$, p.~177]{Toda}, so every prime factor of $k_0(b)$ is at most $2b-1$. By James's theorem, $\pi_*(S^3)$ has no element of order $8$~\cite[Corollary~(1.22), p.~80]{James57}, so the power of $2$ dividing $k_0(b)$ is at most $2^{2(b-1)}$; Toda's tables improve this to $2^{2b-5}$ for $b\geq5$~\cite[Tables~I--II, pp.~186--187]{Toda} (see Corollary~\ref{cor:k0-arithmetic}).

The third result is arithmetic. The bracket in Theorem~\ref{thm:main.thm.C} is a sum of two terms, one coming from the $\hat A$-genus and one from the $L$-genus. Using the double-angle identities for $\tanh$ and $\cosh$, we find, for every $b$, the exact power of $2$ in the denominator of each term: it is $2^m$ for the $\hat A$-term and $2^{2m-1-s}$ for the $L$-term, where $2^s$ is the largest power of $2$ dividing $m$ (see Propositions~\ref{prop:A-valuation} and~\ref{prop:L-valuation}). These powers are different, so they determine the power of $2$ in the denominator of $r_b$ (see Corollary~\ref{cor:rb-valuation}). Comparing it with the bound on $k_0(b)$ shows that $k_0(b)\,r_b$ is not an integer when $b\not\equiv63\pmod{64}$, so the Eells--Kuiper invariant of $\Sigma_{\rho_b,\beta}$ with $k(\beta)=k_0(b)$ is nonzero.

The Eells--Kuiper invariant is additive under connected sum~\cite[Remark, p.~99]{EK}. For a fixed representation $\rho_b$, the connected sum of the star spheres of two maps is the star sphere of their pointwise product (Lemma~\ref{lem:reentrance}(3)), so the star spheres for $\rho_b$ form a subgroup of $\Theta_{4b+3}$. Hence the construction gives further exotic spheres of positive sectional curvature. For $2\le b\le6$, let $q_b$ be the order of $\mu(\Sigma_{\rho_b,\beta})=k_0(b)\,r_b$ in $\Q/\Z$ for a map $\beta$ with $k(\beta)=k_0(b)$. The star spheres for $\rho_b$ represent at least $q_b$ distinct elements of $\Theta_{4b+3}$, of which at least $q_b-1$ are exotic, and $q_b$ ranges from $496$ in dimension $11$ to $1048448$ in dimension $27$ (see Remark~\ref{rem:orders} and Table~\ref{tab:orders}). In dimension $11$ the multiples of such a sphere represent exactly the subgroup $2\Theta_{11}$ (cf.~\cite[Theorem~1.1]{DHZ}).

Our article is organised as follows. Section~\ref{sec:prelim} recalls star spheres and proves the properties of them that we use, recalls the curvature theorems and the Eells--Kuiper invariant, and introduces the representations $\rho_b$. Section~\ref{sec:ek} proves Theorem~\ref{thm:main.thm.C}. Section~\ref{sec:degrees} shows that the degree $k_0(b)$ is realised, and bounds its prime factors and the power of $2$ dividing it. Section~\ref{sec:twoadic} computes the $2$-adic valuation of $r_b$. Section~\ref{sec:main} proves Theorem~\ref{thm:main.thm.A} and Corollary~\ref{cor:main.cor.B}, and treats the class $b\equiv63\pmod{64}$ by an exact computation for $b<2^{14}$.

\subsection*{Acknowledgements and AI Assistance}
Part of this work was carried out while the author was a core participant in the Dual Trimester Program Geometric Statistics: Theory, Application, and Computation at the Hausdorff Research Institute for Mathematics (HIM), University of Bonn. The author thanks HIM for its hospitality and excellent working conditions. The author thanks Lewis Tadman for helpful comments on a first draft of this paper and for suggesting Remarks~\ref{rem:connected-sums} and~\ref{rem:orders}.

Generative AI tools (OpenAI ChatGPT Astra and Anthropic Claude Opus 5.5) were used under the author's direction to assist with mathematical exploration, drafting, writing the computer programs, literature searches, and internal review of this work. The author independently checked the mathematical arguments and references in the final manuscript, made all final decisions concerning its contents, and takes full responsibility for the paper. The new mathematical arguments in the author's earlier work on homotopy spheres in dimensions eight and ten~\cite{GG810} have been formalised in Lean~4 and Mathlib~\cite{GGLean}, with results from the literature represented by explicit hypotheses. A Lean 4 formalisation of the present work has also been developed.

\section{Preliminaries}\label{sec:prelim}

In this section, we fix notation and recall the facts about star homotopy spheres, their curvature and the Eells--Kuiper invariant that we use. We then introduce the semifree representations $\rho_b$ and list some of their elementary properties. Throughout, $G=S^3$ is the group of unit quaternions, and $G_{\mathrm{conj}}$ denotes $G$ with the action of $G$ on itself by conjugation.

\subsection{Star homotopy spheres}\label{subsec:star}
Let $V$ be an oriented Euclidean space of dimension $n$ and let $\rho\colon G\to\SO(V)$ be an orthogonal representation. A map $\beta\colon S(V)\to G_{\mathrm{conj}}$ is \emph{conjugation-equivariant} (\emph{equivariant} for short) if
\[
 \beta(\rho(q)y)=q\,\beta(y)\,q^{-1}\qquad(q\in G,\ y\in S(V)).
\]
If $\beta$ is smooth, its \emph{reentrance} is the map $\hat\beta(y)=\rho(\beta(y))y$, and its \emph{attaching diffeomorphism} is
\[
 \sigma_\beta(y)=\rho\bigl(\beta(y)^{-1}\bigr)y,
\]
the inverse of $\hat\beta$; both are orientation-preserving $G$-equivariant diffeomorphisms of $S(V)$ (see Lemma~\ref{lem:reentrance}). The \emph{star homotopy sphere} associated with $\rho$ and $\beta$ is the oriented smooth manifold
\begin{equation}\label{eq:star-sphere}
 \Sigma_{\rho,\beta}=D(V)\cup_{\sigma_\beta}\bigl(-D(V)\bigr),
\end{equation}
where a point $y$ of the boundary of the first disk is identified with the point $\sigma_\beta(y)$ of the boundary of the second, and $-D(V)$ denotes the disk with the opposite orientation~\cite[Chapter~VI, \S5, p.~101]{Kosinski}, \cite[Lemma~2.4(2)]{GG810}. We call $\Sigma_{\rho,\beta}$ a \emph{star sphere}. Since $\sigma_\beta$ commutes with $\rho$, the linear actions of $G$ on the two disks glue to a smooth action of $G$ on $\Sigma_{\rho,\beta}$.

For $n\geq7$, let $\Theta_n$ be the group of oriented diffeomorphism classes of smooth homotopy $n$-spheres under connected sum~\cite{KM}. The twisted-sphere construction identifies $\Theta_n$ with the group of orientation-preserving diffeomorphisms of $S^{n-1}$ modulo those that extend over $D^n$~\cite[Chapter~VIII, Corollary~5.6, pp.~158--159]{Kosinski}, and composition of attaching maps corresponds to connected sum~\cite[Chapter~VI, Lemma~5.2, p.~101]{Kosinski}. Two conjugation-equivariant maps $\beta_0,\beta_1\colon S(V)\to G_{\mathrm{conj}}$ are \emph{equivariantly homotopic} if there is a continuous map $H\colon S(V)\times[0,1]\to G_{\mathrm{conj}}$ with $H(\cdot,0)=\beta_0$, $H(\cdot,1)=\beta_1$ and $H(\rho(q)y,t)=q\,H(y,t)\,q^{-1}$ for all $q\in G$, $y\in S(V)$ and $t\in[0,1]$. If $H$ is smooth, the map is a \emph{smooth equivariant homotopy}. Write $E(\rho)$ for the group, under pointwise multiplication, of equivariant homotopy classes of continuous equivariant maps $S(V)\to G_{\mathrm{conj}}$. By the equivariant smooth approximation theorem~\cite[Chapter~VI, Theorem~4.2, p.~317]{Bredon}, every class contains smooth maps, and two smooth maps in the same class are joined by a smooth equivariant homotopy.

The following lemma collects the properties of star spheres that we use. It goes back to the equivariant reentrance construction of Dur\'an and Rigas~\cite[Theorem~1.1 and \S2]{DR}.

\begin{lemma}\label{lem:reentrance}
Let $\beta$ and $\gamma$ be smooth conjugation-equivariant maps.
\begin{enumerate}
\item $\hat\beta\circ\hat\gamma=\widehat{\gamma\beta}$, and $\hat\beta$ and $\sigma_\beta=\widehat{\beta^{-1}}$ are orientation-preserving $G$-equivariant diffeomorphisms of $S(V)$ with $\sigma_\beta\circ\sigma_\gamma=\sigma_{\beta\gamma}$.
\item If $\beta$ and $\gamma$ are equivariantly homotopic, then $\sigma_\beta$ and $\sigma_\gamma$ are smoothly isotopic.
\item If $n\geq7$, then
\[
 R_\rho\colon E(\rho)\to\Theta_n,\qquad [\beta]\mapsto[\Sigma_{\rho,\beta}],
\]
is a well-defined homomorphism, and $R_\rho([1])=R_\rho([-1])=0$. In particular $\Sigma_{\rho,\beta^{-1}}=-\Sigma_{\rho,\beta}$, where $-\Sigma$ denotes $\Sigma$ with the opposite orientation.
\end{enumerate}
\end{lemma}

\begin{proof}
(1) Let $y\in S(V)$ and let $z=\hat\gamma(y)=\rho(\gamma(y))y$. Since $\beta$ is conjugation-equivariant, $\beta(z)=\gamma(y)\beta(y)\gamma(y)^{-1}$. Hence
\[
 \hat\beta(\hat\gamma(y))=\rho(\beta(z))z=\rho\bigl(\gamma(y)\beta(y)\gamma(y)^{-1}\bigr)\rho(\gamma(y))y=\rho\bigl(\gamma(y)\beta(y)\bigr)y=\widehat{\gamma\beta}(y),
\]
that is, $\hat\beta\circ\hat\gamma=\widehat{\gamma\beta}$. The pointwise inverse $\beta^{-1}$ is again a smooth conjugation-equivariant map, and $\hat1=\id$. Applying this identity to the pair $(\beta,\beta^{-1})$ gives $\hat\beta\circ\widehat{\beta^{-1}}=\id$, and applying it to the pair $(\beta^{-1},\beta)$ gives $\widehat{\beta^{-1}}\circ\hat\beta=\id$. Hence $\hat\beta$ is a diffeomorphism, with inverse $\widehat{\beta^{-1}}=\sigma_\beta$. Applying the identity to the pair $(\beta^{-1},\gamma^{-1})$ gives
\[
 \sigma_\beta\circ\sigma_\gamma=\widehat{\gamma^{-1}\beta^{-1}}=\widehat{(\beta\gamma)^{-1}}=\sigma_{\beta\gamma}.
\]
The computation of $\hat\beta(\hat\gamma(y))$ uses the equivariance of $\beta$ but not that of $\gamma$. Taking for $\gamma$ the constant map with value $q\in G$, so that $\hat\gamma(y)=\rho(q)y$, the same computation gives $\hat\beta(\rho(q)y)=\rho(q)\hat\beta(y)$. Hence, $\hat\beta$ is $G$-equivariant, and so is $\sigma_\beta=\widehat{\beta^{-1}}$.

For the orientation, extend $\beta$ to $\tilde\beta(x)=\beta(x/|x|)$ on $V\setminus0$ and set $\Phi(x)=\rho(\tilde\beta(x))x$. Then $\Phi$ commutes with dilations, preserves the norm and restricts to $\hat\beta$; in particular $d\Phi_x(x)=\Phi(x)$, so $\hat\beta$ preserves the orientation of $S(V)$ if $\det d\Phi_x>0$ for all $x$. Fix $x\neq0$ and $a=\tilde\beta(x)$, let $\mathfrak g$ be the Lie algebra of $G$, and define $K\colon\mathfrak g\to V$ by $KX=d\rho(X)x$ and $\xi\colon V\to\mathfrak g$ by $\xi=a^{-1}\,d\tilde\beta_x$. Then
\[
 d\Phi_x=\rho(a)\,(I_V+K\xi).
\]
Differentiating $\tilde\beta(\rho(e^{tX})x)=e^{tX}a\,e^{-tX}$ at $t=0$ gives $\xi KX=\mathrm{Ad}_{a^{-1}}X-X$, that is, $I_{\mathfrak g}+\xi K=\mathrm{Ad}_{a^{-1}}$. Since $\det(I_V+K\xi)=\det(I_{\mathfrak g}+\xi K)$ by Sylvester's determinant theorem~\cite[\S0.8.5]{HornJohnson} (one sees this by eliminating either off-diagonal block of $\left(\begin{smallmatrix}I_V&-K\\ \xi&I_{\mathfrak g}\end{smallmatrix}\right)$), we get $\det d\Phi_x=\det\rho(a)\,\det\mathrm{Ad}_{a^{-1}}=1$. Since $\sigma_\beta=\hat\beta^{-1}$ is the inverse of an orientation-preserving diffeomorphism, it also preserves the orientation.

(2) Let $H\colon S(V)\times[0,1]\to G_{\mathrm{conj}}$ be an equivariant homotopy from $\beta$ to $\gamma$, with $G$ acting trivially on $[0,1]$. After reparametrising the interval, we may assume that $H(y,t)=\beta(y)$ for $t\leq1/3$ and $H(y,t)=\gamma(y)$ for $t\geq2/3$. Extend $H$ to $S(V)\times\R$, a manifold without boundary, by $H(y,t)=\beta(y)$ for $t\leq0$ and $H(y,t)=\gamma(y)$ for $t\geq1$. The extended map is smooth on a neighbourhood of the closed invariant set $A=S(V)\times\bigl((-\infty,0]\cup[1,\infty)\bigr)$. By the equivariant smooth approximation theorem relative to $A$~\cite[Chapter~VI, Theorem~4.2, p.~317]{Bredon}, $H$ is equivariantly homotopic relative to $A$ to a smooth equivariant map $H'$. For each $t$, the map $\beta_t=H'(\cdot,t)$ is smooth and conjugation-equivariant, with $\beta_0=\beta$ and $\beta_1=\gamma$, and it depends smoothly on $t$. By~(1), the maps $\sigma_{\beta_t}$ are diffeomorphisms of $S(V)$ depending smoothly on $t$, so they form a smooth isotopy from $\sigma_\beta$ to $\sigma_\gamma$.

(3) If $\sigma_0$ and $\sigma_1$ are isotopic, then $\sigma_1\sigma_0^{-1}$ is isotopic to the identity and therefore extends over $D(V)$, by the isotopy on a collar of $S(V)$ and the identity inside the collar. By~\cite[Chapter~VI, Lemma~5.2, p.~101]{Kosinski}, the twisted sphere of a diffeomorphism that extends over the disk is standard, and the twisted sphere of a composite is the connected sum of the twisted spheres. Hence, by item~(2), $[\Sigma_{\rho,\beta}]$ depends only on $[\beta]$, and by item~(1), $R_\rho([\beta][\gamma])=R_\rho([\beta])+R_\rho([\gamma])$. For $\beta=\pm1$, the map $\sigma_\beta$ is the restriction of $\rho(\pm1)\in\SO(V)$, which extends over $D(V)$.
\end{proof}

\subsection{Positive sectional curvature}\label{subsec:curvature}
In this subsection, we let $n=\dim V\geq3$. Star spheres are the quotients of star bundles by their free $\star$-action~\cite[Lemma~2.4 and Proposition~3.1]{GG810}. The compatible-disk construction of Deng, Hu and Zhang~\cite[Proposition~5.1]{DHZ}, applied to the reentrance $\sigma_\beta$ of $\beta^{-1}$, together with the gluing theorem of Reiser and Wraith~\cite[Theorem~A(i)]{RW}, gives the following general positive curvature result, noted in~\cite[Remark~4.4]{GG810}.

\begin{theorem}[{\cite[Remark~4.4]{GG810}}]\label{thm:curvature-general}
Every star homotopy sphere admits a smooth Riemannian metric of strictly positive sectional curvature.
\end{theorem}

\begin{remark}\label{rem:connected-sums}
Positive sectional curvature is not preserved under connected sum in general. For $n\ge3$, the manifold
$\R P^n\#\R P^n$ has infinite fundamental group, so it carries no metric of positive sectional curvature by Myers'
theorem~\cite[Corollary, p.~404]{Myers}, and the connected sum of sufficiently many copies of $\mathbb{C}P^2$ carries no
metric of nonnegative sectional curvature by Gromov's Betti number theorem~\cite[0.2.A and Corollary, p.~181]{Gromov}.
Star spheres of a fixed representation behave differently. By Lemma~\ref{lem:reentrance}(3), the connected sum of the
star spheres of two maps for $\rho$ is the star sphere of their pointwise product, so it carries a metric of positive
sectional curvature by Theorem~\ref{thm:curvature-general}. Hence, if $\Theta_n$ is generated by star spheres for a
single representation, then every homotopy $n$-sphere carries such a metric. This is the case for $n=7$. The
Gromoll--Meyer sphere generates $\Theta_7\cong\Z/28$~\cite[Theorem~1, p.~270]{DPR} and is a star sphere for the
representation of $S^3$ on $\operatorname{Im}\HH\oplus\HH$ by conjugation (see the introduction), which recovers the
theorem of He, Liu and Yau~\cite{HLY} that every homotopy $7$-sphere carries a metric of positive sectional curvature.
\end{remark}

The construction of~\cite{GG810}, which adapts that of He, Liu and Yau~\cite{HLY}, assumes that the operator norm of the infinitesimal action fields of the representation is bounded above by $2$. For $y\in S(V)$ and $\xi\in\im\HH$, let
\[
 K_y\xi=\frac{d}{dt}\Big|_{t=0}\rho\bigl(e^{t\xi}\bigr)y\in T_yS(V),\qquad
 \|K_y\|=\sup_{|\xi|=1}|K_y\xi|,
\]
with the norm on $\im\HH$ induced from $\HH$ and the round metric on $S(V)$.

\begin{theorem}[{\cite[Theorem~4.1 and Lemma~2.4(3)]{GG810}}]\label{thm:curvature-bounded}
If $\|K_y\|\leq2$ for every $y\in S(V)$, then every star sphere $\Sigma_{\rho,\beta}$ admits a smooth Riemannian metric of strictly positive sectional curvature.
\end{theorem}

The representations $\rho_b$ of Section~\ref{subsec:semifree} satisfy this bound (see Lemma~\ref{lem:action-bound}). 

\subsection{The Eells--Kuiper invariant}\label{subsec:prelim-ek}
Let $\Sigma$ be an oriented homotopy $(4m-1)$-sphere, $m\geq2$, that bounds a compact spin manifold $W$. Since $\Sigma$ is a homotopy sphere, the homomorphism $j\colon H^{4i}(W,\Sigma;\Q)\to H^{4i}(W;\Q)$ induced by the inclusion is an isomorphism for $0<i<m$. Hence each Pontryagin class $p_i(W)$, $0<i<m$, is the image under $j$ of a unique class $j^{-1}(p_i(W))\in H^{4i}(W,\Sigma;\Q)$~\cite[p.~94]{EK}, and Eells and Kuiper~\cite[(6) and~(7), p.~97]{EK} define
\[
 \mu(\Sigma)=\frac{N_m[W]+t_m\,\tau[W]}{a_m}\mod 1,
\]
where $N_m=\Ahat_m-t_mL_m$ is the combination of the $\Ahat$- and $L$-polynomials in which $p_m$ cancels, so that it involves only $p_1,\dots,p_{m-1}$; $t_m=-1/\bigl(2^{2m+1}(2^{2m-1}-1)\bigr)$~\cite[(9), p.~99]{EK}; $a_m=1$ for $m$ even and $a_m=2$ for $m$ odd; and $\tau[W]$ is the signature of $W$, that is, the signature of the symmetric bilinear form $(\alpha,\alpha')\mapsto\langle\alpha\smile\alpha',[W,\Sigma]\rangle$ on $H^{2m}(W,\Sigma;\RR)$~\cite[p.~965]{MilnorDiffSpheres}. The value does not depend on $W$~\cite[Theorem, p.~97]{EK}. In Section~\ref{sec:ek} we use this definition with the notation of Definition~\ref{def:ek-mu}, where $\sgn(W)=\tau[W]$.

The standard sphere bounds a disk, and the star spheres we consider bound spin manifolds constructed in Section~\ref{sec:ek}. We use the following properties: $\mu$ is an invariant of oriented diffeomorphism types, $\mu(S^{4m-1})=0$, and $\mu(-\Sigma)=-\mu(\Sigma)$~\cite[(6) and Theorem, p.~97; Remark, p.~99]{EK}. In particular, a homotopy sphere $\Sigma$ with $\mu(\Sigma)\neq0$ is not diffeomorphic to the standard sphere, that is, it is exotic.

\subsection{The semifree representations}\label{subsec:semifree}
For $b\geq1$, let
\[
 V_b=\R^3\oplus\HH^b,\qquad \rho_b(q)(p,w)=(p,qw)\qquad(p\in\R^3,\ w=(w_1,\dots,w_b)\in\HH^b),
\]
where $qw=(qw_1,\dots,qw_b)$. We orient $\HH$ by the basis $1,i,j,k$ and $V_b$ by the product orientation. Then $\dim V_b=4b+3$. Since left multiplication is free on $\HH^b\setminus0$, the action of $G$ on $S(V_b)$ fixes the $2$-sphere $S(\R^3)$ and is free on its complement. The unit sphere is the join $S(V_b)=S(\R^3)*S(\HH^b)$, and $S(\HH^b)\to S(\HH^b)/G=\HP^{b-1}$ is the quaternionic Hopf bundle.

\begin{lemma}\label{lem:action-bound}
For every $b\geq1$ and every $y\in S(V_b)$, we have $\|K_y\|\leq1$.
\end{lemma}
This was observed for $b=2$ and $b=3$ in~\cite[\S5.3]{DHZ} and~\cite[(7.1)]{CDZ}.

\begin{proof}
Let $y=(p,w)\in S(V_b)$, so that $|p|^2+|w|^2=1$, and let $\xi\in\im\HH$. The map $t\mapsto e^{t\xi}$ is a curve of unit quaternions with derivative $\xi$ at $t=0$, and $\rho_b(e^{t\xi})y=(p,e^{t\xi}w)$. Since right multiplication by $w_k$ is linear on $\HH$, differentiating at $t=0$ gives
\[
 K_y\xi=(0,\xi w),\qquad \xi w=(\xi w_1,\dots,\xi w_b).
\]
This vector is tangent to $S(V_b)$ at $y$, since $\langle(0,\xi w),(p,w)\rangle=\operatorname{Re}\sum_k\xi w_k\bar w_k=\operatorname{Re}(\xi)\,|w|^2=0$, because $\xi$ is purely imaginary. The norm of quaternions is multiplicative, so $|\xi w_k|=|\xi|\,|w_k|$ for each $k$, and therefore
\[
 |K_y\xi|^2=\sum_k|\xi|^2|w_k|^2=|\xi|^2|w|^2 .
\]
Hence $|K_y\xi|=|\xi|\,|w|$. Taking the supremum over $|\xi|=1$ gives $\|K_y\|=|w|\leq1$, with equality if and only if $p=0$.
\end{proof}

By Lemma~\ref{lem:action-bound} and Theorem~\ref{thm:curvature-bounded}, every star sphere for $\rho_b$ admits a metric of strictly positive sectional curvature. The case $b=1$ gives only the standard $7$-sphere (Remark~\ref{rem:b-one}), and we assume $b\geq2$ from Section~\ref{sec:ek} on.

\begingroup\renewcommand{\theenumi}{\roman{enumi}}
\section{The Eells--Kuiper invariant of semifree star spheres}\label{sec:ek}

In this section, we compute the Eells--Kuiper invariant of every star sphere of the
semifree representations
\[
  V_b=\RR^3\oplus\HH^b,\qquad \rho_b(q)(p,w)=(p,qw),\qquad b\ge 2,
\]
and prove Theorem~\ref{thm:main.thm.C}.
The proof relies on the fact that the group $G=S^3$ acts on $\Sigma_{\rho_b,\beta}$ with a $3$-sphere of fixed
points and is free elsewhere. We will show that, outside an open invariant tubular neighbourhood of the fixed
sphere, the action is a principal $S^3$-bundle over $D^4\times\HP^{b-1}$ (Lemma~\ref{lem:ek-twisted} and
Proposition~\ref{prop:ek-filling}(v)), and that its mapping cylinder, closed off over the fixed sphere by a
copy of $D^4\times\HP^b$, is a spin manifold $Z$ with boundary $\Sigma_{\rho_b,\beta}$
(Proposition~\ref{prop:ek-filling}).
The rest of the proof computes the characteristic numbers of $Z$. For $b=2$ and
$b=3$ this filling is the one used by Deng, Hu and Zhang~\cite{DHZ} and by Chen, Deng and
Zhang~\cite{CDZ}, respectively, and our computation recovers their results.

\subsection{Conventions and statement}\label{subsec:ek-statement}

Throughout, let $b\ge2$, $n=4b+3=\dim V_b$ and $m=b+1$, so that $\Sigma_{\rho_b,\beta}$ is a
homotopy $(4m-1)$-sphere. We write $\rho$ for $\rho_b$ when $b$ is fixed.

\emph{Orientations.} The space $\HH$ is oriented by the basis $1,i,j,k$; $\RR^3$, $\HH^b$, $V_b$ and
$\RR\oplus V_b$ carry the product orientations, in the order written. The unit sphere of an
oriented Euclidean space, and more generally the boundary of an oriented manifold, is
oriented by the outward normal first~\cite[Proposition~15.24, p.~386]{Lee}. In particular
$S^3\subset\HH$ is oriented as the boundary of the unit disk.

\emph{The star sphere.} Recall from~\eqref{eq:star-sphere} that
$\Sigma_{\rho,\beta}=D(V_b)\cup_{\sigma_\beta}\bigl(-D(V_b)\bigr)$, with attaching diffeomorphism
$\sigma_\beta(y)=\rho(\beta(y)^{-1})y$. In coordinates $y=(p,w)$,
\begin{equation}\label{eq:ek-sigma}
  \sigma_\beta(p,w)=\bigl(p,\ \beta(p,w)^{-1}w\bigr).
\end{equation}
Since $\sigma_\beta$ commutes with $\rho$, the linear actions on the two disks give a smooth
action of $G$ on $\Sigma_{\rho,\beta}$ (see Lemma~\ref{lem:reentrance}(1)).

\emph{Quaternionic lines.} Throughout this section, $G$ acts on $\HH^b$ by left multiplication. A \emph{line} is a left $\HH$-submodule $\ell=\HH w$ of $\HH^b$,
$w\ne0$, and $\HP^{b-1}$ is the space of lines.  The space $\HP^{b-1}$ is the orbit space $S(\HH^b)/G$ of the
action $q\cdot w=qw$. The tautological bundle $\gamma=\gamma_{b-1}$ has fibre $\ell$ over
$\ell$. A line $\ell$ carries the \emph{canonical orientation}, the image of the orientation of
$\HH$ under $h\mapsto hw$ for any $w\in\ell\setminus0$. This orientation does not depend on $w$. If
$w'=cw$ with $c\in\HH\setminus0$, then $h\mapsto hw'$ is $h\mapsto hw$ composed with right multiplication by $c$,
which preserves the orientation of $\HH$, since its determinant is $|c|^4>0$. Left multiplication
by $i$ makes $\gamma$ a complex vector bundle of rank two, and its complex orientation, given
by the real basis $w,iw,jw,ijw=kw$, is the canonical one. Let
\[
  x=c_2(\gamma)=e(\gamma)\in H^4(\HP^{b-1};\ZZ).
\]
Here $e$ denotes the Euler class of an oriented real vector bundle~\cite[\S9]{MS}, and $c_2$ the second
Chern class of a complex vector bundle~\cite[\S14]{MS}. For a complex vector bundle of rank two, with its
complex orientation, the top Chern class is the Euler class, $c_2=e$~\cite[\S14, Definition, p.~158]{MS}.
Every quaternionic line bundle below is oriented canonically, so its Euler class is its second Chern class.

\emph{Slices.} For a smooth action of $G$ on a manifold $X$ and a point $x\in X$, a subset $S\subset X$
containing $x$ is a \emph{slice} at $x$ if $G_x(S)=S$ and the map $G\times_{G_x}S\to X$, $[g,s]\mapsto gs$, is
an equivariant embedding onto an open neighbourhood of the orbit $Gx$~\cite[Chapter~II, Definition~4.1, p.~82]{Bredon}.
Fix $w_0\in S(\HH^b)$, and let $\HH w_0$ be the quaternionic line spanned by $w_0$. The subspace
$\RR^3\oplus\HH w_0$ of $V_b$ is invariant, since $\rho_b(q)(p,hw_0)=(p,(qh)w_0)$. On its unit sphere, a
$6$-sphere, $G$ fixes $S(\RR^3)$ and acts freely elsewhere. Let
\[
  D_{w_0}=\bigl\{\bigl(p,(1-|p|^2)^{1/2}w_0\bigr):p\in\RR^3,\ |p|<1\bigr\}.
\]

\begin{lemma}[The slice]\label{lem:ek-slice}
Let $w_0\in S(\HH^b)$. The map
\[
  \Psi\colon G\times D_{w_0}\longrightarrow S(\RR^3\oplus\HH w_0)\setminus S(\RR^3),\qquad
  \Psi\bigl(q,(p,tw_0)\bigr)=(p,tqw_0),
\]
is a $G$-equivariant diffeomorphism, where $G$ acts on $G\times D_{w_0}$ by left multiplication on the first factor.
Its inverse is $(p,hw_0)\mapsto\bigl(h/|h|,(p,|h|w_0)\bigr)$. In particular, $D_{w_0}$ is a slice at each of its
points for the action of $G$ on $S(\RR^3\oplus\HH w_0)$, and every point of $S(\RR^3\oplus\HH w_0)\setminus S(\RR^3)$
is $\rho_b(q)y$ for exactly one $q\in G$ and $y\in D_{w_0}$.
\end{lemma}

\begin{proof}
For $(p,tw_0)\in D_{w_0}$ we have $t=(1-|p|^2)^{1/2}>0$. Since $|q|=1$, the point $\Psi(q,(p,tw_0))$ has norm $1$
and a nonzero $\HH w_0$-component, so $\Psi$ is well defined. It is smooth, and
$\rho_b(q')\Psi(q,y)=\Psi(q'q,y)$, so it is equivariant. On $S(\RR^3\oplus\HH w_0)\setminus S(\RR^3)$ the coefficient
$h$ of $w_0$ is nonzero, and $|h|=(1-|p|^2)^{1/2}$, so the stated inverse is well defined and smooth. The two
composites are the identity, since $\Psi\bigl(h/|h|,(p,|h|w_0)\bigr)=(p,hw_0)$, and
$\bigl(tq/|tq|,(p,|tq|w_0)\bigr)=\bigl(q,(p,tw_0)\bigr)$ because $t>0$ and $|q|=1$. The action of $G$ is free off
$S(\RR^3)$, so the stabiliser $G_y$ is trivial for $y\in D_{w_0}$. By
definition~\cite[Chapter~II, Definition~4.1, p.~82]{Bredon}, it then suffices that $\Psi$ is an equivariant
embedding onto an open neighbourhood of each orbit $Gy$. Its image is open, since $S(\RR^3)$ is closed.
\end{proof}

\begin{definition}[Slice degree]\label{def:ek-slice-degree}
For $w_0\in S(\HH^b)$ let
\[
  s_{w_0}\colon D^3\longrightarrow S(V_b),\qquad s_{w_0}(p)=\bigl(p,(1-|p|^2)^{1/2}w_0\bigr),
\]
where $D^3\subset\RR^3$ is the closed unit disk with its standard orientation. The map $s_{w_0}$ is a
diffeomorphism from the interior of $D^3$ onto the slice $D_{w_0}$, and it sends $\partial D^3$ into the fixed
sphere $S(\RR^3)$. The \emph{slice degree} of $\beta$ is
\[
  k(\beta)=\deg\bigl(\beta\circ s_{w_0}\colon(D^3,\partial D^3)\to(S^3,\varepsilon_\beta)\bigr)\in\ZZ,
\]
where $\varepsilon_\beta\in\{\pm1\}$ is the constant value of $\beta$ on $S(\RR^3)$
(see Lemma~\ref{lem:ek-slice-degree}(i)).
\end{definition}

We use Lemma~\ref{lem:ek-slice} in two ways. First, $\beta(\rho_b(q)y)=q\,\beta(y)\,q^{-1}$ by equivariance. So, by
Lemma~\ref{lem:ek-slice}, the restriction of $\beta$ to the invariant $6$-sphere $S(\RR^3\oplus\HH w_0)$ is
determined by $\beta\circ s_{w_0}$ on $D^3$, whose boundary $s_{w_0}$ maps onto the fixed sphere $S(\RR^3)$. The slice
degree is the degree of this map, and Lemma~\ref{lem:ek-slice-degree} shows that it is an invariant of the class of
$\beta$. Second, $s_{w_0}$ is the parametrisation of the slice in which we compute. It gives the attaching map on the
slice (Lemma~\ref{lem:ek-slice-degree}(iv)), the map $g$ of the normal form (Section~\ref{subsec:ek-normal}), the
Euler number in Lemma~\ref{lem:ek-euler}, and the realisation of degrees (Lemma~\ref{lem:realisation}).

\begin{lemma}\label{lem:ek-slice-degree}
Let $\beta,\beta'\colon S(V_b)\to G_{\mathrm{conj}}$ be smooth and equivariant.
\begin{enumerate}
\item $\beta$ is constant on $S(\RR^3)$, with value $\varepsilon_\beta\in\{\pm1\}$.
\item $k(\beta)$ does not depend on $w_0$, and depends only on the class $[\beta]\in E(\rho_b)$ (Section~\ref{sec:prelim}).
\item $k(\beta\beta')=k(\beta)+k(\beta')$ and $k(-\beta)=k(\beta)$. Thus $k\colon E(\rho_b)\to\ZZ$
  is a homomorphism, and $k(\beta^{-1})=-k(\beta)$.
\item For $|p|<1$, the quaternion $\beta(s_{w_0}(p))^{-1}$ is the unique $q\in G$ with
  $\sigma_\beta(s_{w_0}(p))=\rho(q)s_{w_0}(p)$. Thus $-k(\beta)$ is the degree of the
  rotation by which the attaching map moves the slice.
\end{enumerate}
\end{lemma}

\begin{proof}
(i) Suppose $y\in S(\RR^3)$. Then $\rho(q)y=y$ for every $q$, so $\beta(y)=q\beta(y)q^{-1}$ and
$\beta(y)$ lies in the centre $\{\pm1\}$ of $G$. As $S(\RR^3)$ is connected, $\beta$ is
constant there.

(ii) Let $w_1\in S(\HH^b)$. Since $S(\HH^b)$ is connected, there is a path $w_t$, $0\le t\le1$, from $w_0$
to $w_1$. The maps $\beta\circ s_{w_t}$ form a homotopy $D^3\times[0,1]\to S^3$, and each of them sends
$\partial D^3$ to $\varepsilon_\beta$, because $s_{w_t}(\partial D^3)\subset S(\RR^3)$. So $\beta\circ s_{w_0}$ and $\beta\circ s_{w_1}$
are homotopic as maps of pairs $(D^3,\partial D^3)\to(S^3,\varepsilon_\beta)$, and have the same degree. Now let
$\beta_t$, $0\le t\le1$, be an equivariant homotopy from $\beta$ to $\beta'$. By~(i), each $\beta_t$ is constant on $S(\RR^3)$, with value
$\varepsilon_{\beta_t}\in\{\pm1\}$. For a fixed $y\in S(\RR^3)$, the path $t\mapsto\beta_t(y)$ is continuous
with values in the discrete set $\{\pm1\}$, so it is constant, and $\varepsilon_{\beta_t}=\varepsilon_\beta$ for every $t$.
Hence $\beta_t\circ s_{w_0}$ is a homotopy of maps of pairs $(D^3,\partial D^3)\to(S^3,\varepsilon_\beta)$, and
$k(\beta)=k(\beta')$.

(iii) Recall that $\varepsilon_\beta$ is the value of $\beta$ on $S(\RR^3)$. Left multiplication by $-1$ is an
orientation-preserving isometry of $S^3$ that sends $\varepsilon_\beta$ to $-\varepsilon_\beta=\varepsilon_{-\beta}$, so
$k(-\beta)=k(\beta)$. As $\varepsilon_{\beta\beta'}=\varepsilon_\beta\varepsilon_{\beta'}$, we may therefore assume
$\varepsilon_\beta=\varepsilon_{\beta'}=1$. Let $u=\beta\circ s_{w_0}$ and $v=\beta'\circ s_{w_0}$, both maps of pairs
$(D^3,\partial D^3)\to(S^3,1)$. Their pointwise product is $uv=(\beta\beta')\circ s_{w_0}$. Regard $u$, $v$ and $uv$ as
maps $D^3/\partial D^3=S^3\to S^3$. Then $uv=\nu\circ(u\times v)\circ\Delta$, where $\Delta$ is the diagonal and
$\nu\colon S^3\times S^3\to S^3$ is the multiplication. By the K\"unneth theorem,
\[
  H_3(S^3\times S^3;\ZZ)=H_3(S^3)\otimes H_0(S^3)\oplus H_0(S^3)\otimes H_3(S^3).
\]
Both projections composed with $\Delta$ are the identity, so $\Delta$ maps $[S^3]$ to
$[S^3]\otimes[\mathrm{pt}]+[\mathrm{pt}]\otimes[S^3]$. The map $u\times v$ sends this to
$\deg u\cdot[S^3]\otimes[\mathrm{pt}]+\deg v\cdot[\mathrm{pt}]\otimes[S^3]$, and $\nu$ sends each summand to $[S^3]$ times
its coefficient, because $\nu(g,1)=\nu(1,g)=g$. Hence $\deg(uv)=\deg u+\deg v$, that is,
$k(\beta\beta')=k(\beta)+k(\beta')$. Taking $\beta'=\beta^{-1}$ gives the last claim.

(iv) This is \eqref{eq:ek-sigma}, together with the freeness of left multiplication on
$\HH^b\setminus0$.
\end{proof}

When $\beta$ equals $1$ near $S(\RR^3)$, the map $p\mapsto\bigl(w\mapsto\beta(p,(1-|p|^2)^{1/2}w)\,w\bigr)$ is a
family of automorphisms of the Hopf bundle $S(\HH^b)\to\HP^{b-1}$ covering the identity, and $k(\beta)$ is
the degree of its evaluation at $w_0$. This is the degree of $p\mapsto b(p,w_0)$
in~\cite[(3.4)]{DHZ}, and the evaluation degree of~\cite[\S3.1]{CDZ}. Lemma~\ref{lem:ek-euler} gives a second intrinsic
description, namely that $-k(\beta)$ is the second Chern number, over a $4$-sphere $S^4\times\{\ell\}$, of a quaternionic line bundle
that extends the bundle associated with the free part of the action on $\Sigma_{\rho,\beta}$ (Section~\ref{subsec:ek-filling}).

\emph{The numbers.} We recall the definition of a multiplicative sequence~\cite[pp.~219--220]{MS}. Let
$A=\bigoplus_{n\ge0}A^n$ be a commutative graded $\QQ$-algebra with unit, and let $A^\Pi$ be the ring of formal sums
$a=a_0+a_1+a_2+\cdots$ with $a_n\in A^n$. A sequence of polynomials $K_1(x_1),K_2(x_1,x_2),\dots$ with rational
coefficients, with $K_n$ homogeneous of degree $n$ when $x_i$ has degree $i$, is a \emph{multiplicative sequence} if
$K(ab)=K(a)K(b)$ for all such $A$ and all $a,b\in A^\Pi$ with leading term $1$. Here
\[
  K(a)=1+K_1(a_1)+K_2(a_1,a_2)+\cdots .
\]
For each power series $Q(z)=1+\lambda_1z+\lambda_2z^2+\cdots$ there is exactly one multiplicative sequence with
$K(1+z)=Q(z)$~\cite[Lemma~19.1, p.~221]{MS}. If the total
Pontryagin class of a real vector bundle is written formally as $\prod_i(1+z_i)$, then
$K(p)=\sum_jK_j(p_1,\dots,p_j)=\prod_iQ(z_i)$; $Q$ is the \emph{characteristic power series} of the sequence.
The $L$-genus is the multiplicative sequence with characteristic power series
$Q_L(z)=\sqrt z/\tanh\sqrt z$~\cite[Theorem~19.4, p.~224]{MS}, and the $\Ahat$-genus the one with
$Q_{\Ahat}(z)=\tfrac12\sqrt z/\sinh\tfrac12\sqrt z$~\cite[p.~278]{AtiyahHirzebruch}. We write $L_m$ and $\Ahat_m$ for their
components of degree $4m$, as polynomials in $p_1,\dots,p_m$.

Let $a_m=2$ for $m$ odd and $a_m=1$ for $m$ even, and
\[
  t_m=-\frac{1}{2^{2m+1}\,(2^{2m-1}-1)} .
\]
For a power series
$Q(U)$ with $Q(0)=1$, let
\begin{equation}\label{eq:ek-FQ}
  F_Q(b)=[U^{b+1}]\,\frac{Q(U)^{2b+2}}{Q(4U)}-\frac{1}{b+1}\,[U^{b}]\,\frac{\bigl(Q(U)^{2b+2}\bigr)'}{Q(4U)},
\end{equation}
where $[U^j]$ denotes the coefficient of $U^j$ and the prime is $d/dU$, and set
$F_L(b)=F_{Q_L}(b)$, $F_{\Ahat}(b)=F_{Q_{\Ahat}}(b)$ and
\[
  r_b=\frac{F_{\Ahat}(b)-t_m F_L(b)}{a_m}\in\QQ,\qquad m=b+1 .
\]
Here $t_m$ is the ratio of the coefficients of $p_m$ in $\Ahat_m$ and in $L_m$
(see Lemma~\ref{lem:ek-normalisation}). We will first show that the Eells--Kuiper invariant of the star spheres we consider can be expressed as a multiple of the slice degree. We will then obtain a closed form of the factor $r_b$ in Theorem~\ref{thm:main.thm.C}.

\begin{theorem}\label{thm:ek}
Let $b\ge2$ and let $\beta\colon S(V_b)\to G_{\mathrm{conj}}$ be a smooth equivariant map. Then,
in the conventions above,
\[
  \mu(\Sigma_{\rho_b,\beta})=k(\beta)\,r_b\qquad\text{in }\QQ/\ZZ .
\]
\end{theorem}

\begin{remark}[Other conventions]\label{rem:ek-conventions}
The sign in Theorem~\ref{thm:ek} depends on three choices.
\begin{enumerate}
\item \emph{The gluing map.} If the two disks are glued by the reentrance
  $\hat\beta=\sigma_{\beta^{-1}}$ instead of $\sigma_\beta$, the sphere is
  $\Sigma_{\rho,\beta^{-1}}$, and by Theorem~\ref{thm:ek} and Lemma~\ref{lem:ek-slice-degree}(iii) its Eells--Kuiper invariant is
  $-k(\beta)r_b$. This matches the convention $J_\beta=\hat\beta$ of~\cite[(3.6)]{DHZ}
  and~\cite[\S1.3]{CDZ}.
\item \emph{Left or right multiplication.} Let $\rho'_b(q)(p,w)=(p,wq^{-1})$ be the action by right multiplication used
  in~\cite[(1.1)]{CDZ}, and let $c(p,w)=(p,\bar w)$, with $\bar w$ the componentwise
  conjugate. Then $c\circ\rho_b(q)=\rho'_b(q)\circ c$, the map $\beta'=\beta\circ c^{-1}$ is
  $\rho'_b$-equivariant, and $\sigma'_{\beta'}=c\,\sigma_\beta\,c^{-1}$. So $c$ induces a
  diffeomorphism $\Sigma_{\rho_b,\beta}\to\Sigma_{\rho'_b,\beta'}$ of degree
  $\det c=(-1)^{3b}=(-1)^b$, since $c$ acts on both disks. The slice degrees agree, because
  $\beta'(p,t\bar w_0)=\beta(p,tw_0)$. Hence
  $\mu(\Sigma_{\rho'_b,\beta'})=(-1)^b k(\beta')\,r_b$ for $\rho'_b$.
\item \emph{The orientation of $V_b$.} Reversing this orientation reverses the orientations of both disks and
  the sign of $\mu$.
\end{enumerate}
In~\cite{DHZ,CDZ} the two disks are oriented by pulling back the orientation of a
filling, so those papers fix no orientation of $V_b$. Our computation reproduces their
values when their orientation conventions are imposed (Remark~\ref{rem:ek-checks}).
\end{remark}

\emph{Plan of the proof.} In Section~\ref{subsec:ek-normal}, we put $\beta$ in a normal form and
write $\Sigma_{\rho,\beta}$ as a twisted union of $D^4\times S^{4b-1}$ and $S^3\times D^{4b}$.
In Section~\ref{subsec:ek-filling}, we construct the filling $Z$ and determine the orientations. In
Sections~\ref{subsec:ek-ring}--\ref{subsec:ek-pont}, we compute the cohomology ring of the
closed-up filling, its Pontryagin classes and its signature. In Section~\ref{subsec:ek-mu} we recall the
Eells--Kuiper invariant and identify the polynomial $N_m$. In Section~\ref{subsec:ek-evaluation} we evaluate the
Eells--Kuiper invariants and prove Theorem~\ref{thm:ek}, and in Section~\ref{subsec:ek-closed} we give the closed
form of $r_b$.

\subsection{A normal form and the free part}\label{subsec:ek-normal}

In this subsection, we first replace the conjugation-equivariant map $\beta\colon S(V_b)\to G_{\mathrm{conj}}$ by an
equivariantly homotopic map that equals $1$ near the fixed sphere $S(\RR^3)$ of the action $\rho_b$ of $G$ on
$S(V_b)$ (Lemma~\ref{lem:ek-normal-form}; if $\beta$ equals $-1$ there, we first replace $\beta$ by $-\beta$). Such a
map is said to be in \emph{normal form}; the precise condition is stated after Lemma~\ref{lem:ek-normal-form}. The replacement
preserves the star sphere up to orientation-preserving diffeomorphisms and does not change the slice degree.
We then use the normal form to write $\Sigma_{\rho,\beta}$ as a twisted union of $D^4\times S^{4b-1}$ and
$S^3\times D^{4b}$. This decomposition is adapted to the fixed set and the free part of the action of $G$ on
$\Sigma_{\rho,\beta}$.

We begin with the normal form. Let $R=2^{-1/2}$ and, for $\delta\in(0,1)$, let
\[
  U_\delta=\{(p,w)\in S(V_b):|w|\le\delta\},
\]
an invariant neighbourhood of the fixed sphere $S(\RR^3)$.

\begin{lemma}\label{lem:ek-normal-form}
Every smooth equivariant map $\beta\colon S(V_b)\to G_{\mathrm{conj}}$ is equivariantly homotopic to a smooth
equivariant map equal to $\varepsilon_\beta$ on $U_{4/5}$, where $\varepsilon_\beta\in\{\pm1\}$ is the value of $\beta$ on
$S(\RR^3)$ (Lemma~\ref{lem:ek-slice-degree}(i)). If $\beta$ equals $\varepsilon_\beta$ on $U_{4/5}$,
then $\varepsilon_\beta\beta$ equals $1$ there, $\Sigma_{\rho,\varepsilon_\beta\beta}$ is
diffeomorphic to $\Sigma_{\rho,\beta}$ preserving orientation, and
$k(\varepsilon_\beta\beta)=k(\beta)$.
\end{lemma}

\begin{proof}
We first construct the homotopy. Choose a smooth nondecreasing function
\[
  \lambda\colon[0,1]\longrightarrow[0,1]
\]
with $\lambda=0$ on $[0,4/5]$ and $\lambda=1$ near $1$, and let
$\lambda_t=(1-t)+t\lambda$ for $0\le t\le1$. Define
\[
  \phi_t\colon S(V_b)\longrightarrow S(V_b),\qquad
  \phi_t(p,w)=\frac{(p,\lambda_t(|w|)w)}{|(p,\lambda_t(|w|)w)|}.
\]
The denominator does not vanish: if $\lambda_t(|w|)w=0$ and $p=0$, then $|w|=1$ and
$\lambda_t(1)=1$. Each $\phi_t$ is smooth, since $\lambda_t(|w|)$ is constant near $w=0$, and
commutes with $\rho$, which acts on $w$ by left multiplication. Hence
\[
  \beta\circ\phi_t,\qquad 0\le t\le1,
\]
is a smooth equivariant homotopy from $\beta$. Its end $\beta\circ\phi_1$ equals
$\varepsilon_\beta$ on $U_{4/5}$, because
\[
  \phi_1(U_{4/5})\subset S(\RR^3)
\]
and $\beta=\varepsilon_\beta$ on $S(\RR^3)$.

We next prove the second statement. If $\varepsilon_\beta=1$, there is nothing to prove. If
$\varepsilon_\beta=-1$, then $\varepsilon_\beta\beta=-\beta$. By
Lemma~\ref{lem:reentrance}(3), $R_\rho$ is a homomorphism with $R_\rho([-1])=0$, so
\[
  [\Sigma_{\rho,-\beta}]=R_\rho([-1])+R_\rho([\beta])=[\Sigma_{\rho,\beta}],
\]
and $k(-\beta)=k(\beta)$ by Lemma~\ref{lem:ek-slice-degree}.
\end{proof}

By Lemma~\ref{lem:ek-normal-form}, Lemma~\ref{lem:ek-slice-degree}(ii) and the homotopy
invariance of $\Sigma_{\rho,\beta}$ (Lemma~\ref{lem:reentrance}(3)), it is enough to prove
Theorem~\ref{thm:ek} for $\beta$ in \emph{normal form}, that is,
\[
  \beta=1\quad\text{on }U_{4/5}.
\]
We assume this from now on. For $p\in D^3_R$, the closed disk of radius $R$ in $\RR^3$, and
$w\in S^{4b-1}=S(\HH^b)$, set
\[
  g(p,w)=\beta\bigl(p,(1-|p|^2)^{1/2}w\bigr)\in S^3 .
\]
Then
\[
  g(p,qw)=q\,g(p,w)\,q^{-1}\qquad(q\in G),
\]
and $g(p,w)=1$ for $|p|\ge3/5$. For such $p$ and every $w\in S^{4b-1}$, the point
$y=(p,(1-|p|^2)^{1/2}w)$ has $\HH^b$-component of norm $(1-|p|^2)^{1/2}\le4/5$, so $y\in U_{4/5}$, where
$\beta=1$.

We now define the clutching maps. Write $\RR^4=\RR\oplus\RR^3$ with coordinates $a=(s,p)$, and
write $D^4_R$ and $S^3_R=\partial D^4_R$ for the disk and the sphere of radius $R$. For
$a=(s,p)\in S^3_R$, define
\[
  f_a\colon S^{4b-1}\longrightarrow S^{4b-1}
\]
by
\begin{equation}\label{eq:ek-fa}
  f_a(w)=\begin{cases} g(p,w)\,w & s\ge0,\\ w & s\le0.\end{cases}
\end{equation}
The two formulas agree near $s=0$, where $|p|$ is close to $R>3/5$. For $s\ge0$ and $q\in G$,
equivariance of $g$ gives
\[
  f_a(qw)=g(p,qw)\,qw=q\,g(p,w)\,q^{-1}qw=q\,f_a(w),
\]
and the same holds trivially for $s\le0$. The point $f_a(w)=g(p,w)\,w$ lies in the line $\HH w$. Hence, $f_a$
maps each fibre $S(\HH w)=\{qw:q\in G\}$ of the Hopf bundle $S^{4b-1}\to\HP^{b-1}$ to itself and commutes
with left multiplication. Thus, $f_a$ is an automorphism of this bundle covering the identity. The map
$f_a$ is a diffeomorphism. For $s\ge0$, its inverse is
\[
  w\longmapsto g(p,w)^{-1}w,
\]
because $g(p,g(p,w)w)=g(p,w)$ by equivariance. Since $a\mapsto f_a$ joins every $f_a$ to
the identity, each $f_a$ preserves orientation.

Let $M_f$ be the closed manifold obtained from the oriented manifolds
\[
  D^4_R\times S^{4b-1}\qquad\text{and}\qquad S^3_R\times D^{4b}
\]
(oriented with the product orientations, $S^{4b-1}=\partial D^{4b}$) by identifying
$(a,w)\in S^3_R\times S^{4b-1}=\partial(D^4_R\times S^{4b-1})$ with
\[
  \Phi(a,w)=(a,f_a(w))\in S^3_R\times S^{4b-1}=\partial(S^3_R\times D^{4b}).
\]
The two boundary orientations of $S^3_R\times S^{4b-1}$ are opposite and $\Phi$ preserves the
product orientation. Hence $M_f$ is oriented. The group $G$ acts on $M_f$ by left
multiplication on the factors $S^{4b-1}$ and $D^{4b}$. This action is well defined on $M_f$ because the
identification $\Phi$ commutes with it. For $q\in G$, we have $\Phi(a,qw)=(a,f_a(qw))=(a,q\,f_a(w))$.

\begin{lemma}\label{lem:ek-twisted}
There is a $G$-equivariant, orientation-preserving diffeomorphism
\[
  \Sigma_{\rho,\beta}\cong M_f .
\]
\end{lemma}

\begin{proof}
We first consider the standard sphere, the case $f_a=\mathrm{id}$. Write
\[
  S^n=S(\RR\oplus V_b)=S(\RR^4\oplus\HH^b),
\]
with points $(a,w)$. The maps
\[
  v\longmapsto\Bigl(\pm\cos\frac{\pi|v|}{2},\ \frac{\sin(\pi|v|/2)}{|v|}\,v\Bigr)
\]
are $G$-equivariant diffeomorphisms from $D(V_b)$ onto the upper and lower hemispheres
$\{s\ge0\}$ and $\{s\le0\}$, equal to the identity on the boundary. At the poles $(\pm1,0)$ the
outward normals are $\pm\partial_s$, so the first preserves orientation and the second reverses
it. Hence
\[
  S^n=D(V_b)\cup_{\mathrm{id}}\bigl(-D(V_b)\bigr).
\]
Let
\[
  A=\{|w|\ge|a|\},\qquad C=\{|w|\le|a|\}.
\]
The maps
\begin{align*}
  A&\longrightarrow D^4_R\times S^{4b-1}, & (a,w)&\longmapsto(a,\,w/|w|),\\
  C&\longrightarrow S^3_R\times D^{4b}, & (a,w)&\longmapsto(Ra/|a|,\,w/|a|),
\end{align*}
are diffeomorphisms which agree on $A\cap C=\{|a|=|w|=R\}$. The first preserves orientation:
at a point $(0,w)$ the tangent space is $\RR^4\oplus T_wS^{4b-1}$, the outward normal $(0,w)$
lies in $\HH^b$, and moving it past the four vectors of $\RR^4$ costs no sign. As $S^n$ is
connected, the second preserves orientation too. Thus $S^n=M_{\mathrm{id}}$ as oriented
manifolds.

We next introduce the twist. The star sphere is obtained from the two hemispheres by
identifying the equatorial point $y$ of the upper one with the point $\sigma_\beta(y)$ of the
lower one. On $C\cap\{s=0\}$ we have $|w|\le R<4/5$, so $\sigma_\beta=\mathrm{id}$ there. On
$A\cap\{s=0\}$ we use the coordinates $(a,w)=((0,p),w)$ of $A$, in which $w$ is a unit vector.
A point of $S(V_b)$ with these coordinates is $(p,(1-|p|^2)^{1/2}w)$, and there $\beta$ takes
the value $g(p,w)$. So formula~\eqref{eq:ek-sigma} reads
\[
  \sigma_\beta(p,w)=(p,h_p(w)),\qquad h_p(w)=g(p,w)^{-1}w .
\]
Let $A_\pm=A\cap\{\pm s\ge0\}$, and define
\[
  \Psi(a,w)=(a,h_p(w))\quad\text{on }A_+,\qquad \Psi=\mathrm{id}\quad\text{on }A_-,
\]
where $p$ is the $\RR^3$-component of $a$. An equatorial point $(p,w)$ of $A_+$ goes to
$(p,h_p(w))$, which is also the image of its partner $\sigma_\beta(p,w)=(p,h_p(w))$ in $A_-$.
So $\Psi$ is well defined on $A_+\cup_{\sigma_\beta}A_-$.

We now verify that $\Psi$ is smooth across the equator. The smooth structure of a manifold
glued along its boundary does not depend on the collars used to define
it~\cite[VI.5, pp.~100--101]{Kosinski}. For $|p|\ge3/5$ we have $h_p=\mathrm{id}$, and $\Psi$
and the gluing are the identity. Over $|p|<2/3$ we may use the collars
\[
  (t,(p,w))\longmapsto((\pm t,p),w),\qquad t\ge0,
\]
of the equator in $A_\pm$. Since $\sigma_\beta$ preserves $p$, in the resulting chart
$(t,(p,w))$ of $A_+\cup_{\sigma_\beta}A_-$, with $(p,w)$ in the upper copy, $\Psi$ is
\[
  (t,(p,w))\longmapsto((t,p),h_p(w))
\]
on both sides of $t=0$. Thus $\Psi$ is a smooth diffeomorphism from $A_+\cup_{\sigma_\beta}A_-$
onto $D^4_R\times S^{4b-1}$, and it preserves orientation because each $h_p$ does.

Finally, we compare the gluing with $C$. On the face $\{|a|=R,\ s\ge0\}$ of $A_+$, the point
with $\Psi$-coordinates $(a,w)$ is the point with $A$-coordinates
\[
  (a,h_p^{-1}(w))=(a,f_a(w))
\]
of the standard sphere, and that point lies on $C$; on the face $\{|a|=R,\ s\le0\}$ nothing
changes. So the identification of $A$ with $C$ becomes $\Phi$, and $\Psi\cup\mathrm{id}_C$ is
the required diffeomorphism. Both $\Psi$ and the coordinate maps commute with left
multiplication on $w$.
\end{proof}

Under the diffeomorphism given by Lemma~\ref{lem:ek-twisted}, the fixed set of $G$ on $\Sigma_{\rho,\beta}$ is the $3$-sphere
\[
  S^3_R\times\{0\}\subset S^3_R\times D^{4b},
\]
which is the union of the two disks $D(\RR^3)$. The piece $S^3_R\times D^{4b}$ is an invariant
tubular neighbourhood of it, and $G$ acts freely on the complementary piece
\[
  N=D^4_R\times S^{4b-1},\qquad N/G=D^4_R\times\HP^{b-1}.
\]
We call $N$, the complement of an open invariant tubular neighbourhood of the fixed sphere, the
\emph{free part}.

\subsection{The filling}\label{subsec:ek-filling}

In this subsection we construct a compact spin manifold $Z$ with boundary $\Sigma_{\rho,\beta}$
(Proposition~\ref{prop:ek-filling}). We will use $Z$ in our computation of the Eells--Kuiper invariant.
Over the free part $N$, the manifold $Z$ is the mapping cylinder of the orbit map $N\to N/G=D^4_R\times\HP^{b-1}$, and over
the fixed sphere it is closed off by a copy of $D^4\times\HP^b$. We then orient $Z$ so that its boundary is
$\Sigma_{\rho,\beta}$ as an oriented manifold (Proposition~\ref{prop:ek-filling}(iv)). The manifold $Z$ contains the union
$S^4\times\HP^{b-1}$ of $D^4_R\times\HP^{b-1}$ and a second copy of it (the manifold $B$ of Definition~\ref{def:ek-filling}), and we
express the Euler number of the normal bundle of this union, restricted to a $4$-sphere $S^4\times\{\ell\}$, through the slice degree
(Lemma~\ref{lem:ek-euler}). To close off the
mapping cylinder, we first place the disk bundle of $\gamma$ inside $\HP^b$, the space of lines in $\HH^{b+1}$. For
$b=3$ this embedding, and the use of $\HP^b$ to detect the top Pontryagin class, is due to Chen,
Deng and Zhang~\cite[Lemma~5.1 and \S5.2]{CDZ}.
Write $[w:t]$ for the line $\HH(w,t)$, $w\in\HH^b$, $t\in\HH$. Recall from Section~\ref{subsec:ek-statement} that $\gamma=\gamma_{b-1}$ is the tautological bundle over $\HP^{b-1}$, whose fibre over a line $\ell$ is $\ell$, and that $x=c_2(\gamma)$. For every $r\ge1$, let $\gamma_r$ be the tautological bundle over $\HP^r$; in particular, $\gamma_b$ is the tautological bundle over $\HP^b$.

\begin{lemma}\label{lem:ek-HPb}
The map $\iota\colon D(\gamma)\to\HP^b$, $\iota(hw)=[w:\bar h]$ for $w\in S^{4b-1}$ and
$|h|\le1$, is a well-defined smooth embedding whose image is the set $\{[w:t]:|t|\le|w|\}$. It maps the zero
section of $D(\gamma)$ to the standard $\HP^{b-1}=\{[w:0]\}$ and a unit vector $v$ to $[v:1]$. The
complement of the interior of its image is the closed unit ball $D^{4b}$ of the affine chart
$z\mapsto[z:1]$. Hence $\HP^b=\iota(D(\gamma))\cup D^{4b}$, glued along $S^{4b-1}$ by $v\mapsto v$,
and $c_2(\gamma_b)$ restricts to $x=c_2(\gamma_{b-1})$ on $\HP^{b-1}$.
\end{lemma}

\begin{proof}
Replacing $w$ by $qw$ replaces $h$ by $hq^{-1}$ and $[w:\bar h]$ by $[qw:q\bar h]$, the same
line. Hence $\iota$ is well defined, and in a local section $w$ of the Hopf bundle it is smooth
in $h$. A line $[w:t]$ with $|t|\le|w|$ has $w\ne0$; scaling, we may take $|w|=1$, and then
$\bar t w$ is its unique preimage, which is unchanged when $(w,t)$ is replaced by $(qw,qt)$.
Hence $\iota$ is a smooth bijection onto its image with smooth inverse. If $|h|=1$, then
$[w:\bar h]=[hw:h\bar h]=[hw:1]$. A line $[w:t]$ with $|t|\ge|w|$ is $[z:1]$ with
$z=t^{-1}w$ and $|z|\le1$. The tautological bundle of $\HP^b$ restricts to that of
$\HP^{b-1}$, which gives the last claim.
\end{proof}

Let $\tilde D^4$ be a second copy of $D^4_R$, and let $S^4=D^4_R\cup_{S^3_R}\tilde D^4$, with the orientation $o$ induced on
$D^4_R\subset\RR^4$ by the standard orientation on $\RR^4$. For $a=(s,p)\in S^3_R$, let
$\tilde f_a$ be the bundle automorphism of $\gamma$ given on each line by
$\tilde f_a(hw)=h\,f_a(w)$, $w\in S^{4b-1}$, $h\in\HH$. It is left $\HH$-linear on each fibre and
covers the identity of $\HP^{b-1}$.

\begin{definition}[The filling]\label{def:ek-filling}
Let $Z$ be the smooth manifold obtained from $D^4_R\times D(\gamma)$ and $\tilde D^4\times\HP^b$ by
identifying $(a,v)\in S^3_R\times D(\gamma)$ with $(a,\iota(\tilde f_a(v)))\in\partial \tilde D^4\times\HP^b$, and
smoothing the corners. Let $B=S^4\times\HP^{b-1}$ be the union of the zero section
$D^4_R\times\HP^{b-1}$ of $D^4_R\times D(\gamma)$ and of $\tilde D^4\times\HP^{b-1}\subset\tilde D^4\times\HP^b$, where $\HP^{b-1}=\{[w:0]\}$ (Lemma~\ref{lem:ek-HPb}), let $\xi_\beta\to B$ be the quaternionic line
bundle obtained from $D^4_R\times\gamma$ and $\tilde D^4\times\gamma$ by the clutching
$(a,v)\sim(a,\tilde f_a(v))$, and let $F=\{a'\}\times\HP^b$ for a point $a'$ in the interior of
$\tilde D^4$. Let $\pi\colon Z\to S^4$ be the projection to the disk factors.
\end{definition}

The identification that defines $Z$ covers the identity of $S^3_R$, so $\pi$ is well defined.
We now prove some basic properties of the filling $Z$.

\begin{proposition}\label{prop:ek-filling}
\begin{enumerate}
\item $Z$ is a compact smooth $4m$-manifold. It contains $B$ in its interior with normal
  bundle $\xi_\beta$, and $Z=D(\xi_\beta)\cup(\tilde D^4\times D^{4b})$ is $D(\xi_\beta)$ with one
  handle of index $4b$ attached along $\tilde D^4\times S^{4b-1}$.
\item $Z$ is $3$-connected; in particular it is spin.
\item $F\subset\pi^{-1}(a')$ is a copy of $\HP^b$ with trivial normal bundle, and
  $F\cap B=\{a'\}\times\HP^{b-1}$ is the standard $\HP^{b-1}\subset\HP^b$.
\item Orient $B$ by $o$ and an orientation $O$ of $\HP^{b-1}$ (Lemma~\ref{lem:ek-orientation}),
  and orient $Z$ by $B$ followed by the canonical orientation of $\xi_\beta$. Then $\partial Z$,
  with the outward normal first, is $G$-equivariantly and orientation-preservingly
  diffeomorphic to $\Sigma_{\rho,\beta}$.
\item \emph{($S^3$-quotient description.)} Let $N\subset\Sigma_{\rho,\beta}$ be the free part
  defined after Lemma~\ref{lem:ek-twisted}. Then $D^4_R\times D(\gamma)$ is the mapping cylinder
  $N\times_G D(\HH)$ of the orbit map $N\to N/G$, where $G$ acts on $N\times D(\HH)$ by
  $q\cdot(y,h)=(\rho(q)y,hq^{-1})$. The zero section of this mapping cylinder is $N/G$, the restriction of $\xi_\beta$ to
  $N/G$ is the associated bundle $N\times_G\HH$, and
  \[
    Z=\bigl(N\times_G D(\HH)\bigr)\cup_{\partial N\times_G D(\HH)}\bigl(\tilde D^4\times\HP^b\bigr).
  \]
\end{enumerate}
\end{proposition}

\begin{proof}
(i) The pieces are compact manifolds of dimension $4+4b=4m$, and they are glued by a
diffeomorphism between codimension-zero submanifolds of their boundaries. The zero
section $\HP^{b-1}$ of $D(\gamma)$ is fixed by $\tilde f_a$, which is linear on each fibre, so $B$ is a closed submanifold of the interior. Its normal
bundle is $D^4_R\times\gamma$ over the first piece and, by Lemma~\ref{lem:ek-HPb}, $\tilde D^4\times\gamma$
over the second, clutched by $\tilde f_a$; this is $\xi_\beta$. By Lemma~\ref{lem:ek-HPb},
$\tilde D^4\times\HP^b=(\tilde D^4\times\iota(D(\gamma)))\cup(\tilde D^4\times D^{4b})$, which gives the handle
decomposition.

(ii) $B=S^4\times\HP^{b-1}$ is $3$-connected, $D(\xi_\beta)$ deformation retracts onto $B$, and
attaching a handle of index $4b\ge8$ does not change $\pi_j$ for $j\le3$. So $Z$ is
$3$-connected, $H^2(Z;\ZZ/2)=0$ and $w_2(Z)=0$.

(iii) A neighbourhood of $F$ is $\tilde D^4\times\HP^b$, which is a product. The second claim is
Lemma~\ref{lem:ek-HPb}.

(iv) $\partial Z$ consists of $D^4_R\times S(\gamma)$ and $\partial \tilde D^4\times D^{4b}$, the
complement of the interior of $\partial \tilde D^4\times\iota(D(\gamma))$ in $\partial \tilde D^4\times\HP^b$. Identify
$S(\gamma)$ with $S^{4b-1}$ by $v\mapsto v$. By Lemma~\ref{lem:ek-HPb}, the point
$(a,v)\in S^3_R\times S(\gamma)$ is glued to $(a,[f_a(v):1])$, the point $(a,f_a(v))$ of
$S^3_R\times\partial D^{4b}$. Hence $\partial Z$ is diffeomorphic to $M_f$, by the map that is the identity on
each of the two pieces. On the open set $D^4_R\times(D(\gamma)\setminus0)$ of $Z$, the map $(a,v)\mapsto(a,v)$ to
$D^4_R\times(D^{4b}\setminus0)$ carries the orientation of $Z$, given by $o$ on $D^4_R$, then $O$ on $\HP^{b-1}$
(Lemma~\ref{lem:ek-orientation}(i)), then the canonical orientation of the fibres of $\gamma$
(Section~\ref{subsec:ek-statement}), to the product orientation of $\RR^4\oplus\HH^b$, by the definition of $O$ in
Lemma~\ref{lem:ek-orientation}(i).
On the boundary piece $D^4_R\times S^{4b-1}$ the outward normal is the radial vector of
$\HH^b$, and moving it past the four vectors of $\RR^4$ costs no sign; so the boundary
orientation there is the product orientation of $D^4_R\times S^{4b-1}$, which is the
orientation of that piece of $M_f$. A diffeomorphism of connected oriented manifolds that
preserves orientation on a nonempty open set preserves it everywhere. Lemma~\ref{lem:ek-twisted}
completes the proof.

(v) The map $N\times D(\HH)\to D^4_R\times D(\gamma)$, $((a,w),h)\mapsto(a,hw)$, is constant
on orbits, since $(a,qw,hq^{-1})\mapsto(a,hw)$, and it induces a diffeomorphism from the
quotient, because the action is free and every vector of $\gamma$ over $\HH w$ has the form
$hw$ for a unique $h$. Its restriction to $h=0$ is the orbit map, and over $N/G$ it
identifies $N\times_G\HH$ with $D^4_R\times\gamma$. Finally
$\partial N\times_G D(\HH)=S^3_R\times D(\gamma)$, which is the glued part.
\end{proof}

\begin{lemma}\label{lem:ek-orientation}
\begin{enumerate}
\item There is a unique orientation $O$ of $\HP^{b-1}$ for which the diffeomorphism
  $D(\gamma)\setminus0\to D^{4b}\setminus0$, $v\mapsto v$, carries the orientation
  ($O$ followed by the canonical fibre orientation) to the orientation of $\HH^b$. In the
  chart $z\mapsto[z:1]$, $z\in\HH^{b-1}$, it is the standard orientation of $\HH^{b-1}$.
\item $\langle x^{b-1},[\HP^{b-1},O]\rangle=(-1)^{b-1}$.
\end{enumerate}
\end{lemma}

\begin{proof}
(i) Since $\HP^{b-1}$ and $D(\gamma)\setminus0$ are connected, exactly one of the two orientations of $\HP^{b-1}$ works.
In the chart, consider the bundle chart $(z,h)\mapsto h(z,1)$ of $\gamma$. Over the point $[z:1]$ it restricts to
$h\mapsto h(z,1)$, which is the map $h\mapsto hw$, with $w=(z,1)\in\ell\setminus0$, that defines the canonical
orientation of the line $\ell=[z:1]$ (Section~\ref{subsec:ek-statement}). Hence the bundle chart preserves the
canonical orientation of each fibre. Its differential is
$(\zeta,\eta)\mapsto(\eta z+h\zeta,\ \eta)$, which is block triangular with diagonal blocks
left multiplication by $h$ on $\HH^{b-1}$ and the identity of $\HH$, of determinant
$|h|^{4(b-1)}>0$. Hence the standard orientation of the chart is $O$.

(ii) For $\ell\in\HP^{b-1}$ let $s_i(\ell)$ be the orthogonal projection of the basis vector
$e_i$ onto $\ell$, for $i=1,\dots,b-1$. These form a section of $(b-1)\gamma$. With the product orientation of its fibres, the Euler class of
$(b-1)\gamma$ is $e(\gamma)^{b-1}=x^{b-1}$, since the Euler class of a Whitney sum is the cup product of the
Euler classes~\cite[Property~9.6, p.~100]{MS}. A common zero is a line orthogonal to $e_1,\dots,e_{b-1}$, so the only zero is
$\ell_0=\HH e_b$, the centre of the chart. With $\langle v,w\rangle=\sum_jv_j\bar w_j$, which
satisfies $\langle hv,w\rangle=h\langle v,w\rangle$, the projection onto $\HH u$ ($|u|=1$) is
$v\mapsto\langle v,u\rangle u$. For $\ell=[z:1]$ and $u=(z,1)/(1+|z|^2)^{1/2}$ this gives
$s_i=\bigl(\bar z_i/(1+|z|^2)\bigr)(z,1)$. Hence in the chart and its bundle chart the section is
$z\mapsto(\bar z_1,\dots,\bar z_{b-1})/(1+|z|^2)$. Its differential at $0$ is quaternionic conjugation on
each factor $\HH$, of determinant $(-1)^{3(b-1)}=(-1)^{b-1}$. Recall that, for an oriented vector bundle $E$
over a closed oriented manifold $M$ with $\operatorname{rank}E=\dim M$, the Euler number $\langle e(E),[M]\rangle$ is
the sum of the indices of the zeros of any section with isolated zeros. At a zero where the section is
transverse to the zero section, the index is the sign of the determinant of its differential, computed in an
oriented chart of $M$ and an oriented local trivialisation of $E$~\cite[\S12 and Theorem~12.5, p.~147]{MS},
\cite[\S6, proof of Lemma~3, p.~36]{MilnorTDV}. By~(i), the chart $z\mapsto[z:1]$ is oriented by $O$, and the
bundle chart by the canonical orientations of the fibres. So the only zero, at $\ell_0$, has index $(-1)^{b-1}$,
and $\langle x^{b-1},[\HP^{b-1},O]\rangle=(-1)^{b-1}$.
\end{proof}

\begin{lemma}\label{lem:ek-euler}
For every line $\ell_0$, $\langle e(\xi_\beta),[S^4\times\{\ell_0\},o]\rangle=-k(\beta)$.
\end{lemma}

\begin{proof}
Let $\ell_0=\HH w_0$ with $|w_0|=1$, and trivialise $\xi_\beta$ over $D^4_R\times\{\ell_0\}$ and
$\tilde D^4\times\{\ell_0\}$ by $h\mapsto hw_0$; these trivialisations preserve the canonical
orientation. Recall the map $g$ of Section~\ref{subsec:ek-normal}. Since $\tilde f_a(hw_0)=h\,g_a\,w_0$ with $g_a=g(p,w_0)$ for $s\ge0$ and $g_a=1$
for $s\le0$, the transition over $S^3_R$ is $h'=h\,g_a$, where $h$ and $h'$ are the
coordinates over $D^4_R$ and $\tilde D^4$. The section $h'=1$ has no zeros over $\tilde D^4$, and over
$S^3_R$ it reads $h=g_a^{-1}$. The Euler number is the obstruction to extending this nowhere-zero
section over $D^4_R$~\cite[\S12 and Theorem~12.5, p.~147]{MS}, that is, the degree of
$a\mapsto g_a^{-1}$ on $S^3_R=\partial D^4_R$; equivalently, it is the sum of the indices of the
zeros of any extension with isolated zeros (see the proof of Lemma~\ref{lem:ek-orientation}(ii)). Inversion on $S^3$ has degree
$-1$, since it fixes $1$ and reverses $i,j,k$; so the Euler number is $-\deg(a\mapsto g_a)$.
The closed upper half of $S^3_R$ is the graph of $p\mapsto(R^2-|p|^2)^{1/2}$ over $D^3_R$, and
its projection to $D^3_R$ preserves orientation, because at $(R,0)$ the outward normal is
$\partial_s$. The map $a\mapsto g_a$ is $1$ on the lower half, and on the upper half it is
$p\mapsto g(p,w_0)=\beta(s_{w_0}(p))$. As $\beta\circ s_{w_0}=1$ on $3/5\le|p|\le1$, the degree
is $k(\beta)$.
\end{proof}

The computation in the proof of Lemma~\ref{lem:ek-euler}, applied to $\gamma_1$ over $\HP^1$ with fibre coordinates $h_0$ and $h_1$
over the two affine charts, where $h_1=h_0z$ on $|z|=1$, gives $\langle x,[\HP^1,O]\rangle=-1$, in agreement with
Lemma~\ref{lem:ek-orientation}(ii) at $b=2$.

\subsection{The cohomology of the closed-up filling}\label{subsec:ek-ring}

In this subsection we compute the rational cohomology ring of the space $\Zh$ obtained from $Z$ by capping off
its boundary with a ball (Proposition~\ref{prop:ek-ring}). The Eells--Kuiper invariant is computed from the
Pontryagin numbers of $Z$, evaluated on the relative fundamental class $[Z,\partial Z]$
(Section~\ref{sec:prelim}). As shown below, since $\partial Z$ is a homotopy sphere, these relative numbers can be computed in the
cohomology ring of $\Zh$. That ring is generated by two classes $X$ and $Y$ of degree $4$ with two relations, and we
find the values of its top-degree monomials on the fundamental class.

Since $\partial Z$ is a homotopy $(4m-1)$-sphere, the map
$j\colon H^i(Z,\partial Z;\QQ)\to H^i(Z;\QQ)$ is an isomorphism for $0<i<4m-1$. Let $\Zh=Z\cup C(\partial Z)$. Since $\partial Z$ is a homotopy sphere of
dimension $4m-1\ge11$, it is homeomorphic to $S^{4m-1}$~\cite[Theorem~A, p.~391]{Smale}, so the cone $C(\partial Z)$
is homeomorphic to $D^{4m}$, and $\Zh$ is a closed topological manifold. It need not be smooth, since $\partial Z$
need not be diffeomorphic to the standard sphere, and we use only its rational cohomology. Then
$H^i(\Zh;\QQ)=H^i(Z,\partial Z;\QQ)$ for $i>0$, with the relative cup product, and
$[\Zh]\in H_{4m}(\Zh;\QQ)$ denotes the image of $[Z,\partial Z]$. For $0<i<4m$, we identify
$H^i(\Zh;\QQ)$ with $H^i(Z;\QQ)$ by $j$; a class of $\Zh$ is then restricted to $B$ or $F$ through
$Z$.

Let $y_o\in H^4(S^4;\ZZ)$ be the generator with $\langle y_o,[S^4,o]\rangle=1$. Let $X,Y\in
H^4(\Zh;\QQ)$ be the classes with $X|_B=x$ (pulled back from $\HP^{b-1}$) and $Y=\pi^*y_o$. Let $\ell_0\in\HP^{b-1}$, and let
\[
  \kappa=\langle e(\xi_\beta),[S^4\times\{\ell_0\},o]\rangle .
\]
By Lemma~\ref{lem:ek-euler}, $\kappa=-k(\beta)$; in particular, $\kappa$ does not depend on $\ell_0$.

\begin{proposition}\label{prop:ek-ring}
\begin{enumerate}
\item $H^i(\Zh;\QQ)=0$ unless $i\equiv0\pmod 4$ and $0\le i\le4m$; for $1\le j\le b$ the
  classes $X^j$ and $X^{j-1}Y$ form a basis of $H^{4j}(\Zh;\QQ)$, and the restriction
  $H^{4j}(\Zh;\QQ)\to H^{4j}(B;\QQ)\oplus H^{4j}(F;\QQ)$ is injective.
\item $X|_F=c_2(\gamma_b)$ and $Y|_F=0$, under $F\cong\HP^b$.
\item $Y^2=0$, $X^{b+1}+\kappa X^bY=0$, and $\langle X^bY,[\Zh]\rangle=(-1)^{b-1}$.
\end{enumerate}
Consequently
\[
  H^*(\Zh;\QQ)=\QQ[X,Y]\big/\bigl(Y^2,\ X^{b+1}+\kappa X^bY\bigr),\qquad
  \langle X^{b+1},[\Zh]\rangle=(-1)^b\kappa,
\]
with $X$ and $Y$ of degree $4$.
\end{proposition}

\begin{proof}
(i) By Proposition~\ref{prop:ek-filling}(i), $Z\simeq B\cup e^{4b}$, so $H^i(Z)\to H^i(B)$ is an
isomorphism for $i\ne4b$, and $H^*(B;\QQ)=\QQ[x,y_o]/(x^b,y_o^2)$ by the K\"unneth theorem.
This gives the vanishing statement for $0<i<4m-1$; by Lefschetz duality
$H^{4m-1}(Z,\partial Z;\QQ)\cong H_1(Z;\QQ)=0$ and $H^{4m}(Z,\partial Z;\QQ)\cong\QQ$. In
degree $4b$, excision gives $H^{4b}(Z,D(\xi_\beta))\cong H^{4b}(D^{4b}\times \tilde D^4,S^{4b-1}\times
\tilde D^4)\cong\ZZ$, concentrated in that degree, and since $H^{4b-1}(B)=0$,
\[
  0\to H^{4b}(Z,D(\xi_\beta);\QQ)\to H^{4b}(Z;\QQ)\to H^{4b}(B;\QQ)\to0
\]
is exact. The first group maps isomorphically to $H^{4b}(F,F\cap D(\xi_\beta);\QQ)$, and then to
$H^{4b}(F;\QQ)$, by excision and because $H^{4b-1}(\HP^{b-1})=H^{4b}(\HP^{b-1})=0$. Hence the
restriction $H^{4b}(Z;\QQ)\to H^{4b}(B;\QQ)\oplus H^{4b}(F;\QQ)$ is injective. For $j<b$ the classes
$X^j$ and $X^{j-1}Y$ restrict to the basis $x^j$, $x^{j-1}y_o$ of $H^{4j}(B;\QQ)$. For $j=b$
they restrict to $(0,c_2(\gamma_b)^b)$ and $(x^{b-1}y_o,0)$, by~(ii), which are independent. In both cases a
basis of $H^{4j}(\Zh;\QQ)$ restricts to linearly independent classes, so the restriction to $B$ and $F$ is
injective.

(ii) $X|_F$ restricts to $x$ on $\HP^{b-1}\subset F$, and $H^4(\HP^b)\to H^4(\HP^{b-1})$ is an
isomorphism; Lemma~\ref{lem:ek-HPb} gives $X|_F=c_2(\gamma_b)$. Since $\pi$ maps $F$ to a point,
$Y|_F=0$.

(iii) We have $Y^2=\pi^*(y_o^2)=0$. Let $E\in H^4(Z,\partial Z)$ be the Thom class of the normal bundle of
$B$, extended by zero. With the orientation of $Z$ given by $B$ followed by $\xi_\beta$, it
satisfies $E|_B=e(\xi_\beta)$ and $\langle a\smile E,[Z,\partial Z]\rangle=\langle a|_B,[B]\rangle$ for
$a\in H^{4m-4}(Z)$~\cite[\S11]{MS}; as in~\cite[(5.6), p.~12]{CDZ} for $b=3$. Restricting $e(\xi_\beta)$ to the
submanifolds $\{a\}\times\HP^{b-1}$, for a point $a\in S^4$, and $S^4\times\{\ell_0\}$ of $B$ gives $e(\xi_\beta)=x+\kappa y_o$, and $H^4(\Zh)\to H^4(B)$ is
an isomorphism; so $E=X+\kappa Y$. Now $x^b=0$ and, by Lemma~\ref{lem:ek-orientation}(ii) and the
product orientation of $B$, we have $\langle x^{b-1}y_o,[B]\rangle=(-1)^{b-1}$. Therefore,
\begin{align*}
  \langle X^bY\rangle&=\langle X^{b-1}Y(X+\kappa Y)\rangle=\langle x^{b-1}y_o,[B]\rangle=(-1)^{b-1},\\
  \langle X^{b+1}+\kappa X^bY\rangle&=\langle X^bE\rangle=\langle x^b,[B]\rangle=0 .
\end{align*}
Since $H^{4m}(\Zh;\QQ)$ is one-dimensional and $\langle X^bY\rangle\ne0$, the second identity is the
relation $X^{b+1}+\kappa X^bY=0$.

The ring $\QQ[X,Y]/(Y^2,X^{b+1}+\kappa X^bY)$ has the basis $X^j,X^{j-1}Y$ ($1\le j\le b$), together
with $1$ and $X^bY$; in degrees above $4m$ it vanishes, since $X^{b+1}Y=-\kappa X^bY^2=0$ and
$X^{b+2}=-\kappa X^{b+1}Y=0$. By~(i) and~(iii) it maps onto $H^*(\Zh;\QQ)$, and both have the same
dimension in each degree.
\end{proof}

\subsection{Pontryagin classes and signature}\label{subsec:ek-pont}

In this subsection we compute the Pontryagin classes of $Z$ (Proposition~\ref{prop:ek-pont}) and its signature
(Lemma~\ref{lem:ek-signature}). These are the inputs of the Eells--Kuiper invariant
(Definition~\ref{def:ek-mu}), which we evaluate in Section~\ref{subsec:ek-evaluation}. We first compute the Pontryagin
classes of a quaternionic line bundle and of $\HP^r$ (Lemma~\ref{lem:ek-pont-HP}). The Pontryagin classes of $Z$
are then determined by their restrictions to $B$ and to $F$.

\begin{lemma}\label{lem:ek-pont-HP}
For a quaternionic line bundle $\eta$, we have $p(\eta)=(1-c_2(\eta))^2$, and
$p(T\HP^r)=(1-x)^{2r+2}(1-4x)^{-1}$ with $x=c_2(\gamma_r)$.
\end{lemma}

\begin{proof}
For a complex bundle $\eta$ of rank two with $c_1=0$ we have $p_1=c_1^2-2c_2=-2c_2$ and
$p_2=e^2=c_2^2$; see~\cite[(3.1)]{DHZ}. For the second formula, $T\HP^r\oplus\mathrm{Hom}_\HH(\gamma_r,\gamma_r)
\cong\mathrm{Hom}_\HH(\gamma_r,\HH^{r+1})=(r+1)\,\mathrm{Hom}_\HH(\gamma_r,\HH)$. The real bundle
$\mathrm{Hom}_\HH(\gamma_r,\HH)$ is isomorphic to $\gamma_r$ by $v\mapsto\langle\,\cdot\,,v\rangle$, so its total
Pontryagin class is $(1-x)^2$. The bundle $\mathrm{Hom}_\HH(\gamma_r,\gamma_r)$ is a trivial line plus the
adjoint bundle, whose complexification is $\mathrm{Sym}^2\gamma_r$; with Chern roots $\pm t$ of
$\gamma_r$, those of $\mathrm{Sym}^2\gamma_r$ are $2t,0,-2t$, so $p_1=-c_2(\mathrm{Sym}^2\gamma_r)=-4x$.
The resulting formula for $p(T\HP^r)$ is~\cite[(5.4), p.~10]{CDZ}, derived there from~\cite{CrowleyGoette}. Chen, Deng and Zhang use right lines; the argument above does not depend on whether one chooses left or right lines.
\end{proof}

\begin{proposition}\label{prop:ek-pont}
For $1\le j\le b$, the class $p_j(TZ)\in H^{4j}(\Zh;\QQ)$ is the degree-$4j$ component of
\[
  P=(1-X)^{2b}\,(1-X-\kappa Y)^2\,(1-4X)^{-1}.
\]
\end{proposition}

\begin{proof}
For $1\le j\le b$ we have $0<4j<4m-1$, so $p_j(TZ)\in H^{4j}(Z;\QQ)$ is a class of $\Zh$
(Section~\ref{subsec:ek-ring}). By Proposition~\ref{prop:ek-ring}(i), the restriction
$H^{4j}(\Zh;\QQ)\to H^{4j}(B;\QQ)\oplus H^{4j}(F;\QQ)$ is injective. Hence it is enough to show that $p(TZ)$
and $P$ have the same restrictions to $B$ and to $F$.

We first consider $B=S^4\times\HP^{b-1}$. Its normal bundle in $Z$ is $\xi_\beta$
(Proposition~\ref{prop:ek-filling}(i)), so $TZ|_B=TB\oplus\xi_\beta$, and $TB$ is the Whitney sum of the
pullbacks of $TS^4$ and $T\HP^{b-1}$ under the two projections. The bundle $TS^4$ is stably trivial. The normal
bundle of $S^4$ in $\RR^5$ is trivial, so $TS^4\oplus\varepsilon^1\cong\varepsilon^5$. Hence
$p(TS^4)=1$~\cite[Lemma~15.2, p.~174, and Example, p.~175]{MS}. In rational cohomology, the total Pontryagin
class of a Whitney sum is the product of the total Pontryagin classes of the summands~\cite[Theorem~15.3,
p.~175]{MS}. By Lemma~\ref{lem:ek-pont-HP} with $r=b-1$, we have $p(T\HP^{b-1})=(1-x)^{2b}(1-4x)^{-1}$. The bundle
$\xi_\beta$ is a quaternionic line bundle, glued from copies of $\gamma$ by the left $\HH$-linear maps
$\tilde f_a$, and $c_2(\xi_\beta)=e(\xi_\beta)=x+\kappa y_o$ (proof of Proposition~\ref{prop:ek-ring}(iii)).
Hence $p(\xi_\beta)=(1-x-\kappa y_o)^2$, by Lemma~\ref{lem:ek-pont-HP}. Therefore,
\[
  p(TZ)|_B=(1-x)^{2b}\,(1-4x)^{-1}\,(1-x-\kappa y_o)^2=P|_B,
\]
since $X|_B=x$ and $Y|_B=y_o$.

We next consider $F$. By Proposition~\ref{prop:ek-filling}(iii), the normal bundle of $F$ in $Z$ is trivial, so
$TZ|_F\cong T\HP^b\oplus\varepsilon^4$. By~\cite[Lemma~15.2, p.~174]{MS} and Lemma~\ref{lem:ek-pont-HP} with
$r=b$,
\[
  p(TZ)|_F=p(T\HP^b)=(1-x_F)^{2b+2}(1-4x_F)^{-1},\qquad x_F=c_2(\gamma_b).
\]
By Proposition~\ref{prop:ek-ring}(ii), $X|_F=x_F$ and $Y|_F=0$. Hence
\[
  P|_F=(1-x_F)^{2b}(1-x_F)^2(1-4x_F)^{-1}=(1-x_F)^{2b+2}(1-4x_F)^{-1}=p(TZ)|_F.
\]
\end{proof}

\begin{remark}\label{rem:ek-pont-CDZ}
For $b=3$, the classes $p_1$, $p_2$ and $p_3$ of Proposition~\ref{prop:ek-pont} are those
of~\cite[Proposition~5.2, (5.5), p.~11]{CDZ}, with the integer $k$ of~\cite{CDZ} in place of $\kappa$. The
restriction to $F$ is used only in degree $4b$, where it determines the coefficient of $X^b$ in $p_b$; the class
$X^b$ vanishes on $B$. That coefficient enters the Eells--Kuiper invariant only through
$\langle p_1p_b,[\Zh]\rangle$, with the factor $\langle p_1X^b,[\Zh]\rangle=(-1)^b(4-2b)\kappa$
(Proposition~\ref{prop:ek-ring}(iii)), which vanishes for $b=2$.
\end{remark}

We now show that the signature of $Z$ vanishes, so that only Pontryagin numbers contribute to
$\mu(\Sigma_{\rho,\beta})$. Recall from Section~\ref{subsec:prelim-ek} that $\sgn$ denotes the signature.

\begin{lemma}\label{lem:ek-signature}
$\sgn(Z)=0$.
\end{lemma}

\begin{proof}
By definition, $\sgn(Z)$ is the signature of the symmetric bilinear form
$(\alpha,\alpha')\mapsto\langle\alpha\smile\alpha',[Z,\partial Z]\rangle$ on $H^{2m}(Z,\partial Z;\RR)$. Under the
identification $H^{2m}(\Zh;\QQ)=H^{2m}(Z,\partial Z;\QQ)$ of Section~\ref{subsec:ek-ring}, which respects cup
products and takes $[Z,\partial Z]$ to $[\Zh]$, this is the form
$(\alpha,\alpha')\mapsto\langle\alpha\smile\alpha',[\Zh]\rangle$ on $H^{2m}(\Zh;\RR)$.

If $b$ is even, then $m=b+1$ is odd, so $2m\equiv2\pmod4$, and $H^{2m}(\Zh;\QQ)=0$ by
Proposition~\ref{prop:ek-ring}(i). Hence the form is zero and $\sgn(Z)=0$.

If $b$ is odd, let $j=(b+1)/2$. Then $2m=4j$ and $1\le j\le b$, so $X^j$ and $X^{j-1}Y$ form a basis of
$H^{2m}(\Zh;\QQ)$ by Proposition~\ref{prop:ek-ring}(i). Since $b-1$ is even, Proposition~\ref{prop:ek-ring}(iii)
gives
\[
  \langle X^j\smile X^j,[\Zh]\rangle=\langle X^{b+1},[\Zh]\rangle=-\kappa,\qquad
  \langle X^j\smile X^{j-1}Y,[\Zh]\rangle=\langle X^bY,[\Zh]\rangle=1,
\]
and $X^{j-1}Y\smile X^{j-1}Y=0$, since $Y^2=0$. Hence the Gram matrix of the form in this basis is
$\left(\begin{smallmatrix}-\kappa&1\\1&0\end{smallmatrix}\right)$, of determinant $-1$. The determinant is the
product of the two eigenvalues, so one eigenvalue is positive and the other is negative. Hence
$\sgn(Z)=1-1=0$.
\end{proof}

\subsection{The Eells--Kuiper invariant}\label{subsec:ek-mu}

In this subsection we recall the Eells--Kuiper invariant $\mu(\Sigma)$ of an oriented homotopy $(4m-1)$-sphere
$\Sigma$ (Definition~\ref{def:ek-mu}), in the form in which we evaluate it for the star spheres in
Section~\ref{subsec:ek-evaluation}. It combines the Pontryagin numbers and the signature of a compact spin
manifold $W$ bounded by $\Sigma$, and the combination is chosen so that $\mu(\Sigma)$ does not depend on $W$. We
first identify the polynomial in the Pontryagin classes that enters (Lemma~\ref{lem:ek-normalisation}).

Recall the polynomials $L_m$ and $\Ahat_m$ from Section~\ref{subsec:ek-statement}.

\begin{lemma}\label{lem:ek-normalisation}
The ratio of the coefficients of $p_m$ in $\Ahat_m$ and in $L_m$ is $t_m$. Hence
$N_m=\Ahat_m-t_mL_m$ is a polynomial in $p_1,\dots,p_{m-1}$, each of whose monomials has at
least two factors. For $m=2,3,4$,
\[
  \tfrac{1}{a_2}N_2=\tfrac{p_1^2}{896},\quad
  \tfrac{1}{a_3}N_3=\tfrac{-3p_1^3+4p_1p_2}{190464},\quad
  \tfrac{1}{a_4}N_4=\tfrac{153p_1^4-360p_1^2p_2+80p_2^2+192p_1p_3}{187269120},
\]
and $t_2/a_2=-4/896$, $t_3/a_3=-24/190464$, $t_4/a_4=-1/65024$.
\end{lemma}

\begin{proof}
Write the total Pontryagin class formally as $p=\prod_i(1+z_i)$, so that $p_j$ is the $j$-th elementary
symmetric function of the $z_i$, and let $s_j=\sum_iz_i^j$ be the power sums. If $\log Q(z)=\sum_{j\ge1}c_jz^j$,
then the multiplicative sequence $K$ of $Q$ satisfies
\[
  K(p)=\prod_iQ(z_i)=\exp\Bigl(\sum_{j\ge1}c_js_j\Bigr).
\]
By Newton's formula~\cite[Problem~16-A, p.~195]{MS}, each $s_j$ is a polynomial in $p_1,\dots,p_j$, and
\[
  s_m=(-1)^{m-1}m\,p_m+(\text{a polynomial in }p_1,\dots,p_{m-1}).
\]
In the component of degree $4m$ of $\exp(\sum_jc_js_j)$, the only term that contains $p_m$ is $c_ms_m$; every
other term is a product of at least two power sums $s_j$ with $j<m$. Hence the coefficient of $p_m$ in $K_m$ is
$(-1)^{m-1}m\,c_m$ (compare~\cite[Problem~19-B, p.~230]{MS}).

We now compute $c_m$ for $Q_L$ and $Q_{\Ahat}$. Let $B_n$ be the Bernoulli numbers, defined by
$t/(e^t-1)=\sum_{n\ge0}B_nt^n/n!$, so that $(-1)^{n-1}B_{2n}=|B_{2n}|$ for $n\ge1$. Then
\[
  \frac{x}{\tanh x}=\sum_{n\ge0}\frac{2^{2n}B_{2n}}{(2n)!}\,x^{2n}
\]
\cite[Appendix~B, p.~281]{MS}, where Milnor and Stasheff write $B_n$ for $|B_{2n}|$. Since
$\frac{d}{dx}\log(\sinh x/x)=1/\tanh x-1/x$ and $\frac{d}{dx}\log\cosh x=\tanh x=2/\tanh 2x-1/\tanh
x$~\cite[p.~282]{MS}, integration gives
\[
  \log\frac{\sinh x}{x}=\sum_{n\ge1}\frac{2^{2n}B_{2n}}{2n\,(2n)!}\,x^{2n},\qquad
  \log\cosh x=\sum_{n\ge1}\frac{2^{2n}(2^{2n}-1)B_{2n}}{2n\,(2n)!}\,x^{2n}.
\]
Since $Q_{\Ahat}(z)=x/\sinh x$ with $x=\tfrac12\sqrt z$, and $Q_L(z)=x\cosh x/\sinh x$ with $x=\sqrt z$, we obtain
\[
  c_m(Q_{\Ahat})=-\frac{B_{2m}}{2m\,(2m)!},\qquad
  c_m(Q_L)=\frac{2^{2m+1}(2^{2m-1}-1)B_{2m}}{2m\,(2m)!}.
\]
Multiplying by $(-1)^{m-1}m$, we find that the coefficient of $p_m$ is $-|B_{2m}|/(2\,(2m)!)$ in $\Ahat_m$ and
$2^{2m}(2^{2m-1}-1)|B_{2m}|/(2m)!$ in $L_m$. The latter is the coefficient given
in~\cite[p.~965]{MilnorDiffSpheres} and~\cite[Problem~19-C, p.~230]{MS}. The ratio of the two coefficients is
$-1/\bigl(2^{2m+1}(2^{2m-1}-1)\bigr)=t_m$, the number of~\cite[(9), p.~99]{EK}. Hence $p_m$ does not
occur in $N_m=\Ahat_m-t_mL_m$. The only monomial of weight $m$ with one factor is $p_m$, so every monomial of
$N_m$ has at least two factors. The three displayed
polynomials are a direct expansion; they are~\cite[(3.7), p.~7]{DHZ} for $m=3$,
\cite[(1.3), p.~3, and (6.3), p.~12]{CDZ} for $m=4$, and the formula of~\cite[(11), p.~101]{EK}
for $m=2$.
\end{proof}

\begin{definition}\label{def:ek-mu}
Let $\Sigma$ be an oriented homotopy $(4m-1)$-sphere, $m\ge2$, and let $W$ be a compact spin $4m$-manifold whose
boundary $\partial W$, with the boundary orientation, is $\Sigma$ as an oriented manifold. Here a basis of
$T_y\Sigma$ is positively oriented if an outward-pointing vector followed by this basis is a positively oriented
basis of $T_yW$. Since $\Sigma$ is a homotopy sphere, the map $j\colon H^{4i}(W,\Sigma;\QQ)\to H^{4i}(W;\QQ)$ is
an isomorphism for $0<i<m$. Hence each $p_i(TW)$, $0<i<m$, is the image under $j$ of a unique class
$\bar p_i\in H^{4i}(W,\Sigma;\QQ)$~\cite[Conditions~$\mu$(a), p.~96]{EK}. For a monomial $p_{i_1}\cdots p_{i_t}$ of
weight $i_1+\dots+i_t=m$ with $t\ge2$, so that every $i_s<m$, let
\[
  \langle p_{i_1}\cdots p_{i_t},[W,\Sigma]\rangle=\langle\bar p_{i_1}\smile\cdots\smile\bar p_{i_t},[W,\Sigma]\rangle
\]
\cite[(1), p.~94]{EK}, and extend this linearly. By Lemma~\ref{lem:ek-normalisation}, the polynomial
$N_m=\Ahat_m-t_mL_m$ is a linear combination of such monomials, so this defines
$\langle N_m(p(TW)),[W,\Sigma]\rangle$. The \emph{Eells--Kuiper invariant} of $\Sigma$ is
\[
  \mu(\Sigma)=\frac{1}{a_m}\Bigl(\langle N_m(p(TW)),[W,\Sigma]\rangle+t_m\,\sgn(W)\Bigr)\in\QQ/\ZZ .
\]
\end{definition}

This is the invariant of Eells and Kuiper~\cite[(6) and~(7), p.~97]{EK}, written in our notation. We write
$\sgn(W)$ for their $\tau[W]$ (Section~\ref{subsec:prelim-ek}), since $\sigma$ denotes the gluing map. The numbers $t_m$~\cite[(9), p.~99]{EK} and $a_m=4/(3+(-1)^m)$~\cite[p.~97]{EK} are theirs.

The value of $\mu(\Sigma)$ does not depend on $W$~\cite[Theorem, p.~97]{EK}. We
recall the argument~\cite[pp.~97--98]{EK}. Let $W'$ be a second compact spin manifold with oriented boundary
$\Sigma$, and let $M=W\cup_\Sigma(-W')$. Since $H^1(\Sigma;\ZZ/2)=0$, the closed manifold $M$ is
spin~\cite[p.~98]{EK}. By the additivity of relative Pontryagin numbers and of
signatures~\cite[(4), p.~95, and p.~98]{EK},
\begin{gather*}
  \langle N_m(p(TM)),[M]\rangle=\langle N_m(p(TW)),[W,\Sigma]\rangle-\langle N_m(p(TW')),[W',\Sigma]\rangle,\\
  \sgn(M)=\sgn(W)-\sgn(W').
\end{gather*}
By Hirzebruch's signature theorem~\cite[Theorem~19.4, p.~224]{MS},
$\sgn(M)=L_m[M]$, and hence
\[
  \langle N_m(p(TM)),[M]\rangle+t_m\sgn(M)=\Ahat[M].
\]
By a theorem of Atiyah and Hirzebruch, $\Ahat[M]$ is an integer for every closed spin manifold $M$, and
an even integer if $\dim M\equiv4\pmod8$~\cite[Corollary~2, p.~279]{AtiyahHirzebruch}. Since $4m\equiv4\pmod8$
if and only if $m$ is odd, $\Ahat[M]\in a_m\ZZ$. Hence the values of $\mu(\Sigma)$ computed from $W$ and from $W'$
differ by the integer $\Ahat[M]/a_m$.

The value of $\mu$ on the boundary of the parallelisable plumbing of signature $8$ is
$8t_m/a_m=-1/(a_m2^{2m-2}(2^{2m-1}-1))$, which is $-1/992$ for $m=3$ and $-1/8128$ for
$m=4$, as in~\cite[p.~7]{DHZ} and~\cite[(1.3)]{CDZ}.

\subsection{Evaluation}\label{subsec:ek-evaluation}

In this subsection we evaluate the Pontryagin numbers of $Z$ that enter the Eells--Kuiper invariant, and we prove
Theorem~\ref{thm:ek}. By Proposition~\ref{prop:ek-pont}, the Pontryagin classes of $Z$ are components of the
class $P$. The class $P$ is a product of factors $1+z$ with $z$ of degree $4$, together with the inverse of one such
factor. A multiplicative sequence turns such a product into a product of values of its characteristic power series,
and we read off the result in the ring of Proposition~\ref{prop:ek-ring} (Proposition~\ref{prop:ek-evaluation}).

We use the multiplicative sequences of Section~\ref{subsec:ek-statement} in the algebra $A$ with
$A^n=H^{4n}(\Zh;\QQ)$. By
Proposition~\ref{prop:ek-ring}, $H^*(\Zh;\QQ)$ is concentrated in degrees divisible by $4$, so $A$ is commutative,
and $A^n=0$ for $n>m$. A class of degree $4$ in cohomology has degree $1$ in $A$.

\begin{proposition}\label{prop:ek-evaluation}
Let $K$ be the multiplicative sequence of a power series $Q$ with $Q(0)=1$, and let $P$ be
the class of Proposition~\ref{prop:ek-pont}. We take the components of $P$ of degree above $4b$ to be the
formal ones. For $j\ge1$, let $P_j\in H^{4j}(\Zh;\QQ)$ be the component of degree $4j$ of $P$, so that
$P_j=p_j(TZ)$ for $j\le b$ by Proposition~\ref{prop:ek-pont}, while for $j>b$ the class $P_j$ is defined by this
formula alone. Then, with $K_m(P)=K_m(P_1,\dots,P_m)$,
\[
  \langle K_m(P),[\Zh]\rangle=-\kappa\,F_Q(b),
\]
with $F_Q(b)$ as in~\eqref{eq:ek-FQ}.
\end{proposition}

\begin{proof}
In $A^\Pi$,
\[
  P=(1-X)^{2b}\,(1-X-\kappa Y)^2\,(1-4X)^{-1}
\]
is a product of $2b+2$ elements $1+z$ with $z\in A^1$, namely $z=-X$ ($2b$ times) and $z=-X-\kappa Y$ (twice),
and of the inverse of $1+z$ with $z=-4X$. We use two properties of $K$.
\begin{enumerate}
\item[(a)] For $z\in A^1$, we have $K(1+z)=Q(z)$, since the homomorphism of graded algebras $\QQ[t]\to A$, $t\mapsto z$,
  with $t$ of degree $1$, carries $1+t$ to $1+z$ and commutes with $K$. It commutes with $K$ because each $K_n$ is a
  polynomial. Hence $K(1+z)$ is the image of $K(1+t)=Q(t)$, which is $Q(z)$.
\item[(b)] $K(a^{-1})=K(a)^{-1}$ for $a\in A^\Pi$ with leading term $1$, since $K(a)K(a^{-1})=K(1)=1$.
\end{enumerate}
By multiplicativity, (a) and (b),
\[
  K(P)=Q(-X)^{2b}\,Q(-X-\kappa Y)^2\,Q(-4X)^{-1}.
\]
All these power series are finite sums, since $A^n=0$ for $n>m$. Since $Y^2=0$, we have
$(-X-\kappa Y)^n=(-X)^n-n\kappa(-X)^{n-1}Y$ for every $n\ge1$. Hence $f(-X-\kappa Y)=f(-X)-\kappa Y\,f'(-X)$ for
every power series $f$, and with $f=Q^2$ we obtain $Q(-X-\kappa Y)^2=Q(-X)^2-2\kappa Y\,Q(-X)\,Q'(-X)$. Since
$2b+2=2m$, it follows that
\[
  K(P)=\frac{Q(-X)^{2m}}{Q(-4X)}-2\kappa Y\,\frac{Q(-X)^{2m-1}Q'(-X)}{Q(-4X)} .
\]
Write $U=-X$. The component of degree $4m$ of the first term is the coefficient of $U^m$ times $(-X)^m$, and that of
the second term comes from the coefficient of $U^b$, since $Y$ has degree $4$. Hence the component of degree $4m$
of $K(P)$ is $\alpha X^m+\alpha' X^{m-1}Y$, with
\[
  \alpha=(-1)^m[U^m]\frac{Q(U)^{2m}}{Q(4U)},\qquad
  \alpha'=-\frac{\kappa}{m}(-1)^{b}[U^{b}]\frac{(Q(U)^{2m})'}{Q(4U)},
\]
using $2m\,Q^{2m-1}Q'=(Q^{2m})'$. By Proposition~\ref{prop:ek-ring},
$\langle X^m\rangle=(-1)^b\kappa$ and $\langle X^{m-1}Y\rangle=(-1)^{b-1}$, so
\begin{align*}
  \langle K_m(P)\rangle&=(-1)^b\kappa\,\alpha+(-1)^{b-1}\alpha'\\
  &=-\kappa\Bigl([U^m]\frac{Q^{2m}}{Q(4U)}-\frac1m[U^b]\frac{(Q^{2m})'}{Q(4U)}\Bigr)=-\kappa F_Q(b).\qedhere
\end{align*}
\end{proof}

With these results in hand, we can now prove our factorisation of the Eells--Kuiper invariant for the star
spheres we consider.

\begin{proof}[Proof of Theorem~\ref{thm:ek}]
We may assume that $\beta$ is in normal form (Lemma~\ref{lem:ek-normal-form} and the reduction that follows it).
By Proposition~\ref{prop:ek-filling}(ii) and~(iv), $Z$ is a compact spin manifold whose oriented boundary is
$\Sigma_{\rho,\beta}$, a homotopy $(4m-1)$-sphere, so Definition~\ref{def:ek-mu} applies with $W=Z$. The polynomial $N_m$ involves only $p_1,\dots,p_b$ (Lemma~\ref{lem:ek-normalisation}), we
have $P_i=p_i(TZ)$ for $i\le b$ (Proposition~\ref{prop:ek-pont}), and the classes $\bar p_i=j^{-1}(p_i(TZ))$ of
Definition~\ref{def:ek-mu} correspond to the classes $P_i$ under the identification of $H^*(\Zh;\QQ)$ with $H^*(Z,\partial Z;\QQ)$ in positive degrees (Section~\ref{subsec:ek-ring}). Hence
$\langle N_m(p(TZ)),[Z,\partial Z]\rangle=\langle N_m(P),[\Zh]\rangle$. As $N_m=\Ahat_m-t_mL_m$ and $N_m$ has no
$p_m$ term, Proposition~\ref{prop:ek-evaluation} gives
$\langle N_m(P),[\Zh]\rangle=-\kappa\bigl(F_{\Ahat}(b)-t_mF_L(b)\bigr)=-\kappa\,a_mr_b$. By
Lemma~\ref{lem:ek-signature}, $\sgn(Z)=0$ and by Lemma~\ref{lem:ek-euler}, $\kappa=-k(\beta)$. Therefore,
Definition~\ref{def:ek-mu} gives $\mu(\Sigma_{\rho,\beta})=-\kappa\,r_b=k(\beta)\,r_b$.
\end{proof}

\begin{remark}[Signs]\label{rem:ek-signs}
Two signs in the proof of Theorem~\ref{thm:ek} depend on orientation conventions. First, in the proof of
Proposition~\ref{prop:ek-evaluation}, the sign $\langle X^bY,[\Zh]\rangle=(-1)^{b-1}$ cancels the sign $(-1)^m$
produced by the substitution $U=-X$, so no factor $(-1)^b$ survives in $-\kappa F_Q(b)$. The sign $(-1)^{b-1}$ comes from the
orientation $O$ of $\HP^{b-1}$ (Lemma~\ref{lem:ek-orientation}(ii)). Second, the overall sign comes from
$\kappa=-k(\beta)$ (Lemma~\ref{lem:ek-euler}). If the two disks are glued by the reentrance $\sigma_{\beta^{-1}}$
instead of $\sigma_\beta$, the sign of $\mu$ is reversed (Remark~\ref{rem:ek-conventions}(i)).
\end{remark}

\subsection{A closed form}\label{subsec:ek-closed}

In this subsection we write $F_L(b)$ and $F_{\Ahat}(b)$, and hence $r_b$, each as a single coefficient of an
explicit power series in one variable (Proposition~\ref{prop:ek-closed-form}). The definition~\eqref{eq:ek-FQ}
of $F_Q(b)$ involves the $(2b+2)$-th power of a power series. The closed form avoids this power, and we use the closed form
for the computations of $r_b$ in Section~\ref{sec:main}. The closed form follows from a Lagrange--B\"urmann identity
(Lemma~\ref{lem:lagrange}), which we prove first and which is also the starting point of
Section~\ref{sec:twoadic}, where we find the power of $2$ in the denominator of $r_b$.

Recall from Section~\ref{subsec:ek-statement} that, for a power series $Q(U)$ with $Q(0)=1$, the number $F_Q(b)$
is defined by~\eqref{eq:ek-FQ}. There, $[U^j]f$ denotes the coefficient of $U^j$ in a power series $f$, and the
prime is $d/dU$. We write $F_L(b)$ and $F_{\Ahat}(b)$ for $F_Q(b)$ with $Q=Q_L$ and $Q=Q_{\Ahat}$, the
characteristic power series of the $L$-genus and of the $\Ahat$-genus (Section~\ref{subsec:ek-statement}). In the
same way, $[y^{2m}]f$ denotes the coefficient of $y^{2m}$ in a power series $f(y)$. The functions
$\operatorname{artanh}$ and $\operatorname{arsinh}$ are the inverses of $\tanh$ and $\sinh$, and both series in
Proposition~\ref{prop:ek-closed-form} are even power series in $y$ with constant term $1$.

We use the following Lagrange--B\"urmann identity. Let $Q(U)$ be a power series with $Q(0)=1$. Let
$\varphi(x)=x/Q(x^2)$, an odd power series with $\varphi'(0)=1$, let $\psi$ be its compositional inverse, and let
$\omega(x)=1/Q(4x^2)$.

\begin{lemma}\label{lem:lagrange}
For every $j\ge1$,
\[
  F_Q(j-1)=[y^{2j}]\,\omega\bigl(\psi(y)\bigr).
\]
Equivalently, there is an identity of power series in $x$,
\[
  \sum_{j\ge0}F_Q(j-1)\,\varphi(x)^{2j}=\omega(x),
\]
where the term $j=0$ is $1$.
\end{lemma}

\begin{proof}
Let $\Lambda(U)=Q(U)^{2j}$ and $\Omega(U)=1/Q(4U)$. By~\eqref{eq:ek-FQ} with $b=j-1$, we have
$F_Q(j-1)=[U^j]\Lambda\Omega-\frac1j[U^{j-1}]\Lambda'\Omega$. For every power series $f$ we have
$[U^{j-1}]f'=j[U^j]f$. Applying this to $f=\Lambda\Omega$, whose derivative is $\Lambda'\Omega+\Lambda\Omega'$, gives
\[
  [U^j]\Lambda\Omega-\frac1j[U^{j-1}]\Lambda'\Omega=\frac1j[U^{j-1}]\Lambda\Omega' .
\]
Hence $F_Q(j-1)=\frac1j[U^{j-1}]\Lambda\Omega'$. Substitute $U=x^2$. Then $\Lambda(x^2)=Q(x^2)^{2j}$,
$\Omega(x^2)=\omega(x)$ and $\omega'(x)=2x\,\Omega'(x^2)$, so
\[
  [U^{j-1}]\Lambda\Omega'=[x^{2j-2}]\,Q(x^2)^{2j}\,\frac{\omega'(x)}{2x}=\frac12[x^{2j-1}]\,\omega'(x)\,Q(x^2)^{2j} .
\]
Hence $F_Q(j-1)=\frac{1}{2j}[x^{2j-1}]\,\omega'(x)\,Q(x^2)^{2j}$. The Lagrange--B\"urmann
formula~\cite[Theorem~2.1.1, p.~214]{Gessel}, \cite[Theorem~5.4.2, p.~38]{StanleyEC2} for the inverse of
$y=x/Q(x^2)$ states that $[y^k]f(\psi(y))=\frac1k[x^{k-1}]f'(x)\,Q(x^2)^k$ for every power series $f$ and every
$k\ge1$. With $k=2j$ and $f=\omega$, this is the first claim. The series $\omega$ is even and $\psi$ is odd, so
$\omega\circ\psi$ is even. Substituting $y=\varphi(x)$ gives the second claim.
\end{proof}

\begin{proposition}\label{prop:ek-closed-form}
For every $b\ge1$,
\[
  F_L(b)=[y^{2m}]\frac{y}{(1+y^2)\operatorname{artanh}y},\qquad
  F_{\Ahat}(b)=[y^{2m}]\frac{y\sqrt{1+y^2/4}}{2\operatorname{arsinh}(y/2)},
\]
and consequently
\begin{equation}\label{eq:ek-rb-closed}
  r_b=\frac{1}{a_m}\,[y^{2m}]\left(\frac{y\sqrt{1+y^2/4}}{2\operatorname{arsinh}(y/2)}
     -t_m\,\frac{y}{(1+y^2)\operatorname{artanh}y}\right).
\end{equation}
\end{proposition}

\begin{proof}
By Lemma~\ref{lem:lagrange} with $j=m=b+1$, $F_Q(b)=[y^{2m}]\,\omega(\psi(y))=[y^{2m}]\,1/Q\bigl(4\psi(y)^2\bigr)$,
where $\psi$ is the compositional inverse of $\varphi(x)=x/Q(x^2)$.

For the $L$-genus, $Q_L(x^2)=x/\tanh x$, so $\varphi=\tanh$ and $\psi=\operatorname{artanh}$, and
$1/Q_L(4x^2)=\tanh(2x)/(2x)$. With $x=\operatorname{artanh}y$, the double-angle identity
$\tanh 2x=2\tanh x/(1+\tanh^2x)$ gives $\tanh 2x=2y/(1+y^2)$. Hence
\[
  \frac{1}{Q_L\bigl(4\psi(y)^2\bigr)}=\frac{y}{(1+y^2)\operatorname{artanh}y}.
\]

For the $\Ahat$-genus, $Q_{\Ahat}(x^2)=(x/2)/\sinh(x/2)$, so $\varphi(x)=2\sinh(x/2)$ and
$\psi(y)=2\operatorname{arsinh}(y/2)$, and $1/Q_{\Ahat}(4x^2)=\sinh x/x$. With $x=\psi(y)=2c$, where
$\sinh c=y/2$, the double-angle identity $\sinh 2c=2\sinh c\cosh c$ gives $\sinh x=y\sqrt{1+y^2/4}$. Hence
\[
  \frac{1}{Q_{\Ahat}\bigl(4\psi(y)^2\bigr)}=\frac{y\sqrt{1+y^2/4}}{2\operatorname{arsinh}(y/2)}.
\]
The formula for $r_b$ follows from its definition in Section~\ref{subsec:ek-statement}.
\end{proof}

The first values are listed in Table~\ref{tab:ek-rb}. The sign of $r_b$ is $(-1)^b$ in
this range.

\begin{table}[ht]
\centering\footnotesize
\renewcommand{\arraystretch}{1.15}
\begin{tabular}{cccc}
$b$ & $n$ & $r_b$ & denominator of $r_b$\\\hline
2 & 11 & $1/5952$ & $2^6\cdot3\cdot31$\\
3 & 15 & $-1/12192$ & $2^5\cdot3\cdot127$\\
4 & 19 & $449/54942720$ & $2^{10}\cdot3\cdot5\cdot7^2\cdot73$\\
5 & 23 & $-95033/29712614400$ & $2^{10}\cdot3^4\cdot5^2\cdot7\cdot23\cdot89$\\
6 & 27 & $6649999/20925344563200$ & $2^{14}\cdot3^4\cdot5^2\cdot7\cdot11\cdot8191$\\
7 & 31 & $-3695751877/28565710677504000$ & $2^{12}\cdot3^5\cdot5^3\cdot7^3\cdot11\cdot13\cdot31\cdot151$\\
8 & 35 & $80575914071/5983354894712832000$ & $2^{18}\cdot3^7\cdot5^3\cdot7^2\cdot13\cdot131071$\\
9 & 39 & $-6090643156337/1065613112080957440000$ & $2^{18}\cdot3^6\cdot5^4\cdot7\cdot11\cdot13\cdot17\cdot524287$\\
\end{tabular}
\caption{The coefficient $r_b$ of Theorem~\ref{thm:ek}, with $m=b+1$. In every row the denominator of
$r_b$ is divisible by $2^{2m-1}-1$, the odd factor of the denominator of $t_m$; for example $31$ for
$b=2$, and $511=7\cdot73$ for $b=4$.}\label{tab:ek-rb}
\end{table}

\begin{remark}[Checks]\label{rem:ek-checks}
The computations of this section have been checked by computer. The programs and their outputs are archived at Zenodo~\cite{GGINFcode}, where a README describes each program, the statement it supports and its controls. In exact rational arithmetic, for $2\le b\le12$, they check the series $Q_L$ and
$Q_{\Ahat}$ against Bernoulli numbers; $t_m$ and Lemma~\ref{lem:ek-normalisation}; the
consistency of the ring of Proposition~\ref{prop:ek-ring} with the Thom-class identities in the proof of
Proposition~\ref{prop:ek-ring}(iii) and the
restrictions to $B$ and $F$ in Propositions~\ref{prop:ek-ring} and~\ref{prop:ek-pont}; and the
signature. They compute $r_b$ in three ways, from the polynomial
$N_m$ evaluated in the ring, from Proposition~\ref{prop:ek-evaluation} and from
Proposition~\ref{prop:ek-closed-form}, and check the evaluation of
Sections~\ref{subsec:ek-ring}--\ref{subsec:ek-evaluation} by a fixed-point computation on the zero section $B$ for the circle
$S^1=\{e^{i\theta}\}\subset G$ acting on $Z$; Section~\ref{sec:main} extends this last check to
$b\le64$. The circle acts on $D^4_R\times D(\gamma)$ by left multiplication on the fibres of $\gamma$, which
commutes with $\tilde f_a$ since $\tilde f_a$ is left $\HH$-linear on each fibre, and on $\tilde D^4\times\HP^b$ by
$[w:t]\mapsto[w:tz^{-1}]$, $z\in S^1$. Since $\iota(zhw)=[w:\bar h z^{-1}]$, the two actions agree under the
identification that defines $Z$. The fixed set is $B\cup(\tilde D^4\times\{[0:1]\})$. Under the conventions of~\cite{DHZ} ($\langle xy\rangle=1$, $e=x+ky$) the computation gives
$\mu=k/5952$ and the intermediate numbers $\langle\bar p_1p_1^2\rangle=-16k$,
$\langle\bar p_1p_2\rangle=-4k$ of~\cite[(3.10)--(3.12)]{DHZ}; under those of~\cite{CDZ} it gives
\cite[(5.5), (5.7)]{CDZ} and $\mu=15m'/1016$ at $k=180m'$~\cite[Theorem~1.2]{CDZ}. The same
code gives $L=0$ and $\Ahat\in a_m\ZZ$ on the closed spin manifolds $\HP(\eta_k\oplus\HH^b)$,
where $\eta_k\to S^4$ is the quaternionic line bundle with second Chern number $k$, and it
fails on them when their ring relation is perturbed. A second program checks numerically the signs in
Lemmas~\ref{lem:ek-twisted}, \ref{lem:ek-orientation} and~\ref{lem:ek-euler}, including the value
$\chi(S^4)=2$ obtained by the computation in the proof of Lemma~\ref{lem:ek-euler}.
\end{remark}

\endgroup
\section{Realisable degrees}\label{sec:degrees}

The Eells--Kuiper formula of Section~\ref{sec:ek} expresses $\mu(\Sigma_{\rho_b,\beta})$ as a rational multiple of the
slice degree $k(\beta)$. To produce exotic spheres we therefore need to know which degrees are realised by smooth
conjugation-equivariant maps. In this section, we show that an explicit integer $k_0(b)$, the product of the exponents
of the groups $\pi_{4j+2}(S^3)$ for $1\le j\le b-1$, is always realised. These groups are finite by the work of
Serre~\cite[Chapter~V, Proposition~3, p.~494]{Serre}. Recall that the exponent of a finite abelian group is the least
common multiple of the orders of its elements, that is, the least positive integer $e$ with $ex=0$ for every element
$x$. The argument was carried out for $b=3$ by Chen, Deng and Zhang~\cite[\S3.1 and Lemma~3.1, pp.~6--7]{CDZ}. We
observe that it holds for every $b$. We also state the arithmetic properties of $k_0(b)$ we use in the proof of
Theorem~\ref{thm:main.thm.A}: its prime factors are small, and its $2$-adic valuation is at most $2(b-1)$, and at
most $2b-5$ for $b\ge5$.

\subsection{The tower of automorphism groups}
In this subsection we show that the slice degrees of smooth conjugation-equivariant maps for $\rho_b$ contain the
image of a homomorphism $\operatorname{ev}_*\colon\pi_3(\mathcal G_{b-1})\to\ZZ$ (Lemma~\ref{lem:realisation}). We
then bound this image from below by climbing a tower of groups $\mathcal G_{b-1}\to\cdots\to\mathcal G_0=S^3$, whose
successive fibres are the loop spaces $\Omega^{4j}S^3$ (Lemma~\ref{lem:restriction-fibration} and
Proposition~\ref{prop:degree-tower}).

Recall from Section~\ref{sec:prelim} that $S(V_b) = S^2 * S^{4b-1}$, and that $G = S^3$ acts freely on $S^{4b-1} = S(\HH^b)$
by left multiplication, with quotient $\HP^{b-1}$. For $0\le j\le b-1$ let
\[
 \mathcal G_j = \operatorname{Map}_G\bigl(S^{4j+3}, G_{\mathrm{conj}}\bigr)
\]
be the group, under pointwise multiplication, of maps $c\colon S^{4j+3}\to G$ with
$c(qw) = q\,c(w)\,q^{-1}$. In this section, all mapping spaces consist of continuous maps and carry the compact-open topology, and their homotopy groups are based at the constant map $1$. Smooth maps are obtained afterwards, in Lemma~\ref{lem:realisation}. Via the map $c\mapsto\bigl(w\mapsto c(w)^{-1}w\bigr)$, the group $\mathcal G_j$ is isomorphic to the group of automorphisms of the quaternionic Hopf bundle
$S^{4j+3}\to\HP^{j}$ that cover the identity. The map $c\mapsto\bigl(w\mapsto c(w)w\bigr)$ reverses the order of multiplication, as in Lemma~\ref{lem:reentrance}(1). For $j=0$ the base is a point, and an element of $\mathcal G_0$ is determined by its
value at $1$, so
\[
 \mathcal G_0 \cong S^3 .
\]
We regard $S^{4j-1}\subset S^{4j+3}$ as the unit vectors whose last quaternionic coordinate vanishes. These
inclusions are $G$-invariant, since left multiplication by $q\in G$ preserves the vanishing of the last coordinate.
Hence restriction defines homomorphisms
\[
 \operatorname{res}_j\colon \mathcal G_j \longrightarrow \mathcal G_{j-1}, \qquad \operatorname{res}_j(c)=c|_{S^{4j-1}},
 \qquad 1\le j\le b-1 .
\]
The restriction of an equivariant map to an invariant subset is equivariant, and restriction commutes with pointwise
multiplication.
The composite of the $\operatorname{res}_j$, $\operatorname{ev}\colon\mathcal G_{b-1}\to\mathcal G_0 = S^3$, is
evaluation at $w_0 = (1,0,\dots,0)$.

Recall that $E(\rho_b)$ is the group of equivariant homotopy classes of equivariant maps $S(V_b)\to G_{\mathrm{conj}}$ (Section~\ref{subsec:star}). Let $K_b = k\bigl(E(\rho_b)\bigr)\subset\ZZ$ be the set of slice degrees of smooth conjugation-equivariant maps for $\rho_b$, which is a
subgroup by Lemma~\ref{lem:ek-slice-degree}(iii). The next lemma shows that $K_b$ contains the image of
$\operatorname{ev}_*\colon\pi_3(\mathcal G_{b-1})\to\pi_3(S^3)=\ZZ$.

\begin{lemma}\label{lem:realisation}
Let $c\colon(D^3,\partial D^3)\to(\mathcal G_{b-1},1)$ be continuous. Then there is a smooth conjugation-equivariant map $\beta$ for $\rho_b$ with $k(\beta) = \deg\bigl(p\mapsto c(p)(w_0)\bigr) = \operatorname{ev}_*[c]$.
\end{lemma}

\begin{proof}
The construction is that of~\cite[\S3.3, (3.4)--(3.5)]{DHZ}, where $b=2$. Points of $S(V_b)$ with $|p|<1$ are uniquely of the form $(p,(1-|p|^2)^{1/2}w)$ with $p\in\operatorname{int}D^3$
and $w\in S^{4b-1}$. Define $\beta_c\colon S(V_b)\to G$ by
\[
 \beta_c\bigl(p,(1-|p|^2)^{1/2}w\bigr) = c(p)(w)\quad(|p|<1),\qquad \beta_c = 1 \text{ on } S(\RR^3).
\]
After a homotopy of $c$ rel $\partial D^3$ we may assume that $c(p) = 1$ for $|p|\ge\frac13$. Then $\beta_c=1$
on a neighbourhood of $S(\RR^3)$, so $\beta_c$ is continuous. The map $\beta_c$ is conjugation-equivariant for $\rho_b$, that is,
$\beta_c(\rho_b(q)y)=q\,\beta_c(y)\,q^{-1}$ for $q\in G$ and $y\in S(V_b)$. On $S(\RR^3)$ both sides equal $1$. For
$y=(p,(1-|p|^2)^{1/2}w)$ with $|p|<1$, we have $\rho_b(q)y=(p,(1-|p|^2)^{1/2}qw)$, so
$\beta_c(\rho_b(q)y)=c(p)(qw)=q\,c(p)(w)\,q^{-1}$. By equivariant smooth approximation~\cite[Chapter~VI, Theorem~4.2, p.~317]{Bredon}, applied to the closed invariant set $A=\{(p,v)\in S(V_b):|p|\ge\frac12\}$, on a neighbourhood of which
$\beta_c=1$, the map $\beta_c$ is equivariantly homotopic to a smooth equivariant map $\beta$ that equals $1$ on $A$. The proof of
Lemma~\ref{lem:ek-slice-degree}(ii) applies without change to this homotopy, so $k(\beta) = \deg(\beta_c\circ s_{w_0})$, and the slice
$s_{w_0}(p) = (p,(1-|p|^2)^{1/2}w_0)$ gives $\beta_c\circ s_{w_0}(p) = c(p)(w_0)$. Hence
$k(\beta) = \deg(p\mapsto c(p)(w_0))$, the image of $[c]$ under
$\operatorname{ev}_*\colon\pi_3(\mathcal G_{b-1})\to\pi_3(S^3)$.
\end{proof}

\begin{lemma}\label{lem:restriction-fibration}
For $1\le j\le b-1$, the restriction $\operatorname{res}_j\colon\mathcal G_j\to\mathcal G_{j-1}$, $c\mapsto c|_{S^{4j-1}}$, is a Serre
fibration whose fibre over the identity is homeomorphic to the pointed mapping space
$\operatorname{Map}_*(S^{4j},S^3) = \Omega^{4j}S^3$.
\end{lemma}

\begin{proof}
Let $G$ act on the right of $S^{4j+3}$ by $w\cdot q=q^{-1}w$. By the discussion above, $c\mapsto\varphi_c$,
$\varphi_c(w)=c(w)^{-1}w$, maps $\mathcal G_j$ isomorphically onto the group of principal bundle maps of the Hopf
bundle $S^{4j+3}\to\HP^j$ that cover the identity. Its inverse sends $\varphi$ to the map $c$ with
$c(w)^{-1}=\sum_k\varphi(w)_k\,\bar w_k$, since $\varphi(w)=c(w)^{-1}w$ and $\sum_kw_k\bar w_k=1$. Both maps are
continuous, so this is an isomorphism of topological groups. Under this identification, $\operatorname{res}_j$ is
restriction to the part $S^{4j-1}$ of the bundle over $\HP^{j-1}$. Since $\HP^{j-1}$ is a subcomplex of $\HP^j$,
restriction is a Serre fibration~\cite[Proposition~(3.4), p.~27, and diagram~(5.4), p.~31]{Gottlieb}; see
also~\cite[\S3.1]{CDZ}.

By the definition of $\operatorname{res}_j$, its fibre over $1$ is $\{c\in\mathcal G_j : c=1\text{ on }S^{4j-1}\}$.
The pair $(\HP^{j},\HP^{j-1})$ is a relative CW complex with a single $4j$-cell, whose interior is
$\{[z:1] : z\in\HH^j\}$. Over it, $\sigma(z)=(z,1)/|(z,1)|$ is a section of the Hopf bundle. For $c$ in the fibre,
let $f_c(z)=c(\sigma(z))$ and $f_c(\infty)=1$. As $|z|\to\infty$, the point $\sigma(z)$ tends to $S^{4j-1}$, where $c=1$, so
$f_c\colon\HH^j\cup\{\infty\}=S^{4j}\to G$ is continuous and based. Conversely, every point off $S^{4j-1}$ is
$q\sigma(z)$ for exactly one pair $(q,z)$. So a based map $f$ determines $c$ by $c(q\sigma(z))=q\,f(z)\,q^{-1}$ and
$c=1$ on $S^{4j-1}$. The map $c$ is continuous. If $q\sigma(z)$ tends to a point of $S^{4j-1}$, then
$|z|\to\infty$, so $f(z)\to1$ and $|q\,f(z)\,q^{-1}-1|=|f(z)-1|\to0$. Since $|qxq^{-1}-qyq^{-1}|=|x-y|$, these
inverse bijections preserve the uniform distance, which induces the compact-open topology on maps from a compact
space to $G$. Hence they are homeomorphisms between the fibre and $\operatorname{Map}_*(S^{4j},S^3)=\Omega^{4j}S^3$.
\end{proof}

\begin{proposition}\label{prop:degree-tower}
For $1\le j\le b-1$, let $e_j$ denote the exponent of the finite abelian group $\pi_{4j+2}(S^3)$. Then
\[
 k_0(b) := \prod_{j=1}^{b-1} e_j \;\in\; K_b .
\]
\end{proposition}

\begin{proof}
We follow~\cite[proof of Lemma~3.1]{CDZ}, where $b=3$. For $0\le j\le b-1$, let $r_j\colon\mathcal G_j\to\mathcal G_0$ be the composite of the restrictions, with
$r_0=\mathrm{id}$, so that $r_j=r_{j-1}\circ\operatorname{res}_j$ for $j\ge1$, and let
$K_{(j)}=\operatorname{im}(r_j)_*\subset\pi_3(\mathcal G_0)=\ZZ$. Thus $K_{(0)} = \ZZ$ and
$K_{(b-1)} = \operatorname{im}\operatorname{ev}_*$.

Fix $1\le j\le b-1$. By Lemma~\ref{lem:restriction-fibration}, the long exact homotopy sequence of the Serre
fibration $\operatorname{res}_j$ contains the exact segment
\[
 \pi_3(\mathcal G_j)\xrightarrow{\ (\operatorname{res}_j)_*\ }\pi_3(\mathcal G_{j-1})
 \xrightarrow{\ \partial\ }\pi_2(\Omega^{4j}S^3) = \pi_{4j+2}(S^3).
\]
All groups here are homotopy groups in degree at least two, so $\partial$ is a homomorphism of abelian
groups. Let $x\in\pi_3(\mathcal G_{j-1})$. Then $\partial(e_jx) = e_j\,\partial(x) = 0$, since $\pi_{4j+2}(S^3)$ has
exponent $e_j$. So, by exactness, $e_jx=(\operatorname{res}_j)_*(y)$ for some $y\in\pi_3(\mathcal G_j)$. If
$k=(r_{j-1})_*(x)\in K_{(j-1)}$, then $(r_j)_*(y)=(r_{j-1})_*(e_jx)=e_jk$, since $(r_{j-1})_*$ is a homomorphism.
Hence
\[
 e_j\,K_{(j-1)} \subset K_{(j)} .
\]
Now $1\in K_{(0)}$, and $e_1\cdots e_{j-1}\in K_{(j-1)}$ gives $e_1\cdots e_j\in K_{(j)}$. Hence, by induction,
$k_0(b)=e_1\cdots e_{b-1}\in K_{(b-1)}=\operatorname{im}\operatorname{ev}_*$, which is contained in $K_b$ by
Lemma~\ref{lem:realisation}.
\end{proof}

\subsection{Arithmetic of $k_0(b)$}\label{subsec:k0-arithmetic}

In this subsection we bound the prime factors of $k_0(b)$ and the power of $2$ dividing it
(Corollary~\ref{cor:k0-arithmetic}). These are the two properties of $k_0(b)$ that the proof of
Theorem~\ref{thm:main.thm.A} uses. The bounds come from what is known about the torsion in $\pi_*(S^3)$
(Lemma~\ref{lem:exponents}).

Let $p$ be a prime. The \emph{$p$-primary component} $A_{(p)}$ of a finite abelian group $A$ is the subgroup of
elements whose order is a power of $p$, and the \emph{$p$-primary part} of a positive integer $e$ is the largest
power of $p$ dividing $e$. The exponent of $A$ is the product of the exponents of the groups $A_{(p)}$, so the
$p$-primary part of $e_j$ is the exponent of $\pi_{4j+2}(S^3)_{(p)}$.

\begin{lemma}\label{lem:exponents}
Let $j\ge1$.
\begin{enumerate}
\item If an odd prime $p$ divides $e_j$, then $p\le 2j+1$, and $p^2\nmid e_j$.
\item The $2$-primary part of $e_j$ divides $4$.
\end{enumerate}
\end{lemma}

\begin{proof}
(1) The $p$-primary component of $\pi_k(S^3)$ vanishes for $3<k<2p$, the first nonzero one being
$\pi_{2p}(S^3)_{(p)}\cong\ZZ/p$ \cite[(13.6)$'$, p.~177]{Toda}. So $p\mid e_j$ forces $4j+2\ge 2p$, that is $p\le 2j+1$. By
Selick's theorem~\cite[Theorem, p.~407]{Selick78}, every nonzero element of the $p$-primary component of $\pi_k(S^3)$,
$k>3$, has order $p$. Hence $p^2\nmid e_j$.

(2) By James~\cite[Corollary~(1.22), p.~80]{James57}, the homotopy groups of $S^{n+1}$ contain no elements of
order $2^{k}$, where $k = n+1$ or $2n+2$ according as $n$ is even or odd. For $n=2$ this says that
$\pi_*(S^3)$ has no element of order $8$, so the $2$-primary component of each $\pi_{4j+2}(S^3)$ has exponent
dividing $4$.
\end{proof}

For a nonzero rational number $x$, we write $v_2(x)$ for its $2$-adic valuation: the integer $v$ such that
$x=2^v a/c$ with $a$ and $c$ odd integers.

\begin{corollary}\label{cor:k0-arithmetic}
For every $b\ge2$:
\begin{enumerate}
\item every prime divisor of $k_0(b)$ is at most $2b-1$;
\item $v_2(k_0(b))\le 2(b-1)$;
\item if $b\ge5$, then $v_2(k_0(b))\le 2b-5$.
\end{enumerate}
\end{corollary}

\begin{proof}
For an odd prime $p\mid k_0(b)$, some $e_j$ with $1\le j\le b-1$ is divisible by $p$, so
$p\le 2j+1\le 2b-1$ by Lemma~\ref{lem:exponents}(1); the prime $2$ also satisfies the bound since $b\ge2$.
Part (2) follows from Lemma~\ref{lem:exponents}(2), since $k_0(b)$ is a product of $b-1$ factors $e_j$.
For (3), Toda's tables give $\pi_6(S^3)\cong\ZZ/12$, $\pi_{10}(S^3)\cong\ZZ/15$
\cite[Table~I, p.~186]{Toda}, $\pi_{14}(S^3)\cong\ZZ/84\oplus(\ZZ/2)^2$ and $\pi_{18}(S^3)\cong\ZZ/30$
\cite[Table~II, p.~187]{Toda}. Hence $e_1,e_2,e_3,e_4 = 12,15,84,30$, and
$v_2(e_1e_2e_3e_4) = 2+0+2+1 = 5$. Each of the remaining $b-5$ factors $e_j$, $5\le j\le b-1$, contributes at most
$2$ by Lemma~\ref{lem:exponents}(2), so $v_2(k_0(b))\le 5+2(b-5) = 2b-5$.
\end{proof}

\begin{remark}\label{rem:k0-values}
Toda's tables also give $\pi_{22}(S^3)\cong\ZZ/132\oplus\ZZ/2$ \cite[Table~II, p.~187]{Toda}, so
$e_1,\dots,e_5 = 12, 15, 84, 30, 132$ and
$k_0(2),\dots,k_0(6) = 12,\ 180,\ 15120,\ 453600,\ 59875200$. The values $k_0(2)=12$ and $k_0(3)=180$ are the thresholds
in~\cite[Lemma~3.1]{DHZ} and~\cite[Lemma~3.1]{CDZ}.
\end{remark}

\section{The $2$-adic valuation of the Eells--Kuiper invariant}\label{sec:twoadic}

In this section we determine the $2$-adic valuation of the coefficient $r_b$ in the Eells--Kuiper formula
$\mu(\Sigma_{\rho_b,\beta})=k(\beta)\,r_b$ of Section~\ref{sec:ek}. Recall from Section~\ref{subsec:k0-arithmetic} that the $2$-adic valuation $v_2(x)$ of a
nonzero rational number $x$ is the integer $v$ such that $x=2^v a/c$ with $a$ and $c$ odd integers. With $m=b+1$, we show that the $\Ahat$- and $L$-contributions $F_{\Ahat}(b)$ and
$t_mF_L(b)$ have valuations $-m$ and $-(2m-1)+v_2(m)$ respectively, for every $b$. The proof rests on two observations. First, the
Lagrange--B\"urmann identity of Lemma~\ref{lem:lagrange} expresses each contribution as one coefficient of an explicit power
series. Second, these series satisfy functional equations coming from the double-angle identities for $\tanh 2x$
and $\cosh 2x$, and these equations force $2$-adic integrality. Combined with Corollary~\ref{cor:k0-arithmetic}, this shows that
$k_0(b)\,r_b$ has negative $2$-adic valuation, and hence $\mu\neq0$, for all $b\not\equiv 63 \pmod{64}$.

Throughout, let $m=b+1$. For a commutative ring $R$, we write $R[[T]]$ for the ring of formal power series in $T$
with coefficients in $R$. We use the numbers $F_Q(b)$ of
\eqref{eq:ek-FQ}, for a power series $Q(U)\in\QQ[[U]]$ with $Q(0)=1$; there $[U^{b}] = [U^{m-1}]$. In
particular $F_{\Ahat}$ and $F_L$ are defined by the characteristic power series $Q_{\Ahat}$ and $Q_L$ of
Section~\ref{subsec:ek-statement}.

\subsection{Generating functions}
We use the notation $\varphi$ of Lemma~\ref{lem:lagrange}. For the $L$-genus, $\varphi(x) = \tanh x$ and
$1/Q_L(4x^2) = \tanh(2x)/(2x)$. For the $\Ahat$-genus, $\varphi(x) = 2\sinh(x/2)$, $\varphi(x)^2 = 2(\cosh x-1)$
and $1/Q_{\Ahat}(4x^2) = \sinh x/x$ (see the proof of Proposition~\ref{prop:ek-closed-form}). With the even series
\[
 T = \tanh^2 x, \qquad W = \cosh x - 1,
\]
the second form of Lemma~\ref{lem:lagrange} therefore reads
\begin{equation}\label{eq:gen-fns}
 \sum_{m\ge0} F_L(m-1)\,T^m = \frac{\tanh 2x}{2x},\qquad
 \sum_{m\ge0} 2^mF_{\Ahat}(m-1)\,W^m = \frac{\sinh x}{x}.
\end{equation}
Since $T = x^2+O(x^4)$ and $W = x^2/2+O(x^4)$, every even power series in $x$ can be written uniquely as a
power series in $T$, and also as one in $W$.

\subsection{Functional equations from the double-angle identities}
In this subsection we show that the generating functions in~\eqref{eq:gen-fns} satisfy functional equations coming
from the double-angle identities (Lemma~\ref{lem:functional}). We then use these equations to write the generating
functions as infinite products with $2$-adic integer coefficients (Lemma~\ref{lem:integrality}). This integrality
gives the valuations in Section~\ref{subsec:valuations}.

Define $h\in\QQ[[T]]$ and $e\in\QQ[[W]]$ by
\[
 h(\tanh^2x) = \frac{\tanh x}{x},\qquad e(\cosh x-1) = \frac{\sinh x}{x}.
\]
Let $\varrho(T) = 4T/(1+T)^2$ and $\varsigma(W) = 4W+2W^2$.

\begin{lemma}\label{lem:functional}
The series $h$ and $e$ satisfy
\[
 h(T) = (1+T)\,h\bigl(\varrho(T)\bigr),\qquad e\bigl(\varsigma(W)\bigr) = (1+W)\,e(W),
\]
and the generating functions in \eqref{eq:gen-fns} are $h(\varrho(T)) = h(T)/(1+T)$ and $e(W)$ respectively.
\end{lemma}

\begin{proof}
From $\tanh 2x = 2\tanh x/(1+\tanh^2x)$ we get $\tanh^2 2x = \varrho(T)$ and
\[
 \frac{\tanh 2x}{2x} = \frac{1}{1+T}\cdot\frac{\tanh x}{x} = \frac{h(T)}{1+T}.
\]
The left-hand side also equals $h(\tanh^22x) = h(\varrho(T))$, which gives the first equation and the claim about
the $L$-generating function. Similarly, $\cosh 2x - 1 = 2(1+W)^2-2 = \varsigma(W)$ and
$\sinh 2x/(2x) = \cosh x\cdot\sinh x/x = (1+W)e(W)$. The left-hand side is $e(\varsigma(W))$.
\end{proof}

We write $\ZZ_2$ for the ring of $2$-adic integers and $\QQ_2$ for its field of fractions. A rational number $x$
lies in $\ZZ_2$ if and only if $x=0$ or $v_2(x)\ge0$. A sequence in $\ZZ_2[[T]]$ converges coefficientwise if each
coefficient converges $2$-adically. For a power series $P$ with $P(0)=0$, we write $P^{\circ j}$ for the $j$-fold
composite $P\circ\cdots\circ P$, with $P^{\circ0}$ the identity series.

\begin{lemma}\label{lem:integrality}
The products
\[
 H(T) = \prod_{j\ge0}\bigl(1+\varrho^{\circ j}(T)\bigr),\qquad
 E(W) = \prod_{j\ge0}\bigl(1+\varsigma^{\circ j}(W)\bigr)^{-1}
\]
converge coefficientwise in $\ZZ_2[[T]]$ and $\ZZ_2[[W]]$, and $h = H$, $e = E$. In particular
$h\in\ZZ_2[[T]]$, $e\in\ZZ_2[[W]]$, and
\[
 e(W)\equiv (1+W)^{-1}\pmod 2 .
\]
\end{lemma}

\begin{proof}
We have $\varrho(T)\in 4T\,\ZZ[[T]]$. If $P\in 4^jT\,\ZZ[[T]]$, then $\varrho(P) = 4P/(1+P)^2\in 4^{j+1}T\,\ZZ[[T]]$.
Hence $\varrho^{\circ j}\in 4^jT\,\ZZ[[T]]$. Similarly, $\varsigma(P) = 2P(2+P)$ shows
$\varsigma^{\circ j}\in 2^jW\,\ZZ[W]$. The $j$-th factor of each product is therefore congruent to $1$ modulo
$4^j$, respectively $2^j$. So the partial products stabilise modulo every power of $2$, and the products
converge in $\ZZ_2[[\,\cdot\,]]$ with constant term $1$.

Since $\varrho(0)=0$, the coefficient $[T^n]\,P(\varrho(T))$ is an integral linear combination of the coefficients
$p_0,\dots,p_n$ of $P$. So composition with $\varrho$ is coefficientwise continuous, and likewise for $\varsigma$.
Passing to the limit in the telescoping identity
\[
 (1+T)\prod_{j=0}^{N}\bigl(1+\varrho^{\circ j}(\varrho(T))\bigr) = \prod_{j=0}^{N+1}\bigl(1+\varrho^{\circ j}(T)\bigr)
\]
gives $H(T) = (1+T)H(\varrho(T))$. In the same way $E(\varsigma(W)) = (1+W)E(W)$.

For uniqueness, let $d = h-H\in\QQ_2[[T]]$. Then $d(T) = (1+T)\,d(\varrho(T))$ and $d(0) = 0$, because
$h(0) = H(0) = 1$. If $d\ne0$, let $aT^k$ with $k\ge1$ be its lowest term. The lowest term of
$(1+T)d(\varrho(T))$ is $a\,4^kT^k$, so $a = 4^ka$ and $a=0$, a contradiction. Hence $h = H$. The same argument,
with $\varsigma(W) = 4W+O(W^2)$, gives $e = E$. Finally, every factor of $E$ with $j\ge1$ is congruent to $1$
modulo $2$, so $E\equiv(1+W)^{-1}\pmod 2$.
\end{proof}

\subsection{The valuations}\label{subsec:valuations}
In this subsection we compute the $2$-adic valuations of $F_{\Ahat}(b)$ (Proposition~\ref{prop:A-valuation}) and of
$F_L(b)$ (Proposition~\ref{prop:L-valuation}), using the integrality of Lemma~\ref{lem:integrality}. For $F_L(b)$
we determine its residue modulo a power of $2$, which separates the term $k=1$ of the sum in the proof of
Proposition~\ref{prop:L-valuation} from the others.

\begin{proposition}\label{prop:A-valuation}
For every $b\ge0$, we have $v_2\bigl(F_{\Ahat}(b)\bigr) = -(b+1)$.
\end{proposition}

\begin{proof}
By \eqref{eq:gen-fns}, $2^mF_{\Ahat}(m-1) = [W^m]\,e(W)$. By Lemma~\ref{lem:integrality} this coefficient
lies in $\ZZ_2$ and is congruent to $(-1)^m$ modulo $2$, so it is a $2$-adic unit. Hence
$v_2(F_{\Ahat}(m-1)) = -m$.
\end{proof}

\begin{proposition}\label{prop:L-valuation}
For every $b\ge0$, we have $v_2\bigl(F_L(b)\bigr) = 2+v_2(b+1)$.
\end{proposition}

\begin{proof}
Let $\theta(T) = T/(1+T)^2$, so that $\varrho = 4\theta$. By Lemma~\ref{lem:functional}, the $L$-generating
function is
\[
 \sum_{m\ge0}F_L(m-1)\,T^m = h(\varrho(T)) = \sum_{k\ge0} h_k\,4^k\,\theta(T)^k,
\]
where $h_k = [T^k]h\in\ZZ_2$ by Lemma~\ref{lem:integrality}. For $1\le k\le m$,
\[
 [T^m]\,\theta^k = [T^{m-k}](1+T)^{-2k} = (-1)^{m-k}\binom{m+k-1}{2k-1},
\]
and $[T^m]\theta^k = 0$ for $k>m$. Hence
\[
 F_L(m-1) = \sum_{k=1}^{m}(-1)^{m-k}\,4^k\,h_k\binom{m+k-1}{2k-1}.
\]
The term $k=1$: since $\tanh x/x = 1 - x^2/3 + O(x^4)$ and $T = x^2+O(x^4)$, we have $h_1 = -1/3$. So this
term equals $(-1)^m\,4m/3$, whose valuation is $2+v_2(m)$.

The terms $k\ge2$: we have
\[
 (2k-1)!\binom{m+k-1}{2k-1} = \prod_{i=-(k-1)}^{k-1}(m+i) = m\prod_{i=1}^{k-1}(m^2-i^2),
\]
an integer multiple of $m$. By Legendre's formula, $v_2((2k-1)!) = (2k-1)-s_2(2k-1)$, where $s_2$ is the binary
digit sum. Since $s_2(2k-1) = s_2(k-1)+1$, this equals $2k-2-s_2(k-1)$, which is at most $2k-3$ for $k\ge2$.
Therefore
\[
 v_2\Bigl(4^kh_k\binom{m+k-1}{2k-1}\Bigr)\ \ge\ 2k+v_2(m)-(2k-3) = v_2(m)+3 .
\]
Summing, $F_L(m-1)\equiv(-1)^m\,4m/3\pmod{2^{v_2(m)+3}}$. The right-hand side has valuation $2+v_2(m)$, hence
$v_2(F_L(m-1)) = 2+v_2(m)$, which is the claim with $m=b+1$.
\end{proof}

\subsection{Consequence for the Eells--Kuiper invariant}
In this subsection we combine the valuations of Section~\ref{subsec:valuations} into the valuation of $r_b$
(Corollary~\ref{cor:rb-valuation}). We then compare it with the bound on $v_2(k_0(b))$ of
Corollary~\ref{cor:k0-arithmetic}, which shows that $\mu(\Sigma_{\rho_b,\beta})\ne0$ for $\beta$ with $k(\beta)=k_0(b)$ whenever $b\not\equiv63\pmod{64}$
(Theorem~\ref{thm:mu-nonzero}).

Recall from Section~\ref{sec:ek} that $t_m = -1/\bigl(2^{2m+1}(2^{2m-1}-1)\bigr)$, that $a_m = 1$ for $m$ even and
$a_m = 2$ for $m$ odd, and that, by Theorem~\ref{thm:ek},
\[
 \mu\bigl(\Sigma_{\rho_b,\beta}\bigr) = k(\beta)\,r_b\quad\text{in }\QQ/\ZZ,\qquad
 r_b = \frac{F_{\Ahat}(b) - t_m F_L(b)}{a_m}.
\]
Below we use only the $2$-adic valuation of $r_b$, so the sign conventions of Section~\ref{sec:ek} play no role.

\begin{corollary}\label{cor:rb-valuation}
For every $b\ge2$,
\[
 v_2(r_b) = \begin{cases} -2m & m \text{ odd},\\ -2m+1+v_2(m) & m\text{ even}.\end{cases}
\]
\end{corollary}

\begin{proof}
Since $2^{2m-1}-1$ is odd, Proposition~\ref{prop:L-valuation} gives
$v_2(t_mF_L(b)) = -(2m+1)+2+v_2(m) = -2m+1+v_2(m)$. By Proposition~\ref{prop:A-valuation},
$v_2(F_{\Ahat}(b)) = -m$. For $m\ge3$ we have $m-v_2(m)\ge2$, so $-m > -2m+1+v_2(m)$. The two terms have
different valuations, so $v_2(F_{\Ahat}(b)-t_mF_L(b)) = -2m+1+v_2(m)$. Dividing by $a_m$ subtracts
$v_2(a_m)$, which is $1$ for $m$ odd (where $v_2(m)=0$) and $0$ for $m$ even.
\end{proof}

Together with Proposition~\ref{prop:degree-tower}, which provides $\beta$ with $k(\beta)=k_0(b)$, the next theorem
proves Theorem~\ref{thm:main.thm.A} (see Section~\ref{sec:main}).

\begin{theorem}\label{thm:mu-nonzero}
Let $b\ge2$ with $b\not\equiv63\pmod{64}$, and let $\beta$ be a smooth conjugation-equivariant map for $\rho_b$ with $k(\beta) = k_0(b)$
(Proposition~\ref{prop:degree-tower}). Then $\mu(\Sigma_{\rho_b,\beta})\ne0$.
\end{theorem}

\begin{proof}
By Corollary~\ref{cor:k0-arithmetic}(2), $v_2(k_0(b))\le 2(b-1) = 2m-4$. By Corollary~\ref{cor:rb-valuation}:
\begin{itemize}
\item if $b$ is even, then $m$ is odd and $v_2(r_bk_0(b))\le -2m+2m-4 = -4$;
\item if $b$ is odd, then $m$ is even and $v_2(r_bk_0(b))\le -2m+1+v_2(m)+2m-4 = v_2(m)-3$.
\end{itemize}
The bound $v_2(m)-3$ is negative unless $8\mid m$, that is unless $b\equiv7\pmod 8$. In that case $b\ge7$, so
Corollary~\ref{cor:k0-arithmetic}(3) gives $v_2(k_0(b))\le 2b-5 = 2m-7$, and
$v_2(r_bk_0(b))\le -2m+1+v_2(m)+2m-7 = v_2(m)-6$. This is negative unless $64\mid m$, that is unless
$b\equiv63\pmod{64}$. In every remaining case $r_bk_0(b)$ has negative $2$-adic valuation, so it is not an
integer and $\mu(\Sigma_{\rho_b,\beta}) = k_0(b)\,r_b\ne0$ in $\QQ/\ZZ$.
\end{proof}

\section{Exotic star spheres in infinitely many dimensions}\label{sec:main}

In this section, we combine the results of Sections~\ref{sec:ek}--\ref{sec:twoadic} to prove Theorem~\ref{thm:main.thm.A} and Corollary~\ref{cor:main.cor.B}. We then consider the class $b\equiv63\pmod{64}$ in Section~\ref{subsec:b63}.

\subsection{Proofs of Theorem~\ref{thm:main.thm.A} and Corollary~\ref{cor:main.cor.B}}
In this subsection we prove the main results of the paper. For $b\geq2$ with $b\not\equiv63\pmod{64}$,
Proposition~\ref{prop:degree-tower} provides a smooth conjugation-equivariant map $\beta$ with slice degree $k_0(b)$.
Theorem~\ref{thm:ek} gives $\mu(\Sigma_{\rho_b,\beta})=k_0(b)\,r_b$, and Theorem~\ref{thm:mu-nonzero} shows that this
is nonzero. Hence $\Sigma_{\rho_b,\beta}$ is an exotic sphere (Theorem~\ref{thm:main.thm.A}). The curvature theorems
of Section~\ref{sec:prelim}, together with Lemma~\ref{lem:action-bound}, then give $\Sigma_{\rho_b,\beta}$ a metric
of positive sectional curvature (Corollary~\ref{cor:main.cor.B}).

\begin{proof}[Proof of Theorem~\ref{thm:main.thm.A}]
Let $b\geq2$ with $b\not\equiv63\pmod{64}$, and let $n=4b+3$. By Proposition~\ref{prop:degree-tower}, there is a smooth conjugation-equivariant map \[
  \beta\colon S(V_b)\longrightarrow G_{\mathrm{conj}}
  \qquad\text{with}\qquad k(\beta)=k_0(b).
\]
By Theorem~\ref{thm:ek}, the Eells--Kuiper invariant of its star sphere is
\[
  \mu(\Sigma_{\rho_b,\beta})=k_0(b)\,r_b\in\Q/\Z .
\]
We now verify that this is nonzero. Since $b\geq2$, $b\not\equiv63\pmod{64}$ and
$k(\beta)=k_0(b)$, Theorem~\ref{thm:mu-nonzero} applies, and its proof shows that
\[
  v_2\bigl(k_0(b)\,r_b\bigr)<0 .
\]
Hence $k_0(b)\,r_b\notin\Z$ and $\mu(\Sigma_{\rho_b,\beta})\neq0$. Since $\mu(S^n)=0$ and $\mu$ is an invariant of oriented
diffeomorphism types, $\Sigma_{\rho_b,\beta}$ is not diffeomorphic to $S^n$ preserving
orientation. Reversing the orientation changes the sign of $\mu$, so $\Sigma_{\rho_b,\beta}$ is not diffeomorphic
to $S^n$. Thus $\Sigma_{\rho_b,\beta}$ is an exotic $n$-sphere, and it is a star sphere
for $\rho_b$ by construction.
\end{proof}

\begin{proof}[Proof of Corollary~\ref{cor:main.cor.B}]
Let $n\equiv3\pmod4$ with $n\geq11$ and $n\not\equiv255\pmod{256}$, and set
\[
  b=\frac{n-3}{4}.
\]
Then $b\geq2$, and $b\not\equiv63\pmod{64}$, since $n=4b+3\equiv255\pmod{256}$ if and only if $b\equiv63\pmod{64}$. Let $\Sigma_{\rho_b,\beta}$ be the exotic star sphere of Theorem~\ref{thm:main.thm.A}.
By Lemma~\ref{lem:action-bound},
\[
  \|K_y\|\leq1\leq2\qquad\text{for every }y\in S(V_b),
\]
so Theorem~\ref{thm:curvature-bounded} gives a metric of strictly positive sectional curvature
on $\Sigma_{\rho_b,\beta}$. Theorem~\ref{thm:curvature-general} gives the same conclusion.
\end{proof}

\subsection{The class \texorpdfstring{$b\equiv63\pmod{64}$}{b = 63 mod 64}}\label{subsec:b63}
In this subsection we treat the remaining class $b\equiv63\pmod{64}$, where the bound on $v_2(k_0(b))$ of
Corollary~\ref{cor:k0-arithmetic} is too weak for the $2$-adic argument of Theorem~\ref{thm:mu-nonzero}. We use
instead an odd prime. A primitive prime divisor of $2^{2m-1}-1$ that does not divide the numerator of $F_L(b)$
divides the denominator of $r_b$ but not $k_0(b)$ (Lemma~\ref{lem:zsigmondy} and Corollary~\ref{cor:k0-arithmetic}).
Whether it divides the numerator is a finite computation, which we carry out for $b\equiv63\pmod{64}$ with
$b<2^{14}$ (Proposition~\ref{prop:computation}). The resulting Corollary~\ref{cor:computation} is computer-assisted.
Theorem~\ref{thm:main.thm.A} and Corollary~\ref{cor:main.cor.B} do not depend on it.

By Corollary~\ref{cor:k0-arithmetic}, every prime factor of $k_0(b)$ is at most $2b-1$. Hence, if the denominator of $r_b$ in lowest terms has a prime factor $p>2b-1$, then $p$ divides the denominator of $k_0(b)\,r_b$, and $k_0(b)\,r_b\notin\Z$. Such a prime can come from the factor $2^{2m-1}-1$ in the denominator of $t_m$, where $m=b+1$.

\begin{lemma}\label{lem:zsigmondy}
Let $b\geq2$, $m=b+1$, and let $p$ be a prime divisor of $2^{2m-1}-1$ that does not divide $2^j-1$ for any $1\leq j<2m-1$. Then $p\geq4m-1>2b-1$, and $p$ does not divide the denominator of $F_{\Ahat}(b)$ or of $F_L(b)$. Moreover, if $p$ does not divide the numerator of $F_L(b)$, then $p$ divides the denominator of $r_b$.
\end{lemma}

\begin{proof}
The multiplicative order of $2$ modulo $p$ is $2m-1$, so $2m-1$ divides $p-1$. Since $p$ and $2m-1$ are odd, $2(2m-1)$ divides $p-1$, and $p\geq4m-1$. Recall that a rational number $x$ is \emph{$p$-integral} if $p$ does not divide the denominator of $x$ in lowest terms, that is, if $x=0$ or $v_p(x)\geq0$, where $v_p$ is the $p$-adic valuation, defined as $v_2$ in Section~\ref{subsec:k0-arithmetic} with $p$ in place of $2$. The factor $1/m$ in~\eqref{eq:ek-FQ} is $p$-integral since $m<p$, and the coefficients of $Q_L$ and $Q_{\Ahat}$ in degree at most $m$ involve the Bernoulli numbers $B_{2j}$ with $j\leq m$, together with the factorials $(2j)!$ and powers of $2$. By the theorem of von Staudt and Clausen~\cite[pp.~373--374]{vonStaudt}, \cite{Clausen}, \cite[Theorem~118, p.~115]{HardyWright} the primes in their denominators are at most $2m+1<p$. Hence $F_L(b)$ and $F_{\Ahat}(b)$ are $p$-integral. The prime $p$ divides $2^{2m+1}(2^{2m-1}-1)$, the denominator of $t_m$, to the same power as it divides $2^{2m-1}-1$, and $a_m\in\{1,2\}$. So if $p$ does not divide the numerator of $F_L(b)$, then $t_mF_L(b)$ has negative $p$-adic valuation while $F_{\Ahat}(b)$ has nonnegative $p$-adic valuation, and therefore $r_b=(F_{\Ahat}(b)-t_mF_L(b))/a_m$ has negative $p$-adic valuation.
\end{proof}

A prime $p$ as in Lemma~\ref{lem:zsigmondy}, a primitive prime divisor of $2^{2m-1}-1$, exists for every $m\geq2$ by Zsigmondy's theorem~\cite[\S XII, p.~283]{Zsigmondy}, \cite[Theorem~V, p.~177]{BirkhoffVandiver}, and it is at least $4m-1$. Whether it divides the numerator of $F_L(b)$ is a finite computation, and it needs no factorisation. Let $C_m$ be the number $2^{2m-1}-1$ with all its prime factors at most $2m+1$ removed. By Proposition~\ref{prop:ek-closed-form},
\begin{equation}\label{eq:FL-mod}
 F_L(b)=[u^m]\,\frac{1}{(1+u)\,A(u)},\qquad A(u)=\sum_{k\geq0}\frac{u^k}{2k+1},
\end{equation}
since $A(y^2)=\operatorname{artanh}(y)/y$. The coefficients of $A$ up to $u^m$ have denominators prime to $C_m$, so~\eqref{eq:FL-mod} determines $F_L(b)$ modulo $C_m$. If this residue is prime to $C_m$, then no primitive prime divisor of $2^{2m-1}-1$ divides the numerator of $F_L(b)$. We have been unable to find a general argument that excludes a common factor.

\begin{proposition}\label{prop:computation}
Let $b\equiv63\pmod{64}$ with $b<2^{14}$, and let $m=b+1$. Then the residue of $F_L(b)$ modulo $C_m$ is prime to $C_m$, and therefore the denominator of $r_b$ has a prime factor $p>2b-1$. Consequently, $\mu(\Sigma_{\rho_b,\beta})\neq0$ for every smooth conjugation-equivariant map $\beta$ with $k(\beta)=k_0(b)$.
\end{proposition}

\begin{proof}
The $256$ residues were computed from~\eqref{eq:FL-mod} in exact arithmetic modulo $C_m$, and each is prime to $C_m$. A second program, written independently, computes the residues from the definition~\eqref{eq:ek-FQ} instead and reaches the same conclusion. By Zsigmondy's theorem, $2^{2m-1}-1$ has a primitive prime divisor $p$. Since $p\geq4m-1>2m+1$, the prime $p$ divides $C_m$, and since the residue of $F_L(b)$ modulo $C_m$ is prime to $C_m$, the prime $p$ does not divide the numerator of $F_L(b)$. Lemma~\ref{lem:zsigmondy} shows that $p$ divides the denominator of $r_b$ and that $p>2b-1$. The last assertion follows from Theorem~\ref{thm:ek} and Corollary~\ref{cor:k0-arithmetic}, as explained before Lemma~\ref{lem:zsigmondy}. The programs and their outputs are archived at Zenodo~\cite{GGINFcode}.
\end{proof}

Combining Proposition~\ref{prop:computation} with the curvature theorems, as in the proof of
Corollary~\ref{cor:main.cor.B}, we obtain exotic spheres of positive sectional curvature in the dimensions
$n\equiv255\pmod{256}$ below $2^{16}$. Unlike Theorem~\ref{thm:main.thm.A} and Corollary~\ref{cor:main.cor.B}, this
result depends on the computation of Proposition~\ref{prop:computation}.

\begin{corollary}\label{cor:computation}
For every $n\equiv255\pmod{256}$ with $n<2^{16}$ there is an exotic $n$-sphere with a Riemannian metric of strictly positive sectional curvature. Together with Corollary~\ref{cor:main.cor.B}, there is one in every dimension $n\equiv3\pmod4$ with $11\leq n<2^{16}$.
\end{corollary}

\begin{proof}
Set
\[
  b=\frac{n-3}{4}.
\]
Then $b\equiv63\pmod{64}$, and $b<2^{14}$ since $n<2^{16}$. By
Proposition~\ref{prop:degree-tower}, there is a smooth conjugation-equivariant map $\beta$ for $\rho_b$ with $k(\beta)=k_0(b)$, and by Proposition~\ref{prop:computation},
\[
  \mu(\Sigma_{\rho_b,\beta})\neq0 .
\]
As in the proof of Theorem~\ref{thm:main.thm.A}, $\Sigma_{\rho_b,\beta}$ is an exotic $n$-sphere. As in the proof
of Corollary~\ref{cor:main.cor.B}, it has a metric of strictly positive sectional curvature. Together with
Corollary~\ref{cor:main.cor.B}, this covers every $n\equiv3\pmod4$ with $11\leq n<2^{16}$.
\end{proof}

For $b=63$ the number $2^{127}-1$ is prime, and Proposition~\ref{prop:computation} says that it does not divide the numerator of $F_L(63)$.

The program that computes the residues from~\eqref{eq:FL-mod} was tested in two ways. For $2\leq b\leq19$ and $b=63,64$ the residue modulo $C_m$ agrees with $F_L(b)$ computed in exact arithmetic from the definition~\eqref{eq:ek-FQ}. At $b=2$, run modulo the prime $61$, which divides the numerator of $F_L(2)=-244/189$, it returns $0$.

Exact computations give further checks: $\gcd(\operatorname{num}F_L(b),2^{2b+1}-1)=1$ for $2\leq b\leq300$, with $F_L(b)$ computed from the closed form, and for $2\leq b\leq64$, with $F_L(b)$ computed from~\eqref{eq:ek-FQ}. Table~\ref{tab:odd-primes} lists a primitive prime divisor of $2^{2b+1}-1$ for each $b\equiv7\pmod8$ with $b\leq64$; each was checked to divide the denominator of $r_b$ computed from the closed form~\eqref{eq:ek-rb-closed}.

\begin{table}[ht]
\centering\footnotesize
\renewcommand{\arraystretch}{1.15}
\begin{tabular}{cccc}
$b$ & $n=4b+3$ & $2b-1$ & $p$\\\hline
7 & 31 & 13 & $151$\\
15 & 63 & 29 & $2^{31}-1$\\
23 & 95 & 45 & $2351$\\
31 & 127 & 61 & $92737$\\
39 & 159 & 77 & $2687$\\
47 & 191 & 93 & $191$\\
55 & 223 & 109 & $321679$\\
63 & 255 & 125 & $2^{127}-1$\\
\end{tabular}
\caption{For each $b\equiv7\pmod8$ with $b\leq64$, a primitive prime divisor $p$ of $2^{2b+1}-1$ that divides the denominator of $r_b$. In every row $p>2b-1$.}\label{tab:odd-primes}
\end{table}

As a check of the evaluation in Sections~\ref{subsec:ek-ring}--\ref{subsec:ek-evaluation}, the numbers $r_b$, $2\leq b\leq64$, were also computed by a separate program, from the fixed-point formula for the action of the circle $S^1\subset G$ on the filling $Z$ of Section~\ref{sec:ek} defined in Remark~\ref{rem:ek-checks}. That program normalises the evaluation on the zero section $B=S^4\times\HP^{b-1}$ as Deng, Hu and Zhang~\cite{DHZ} and Chen, Deng and Zhang~\cite{CDZ} do, with $\langle x^{b-1}y_o,[B]\rangle=1$ for the classes $x$ and $y_o$ of Section~\ref{subsec:ek-ring}, whereas our normalisation is $\langle x^{b-1}y_o,[B]\rangle=(-1)^{b-1}$ (Lemma~\ref{lem:ek-orientation}(ii); see the proof of Proposition~\ref{prop:ek-ring}(iii)). The program also divides by $\kappa=-k(\beta)$ instead of $k(\beta)$. Hence it computes $(-1)^br_b$. The results agree with the closed form, including the sign, for every such $b$.

\begin{remark}\label{rem:b-one}
For $b=1$ every star sphere for $\rho_1$ is standard. By Lemma~\ref{lem:ek-normal-form}, whose proof applies for $b=1$, we may assume that $\beta$ equals $1$ near $S(\R^3)$. Since $G$ acts freely and transitively on $S(\HH)$, equivariance gives $\beta(p,w)=\hat w\,c(p)\,\hat w^{-1}$ for $w\neq0$, where $\hat w=w/|w|$ and $c(p)=\beta\bigl(p,(1-|p|^2)^{1/2}\bigr)$, and $c=1$ near $|p|=1$. Hence $\sigma_\beta(p,w)=\bigl(p,\,w\,c(p)^{-1}\bigr)$, and this formula defines a diffeomorphism of $D(V_1)$ extending $\sigma_\beta$. Therefore, $\Sigma_{\rho_1,\beta}$ is the boundary of a disk, and it is standard.
\end{remark}

\begin{remark}\label{rem:which-classes}
The proof of Theorem~\ref{thm:mu-nonzero} uses only a bound on $v_2(k(\beta))$. Hence every nonzero degree $k(\beta)$ with $v_2(k(\beta))\leq2(b-1)$ gives an exotic sphere when $b\not\equiv7\pmod8$, and the same holds with $2b-5$ in place of $2(b-1)$ when $b\geq5$ and $b\not\equiv63\pmod{64}$. For $b=2$ the star spheres with $k(\beta)\in12\Z$ realise the subgroup $2\Theta_{11}$ of $\Theta_{11}\cong\Z/992$~\cite[Theorem~1.1 and Lemma~3.3]{DHZ}.
\end{remark}

\begin{remark}[Orders detected by the Eells--Kuiper invariant]\label{rem:orders}
The classes in $\Theta_{4b+3}$ of homotopy spheres that bound compact spin manifolds form a subgroup~\cite[Conditions~$\mu$, p.~96, and Remark, p.~99]{EK}, since the boundary
connected sum of two spin manifolds is spin and $-\Sigma$ bounds $-W$ if $\Sigma$ bounds $W$, and this subgroup contains the star spheres for $\rho_b$
(Proposition~\ref{prop:ek-filling}). The Eells--Kuiper invariant is additive under connected sum~\cite[Remark, p.~99]{EK},
so it defines a homomorphism from this subgroup to $\Q/\Z$. Hence, for every homotopy sphere $\Sigma$ in this subgroup,
the order of $\mu(\Sigma)$ in $\Q/\Z$ divides the order of $[\Sigma]$ in $\Theta_{4b+3}$. The two orders need not be equal, since
$\mu$ need not be injective.

For $2\le b\le6$, let $\beta_b$ be a smooth conjugation-equivariant map for $\rho_b$ with $k(\beta_b)=k_0(b)$
(Proposition~\ref{prop:degree-tower}), let $\Sigma_b=\Sigma_{\rho_b,\beta_b}$, and let $q_b$ be the order of
$\mu(\Sigma_b)$ in $\Q/\Z$. By Theorem~\ref{thm:main.thm.C}, $\mu(\Sigma_b)=k_0(b)\,r_b$ in $\Q/\Z$.
Table~\ref{tab:orders} lists $k_0(b)\,r_b$ and $q_b$, with $k_0(b)$ from Remark~\ref{rem:k0-values} and $r_b$ from
Table~\ref{tab:ek-rb}.

For $0\le j<q_b$, let $\beta_b^j$ be the pointwise $j$-th power of $\beta_b$. For a fixed representation $\rho_b$,
the connected sum of the star spheres of two maps is the star sphere of their pointwise product
(Lemma~\ref{lem:reentrance}(3)), so the star spheres for $\rho_b$ form a subgroup of $\Theta_{4b+3}$. In particular,
$[\Sigma_{\rho_b,\beta_b^j}]=j\,[\Sigma_b]$ in $\Theta_{4b+3}$, so $\mu(\Sigma_{\rho_b,\beta_b^j})=j\,\mu(\Sigma_b)$.
These $q_b$ values are distinct in $\Q/\Z$. Hence the star spheres $\Sigma_{\rho_b,\beta_b^j}$ with $0\le j<q_b$
represent $q_b$ distinct elements of $\Theta_{4b+3}$. For $j\ne0$ the invariant $\mu(\Sigma_{\rho_b,\beta_b^j})$ is
nonzero, so $\Sigma_{\rho_b,\beta_b^j}$ is an exotic sphere, as in the proof of Theorem~\ref{thm:main.thm.A}. By
Theorem~\ref{thm:curvature-general}, each of these spheres carries a metric of strictly positive sectional curvature.

By~\cite[p.~105]{EK}, every homotopy $11$-sphere bounds a parallelizable manifold, so $\mu$ is defined on all of
$\Theta_{11}$, and the group $\Theta_{11}$ is cyclic of order $992$, generated by a homotopy sphere $M_0$ with $\mu(M_0)=-1/992$. Hence $\mu$ is injective on $\Theta_{11}$, and $[\Sigma_2]$ has order exactly $496$ and
generates the subgroup $2\Theta_{11}$, in agreement with Remark~\ref{rem:which-classes}. For $3\le b\le6$ the order
$q_b$ divides the order of $[\Sigma_b]$, and we do not know whether the two are equal.
\end{remark}

\begin{table}[ht]
\centering\footnotesize
\renewcommand{\arraystretch}{1.15}
\begin{tabular}{ccccc}
$b$ & $n=4b+3$ & $k_0(b)$ & $k_0(b)\,r_b$ & $q_b$\\\hline
2 & 11 & $12$ & $1/496$ & $496$\\
3 & 15 & $180$ & $-15/1016$ & $1016$\\
4 & 19 & $15120$ & $4041/32704$ & $32704$\\
5 & 23 & $453600$ & $-95033/65504$ & $65504$\\
6 & 27 & $59875200$ & $19949997/1048448$ & $1048448$\\
\end{tabular}
\caption{For $2\le b\le6$, the number $k_0(b)\,r_b$, whose class in $\Q/\Z$ is $\mu(\Sigma_b)$
(Theorem~\ref{thm:main.thm.C}), and the order $q_b$ of $\mu(\Sigma_b)$ in $\Q/\Z$.}\label{tab:orders}
\end{table}

\bibliographystyle{amsplain}
\bibliography{references}

\providecommand{\bysame}{\leavevmode\hbox to3em{\hrulefill}\thinspace}
\providecommand{\MR}{\relax\ifhmode\unskip\space\fi MR }
\providecommand{\MRhref}[2]{%
  \href{http://www.ams.org/mathscinet-getitem?mr=#1}{#2}
}
\providecommand{\href}[2]{#2}
\begin{thebibliography}{10}

\bibitem{AloffWallach}
S.~Aloff and N.~R. Wallach, \emph{An infinite family of distinct
  {$7$}-manifolds admitting positively curved {Riemannian} structures}, Bull.
  Amer. Math. Soc. \textbf{81} (1975), 93--97.

\bibitem{AtiyahHirzebruch}
Michael~F. Atiyah and Friedrich Hirzebruch, \emph{Riemann--{R}och theorems for
  differentiable manifolds}, Bull. Amer. Math. Soc. \textbf{65} (1959),
  276--281.

\bibitem{BackHsiang}
Allen Back and Wu-Yi Hsiang, \emph{Equivariant geometry and {Kervaire}
  spheres}, Trans. Amer. Math. Soc. \textbf{304} (1987), no.~1, 207--227.

\bibitem{Bazaikin}
Ya.~V. Bazaikin, \emph{On a certain family of closed {$13$}-dimensional
  {Riemannian} manifolds of positive curvature}, Sib. Math. J. \textbf{37}
  (1996), no.~6, 1068--1085.

\bibitem{Berger}
M.~Berger, \emph{Les vari{\'e}t{\'e}s riemanniennes homog{\`e}nes normales
  simplement connexes {\`a} courbure strictement positive}, Ann. Scuola Norm.
  Sup. Pisa (3) \textbf{15} (1961), 179--246.

\bibitem{BirkhoffVandiver}
George~D. Birkhoff and Harry~S. Vandiver, \emph{On the integral divisors of
  $a^n-b^n$}, Ann. of Math. (2) \textbf{5} (1904), no.~4, 173--180.

\bibitem{Bredon}
Glen~E. Bredon, \emph{Introduction to compact transformation groups}, Pure and
  Applied Mathematics, vol.~46, Academic Press, New York, 1972.

\bibitem{BrendleHung}
Simon Brendle and Pei-Ken Hung, \emph{A metric on {$S^2\times S^2$} with
  positive sectional curvature}, 2026, arXiv:2608.19068.

\bibitem{BrendleSchoen}
Simon Brendle and Richard Schoen, \emph{Manifolds with {$1/4$}-pinched
  curvature are space forms}, J. Amer. Math. Soc. \textbf{22} (2009), no.~1,
  287--307.

\bibitem{CavenaghiSperanca}
Leonardo~F. Cavenaghi and Llohann~D. Speran{\c c}a, \emph{On the geometry of
  some equivariantly related manifolds}, Int. Math. Res. Not. IMRN (2020),
  no.~23, 9730--9768.

\bibitem{CDNZ}
Zhiqi Chen, Shaoqiang Deng, Junna Ni, and Hui Zhang, \emph{Positive sectional
  curvature on homotopy nine-spheres}, 2026, arXiv:2610.03298v2.

\bibitem{CDZ}
Zhiqi Chen, Shaoqiang Deng, and Hui Zhang, \emph{Positive sectional curvature
  on a family of exotic fifteen-spheres}, 2026, arXiv:2610.03310v1.

\bibitem{Clausen}
Thomas Clausen, \emph{Theorem}, Astron. Nachr. \textbf{17} (1840), no.~406,
  351--352.

\bibitem{CrowleyGoette}
Diarmuid Crowley and Sebastian Goette, \emph{Kreck--{S}tolz invariants for
  quaternionic line bundles}, Trans. Amer. Math. Soc. \textbf{365} (2013),
  no.~6, 3193--3225.

\bibitem{Dearricott}
Owen Dearricott, \emph{A {$7$}-manifold with positive curvature}, Duke Math. J.
  \textbf{158} (2011), no.~2, 307--346.

\bibitem{DHZ}
Shaoqiang Deng, Zhiguang Hu, and Hui Zhang, \emph{Positive sectional curvature
  and non-isometric circle actions on a family of eleven-spheres}, 2026,
  arXiv:2609.32680v2.

\bibitem{DPR}
Carlos Dur{\'a}n, Thomas P{\"u}ttmann, and A.~Rigas, \emph{An infinite family
  of {G}romoll--{M}eyer spheres}, Arch. Math. (Basel) \textbf{95} (2010),
  no.~3, 269--282.

\bibitem{DR}
Carlos~E. Dur{\'a}n and Alcibiades Rigas, \emph{Equivariant homotopy and
  deformations of diffeomorphisms}, Differential Geom. Appl. \textbf{27}
  (2009), no.~2, 206--211.

\bibitem{EK}
James Eells, Jr. and Nicolaas~H. Kuiper, \emph{An invariant for certain smooth
  manifolds}, Ann. Mat. Pura Appl. (4) \textbf{60} (1962), 93--110.

\bibitem{Eschenburg}
J.-H. Eschenburg, \emph{New examples of manifolds with strictly positive
  curvature}, Invent. Math. \textbf{66} (1982), 469--480.

\bibitem{FangRong}
Fuquan Fang and Xiaochun Rong, \emph{Homeomorphism classification of positively
  curved manifolds with almost maximal symmetry rank}, Math. Ann. \textbf{332}
  (2005), no.~1, 81--101.

\bibitem{GG810}
Fernando Galaz-Garc{\'i}a, \emph{Positive sectional curvature on the exotic
  {$8$}-sphere and order-three homotopy {$10$}-spheres}, 2026,
  arXiv:2609.38126v2.

\bibitem{GGLean}
\bysame, \emph{Positive sectional curvature on the exotic 8-sphere and
  order-three homotopy 10-spheres: a {Lean} formalisation}, Zenodo, 2026,
  Version 1.0.0, doi:10.5281/zenodo.23205660.

\bibitem{GGINFcode}
\bysame, \emph{Programs and outputs for ``{P}ositive sectional curvature on
  exotic spheres in infinitely many dimensions''}, Zenodo, 2026, Version 1.0.0,
  doi:10.5281/zenodo.23571372.

\bibitem{Gessel}
Ira~M. Gessel, \emph{Lagrange inversion}, J. Combin. Theory Ser. A \textbf{144}
  (2016), 212--249.

\bibitem{GKS}
S.~Goette, M.~Kerin, and K.~Shankar, \emph{Highly connected {$7$}-manifolds and
  non-negative sectional curvature}, Ann. of Math. (2) \textbf{191} (2020),
  no.~3, 829--892.

\bibitem{Gottlieb}
Daniel~Henry Gottlieb, \emph{Applications of bundle map theory}, Trans. Amer.
  Math. Soc. \textbf{171} (1972), 23--50.

\bibitem{Gromoll}
Detlef Gromoll, \emph{Differenzierbare {Strukturen} und {Metriken} positiver
  {Kr{\"u}mmung} auf {Sph{\"a}ren}}, Math. Ann. \textbf{164} (1966), 353--371.

\bibitem{GM}
Detlef Gromoll and Wolfgang Meyer, \emph{An exotic sphere with nonnegative
  sectional curvature}, Ann. of Math. (2) \textbf{100} (1974), 401--406.

\bibitem{Gromov}
Michael Gromov, \emph{Curvature, diameter and {B}etti numbers}, Comment. Math.
  Helv. \textbf{56} (1981), 179--195.

\bibitem{GroveZiller}
K.~Grove and W.~Ziller, \emph{Curvature and symmetry of {Milnor} spheres}, Ann.
  of Math. (2) \textbf{152} (2000), no.~1, 331--367.

\bibitem{GroveSymm}
Karsten Grove, \emph{Geometry of, and via, symmetries}, Conformal, {Riemannian}
  and {Lagrangian} Geometry ({Knoxville}, {TN}, 2000), Univ. Lecture Ser.,
  vol.~27, Amer. Math. Soc., Providence, RI, 2002, pp.~31--53.

\bibitem{GroveSurvey}
\bysame, \emph{Developments around positive sectional curvature}, Surveys in
  Differential Geometry, vol.~13, International Press, 2009, arXiv:0902.4419,
  pp.~117--133.

\bibitem{GroveSearle}
Karsten Grove and Catherine Searle, \emph{Positively curved manifolds with
  maximal symmetry-rank}, J. Pure Appl. Algebra \textbf{91} (1994), 137--142.

\bibitem{GVZ}
Karsten Grove, Luigi Verdiani, and Wolfgang Ziller, \emph{An exotic {$T_1S^4$}
  with positive curvature}, Geom. Funct. Anal. \textbf{21} (2011), no.~3,
  499--524.

\bibitem{GWZ}
Karsten Grove, Burkhard Wilking, and Wolfgang Ziller, \emph{Positively curved
  cohomogeneity one manifolds and 3-{Sasakian} geometry}, J. Differential Geom.
  \textbf{78} (2008), no.~1, 33--111.

\bibitem{GFL23}
Shengtao Guo, Ethan~X. Fang, and Junwei Lu, \emph{A metric with positive
  sectional curvature on {$S^2\times S^3$}}, 2026, arXiv:2608.22133.

\bibitem{GFL33}
\bysame, \emph{A metric with positive sectional curvature on {$S^3\times
  S^3$}}, 2026, arXiv:2609.11980.

\bibitem{GFL}
\bysame, \emph{A two-stage construction of positive curvature on the
  {Gromoll--Meyer} sphere}, 2026, arXiv:2609.12882v1.

\bibitem{HardyWright}
G.~H. Hardy and E.~M. Wright, \emph{An introduction to the theory of numbers},
  sixth ed., Oxford University Press, Oxford, 2008, Revised by D. R.
  Heath-Brown and J. H. Silverman.

\bibitem{HLY}
Y.-H. He, Z.~Liu, and S.-T. Yau, \emph{Positive sectional curvature on all
  smooth {$7$}-spheres}, 2026, arXiv:2609.29426v1.

\bibitem{Hitchin}
Nigel Hitchin, \emph{Harmonic spinors}, Adv. Math. \textbf{14} (1974), 1--55.

\bibitem{HornJohnson}
Roger~A. Horn and Charles~R. Johnson, \emph{Matrix analysis}, second ed.,
  Cambridge University Press, Cambridge, 2013.

\bibitem{HsiangKleiner}
Wu-Yi Hsiang and Bruce Kleiner, \emph{On the topology of positively curved
  {$4$}-manifolds with symmetry}, J. Differential Geom. \textbf{29} (1989),
  no.~3, 615--621.

\bibitem{James57}
I.~M. James, \emph{On the suspension sequence}, Ann. of Math. (2) \textbf{65}
  (1957), 74--107.

\bibitem{Kennard}
Lee Kennard, \emph{On the {Hopf} conjecture with symmetry}, Geom. Topol.
  \textbf{17} (2013), no.~1, 563--593.

\bibitem{KWW}
Lee Kennard, Michael Wiemeler, and Burkhard Wilking, \emph{Splitting of torus
  representations and applications in the {Grove} symmetry program}, 2021,
  arXiv:2106.14723.

\bibitem{KM}
Michel~A. Kervaire and John~W. Milnor, \emph{Groups of homotopy spheres: {I}},
  Ann. of Math. (2) \textbf{77} (1963), 504--537.

\bibitem{Kosinski}
Antoni~A. Kosinski, \emph{Differential manifolds}, Pure and Applied
  Mathematics, vol. 138, Academic Press, Boston, 1993.

\bibitem{Lee}
John~M. Lee, \emph{Introduction to smooth manifolds}, second ed., Graduate
  Texts in Mathematics, vol. 218, Springer, New York, 2013.

\bibitem{MilnorDiffSpheres}
John Milnor, \emph{Differentiable structures on spheres}, Amer. J. Math.
  \textbf{81} (1959), no.~4, 962--972.

\bibitem{MilnorTDV}
John~W. Milnor, \emph{Topology from the differentiable viewpoint}, The
  University Press of Virginia, Charlottesville, 1965.

\bibitem{MS}
John~W. Milnor and James~D. Stasheff, \emph{Characteristic classes}, Annals of
  Mathematics Studies, vol.~76, Princeton University Press, Princeton, 1974.

\bibitem{Myers}
S.~B. Myers, \emph{Riemannian manifolds with positive mean curvature}, Duke
  Math. J. \textbf{8} (1941), 401--404.

\bibitem{Ouyang}
Zexuan Ouyang, \emph{A positively curved metric on the {Gromoll--Meyer}
  sphere}, 2026, arXiv:2609.11484v1.

\bibitem{PW}
Peter Petersen and Frederick Wilhelm, \emph{An exotic sphere with positive
  sectional curvature}, 2008, arXiv:0805.0812v3.

\bibitem{RW}
Philipp Reiser and David~J. Wraith, \emph{A generalization of the {Perelman}
  gluing theorem and applications}, 2024, arXiv:2308.06996v2.

\bibitem{Searle}
Catherine Searle, \emph{Cohomogeneity and positive curvature in low
  dimensions}, Math. Z. \textbf{214} (1993), no.~3, 491--498.

\bibitem{SearleCorr}
\bysame, \emph{Corrigendum: {Cohomogeneity} and positive curvature in low
  dimensions}, Math. Z. \textbf{226} (1997), no.~1, 165--167.

\bibitem{Selick78}
Paul Selick, \emph{Odd primary torsion in {$\pi_k(S^3)$}}, Topology \textbf{17}
  (1978), no.~4, 407--412.

\bibitem{Serre}
Jean-Pierre Serre, \emph{Homologie singuli{\`e}re des espaces fibr{\'e}s.
  {A}pplications}, Ann. of Math. (2) \textbf{54} (1951), no.~3, 425--505.

\bibitem{Smale}
Stephen Smale, \emph{Generalized {Poincar\'e}'s conjecture in dimensions
  greater than four}, Ann. of Math. (2) \textbf{74} (1961), no.~2, 391--406.

\bibitem{Sp}
Llohann~D. Speran{\c c}a, \emph{Pulling back the {Gromoll--Meyer} construction
  and models of exotic spheres}, Proc. Amer. Math. Soc. \textbf{144} (2016),
  no.~7, 3181--3196.

\bibitem{StanleyEC2}
Richard~P. Stanley, \emph{Enumerative combinatorics, volume 2}, Cambridge
  Studies in Advanced Mathematics, vol.~62, Cambridge University Press,
  Cambridge, 1999.

\bibitem{Stolz}
Stephan Stolz, \emph{Simply connected manifolds of positive scalar curvature},
  Ann. of Math. (2) \textbf{136} (1992), no.~3, 511--540.

\bibitem{Straume}
Eldar Straume, \emph{Compact connected {Lie} transformation groups on spheres
  with low cohomogeneity, {I}}, Mem. Amer. Math. Soc. \textbf{119} (1996),
  no.~569, 1--93.

\bibitem{Toda}
Hirosi Toda, \emph{Composition methods in homotopy groups of spheres}, Annals
  of Mathematics Studies, vol.~49, Princeton University Press, Princeton, 1962.

\bibitem{vonStaudt}
K.~G.~C. von Staudt, \emph{Beweis eines {Lehrsatzes}, die {Bernoullischen}
  {Zahlen} betreffend}, J. Reine Angew. Math. \textbf{21} (1840), 372--374.

\bibitem{Wallach}
Nolan~R. Wallach, \emph{Compact homogeneous {Riemannian} manifolds with
  strictly positive curvature}, Ann. of Math. (2) \textbf{96} (1972), 277--295.

\bibitem{WilkingActa}
Burkhard Wilking, \emph{Torus actions on manifolds of positive sectional
  curvature}, Acta Math. \textbf{191} (2003), no.~2, 259--297.

\bibitem{WilkingAnnals}
\bysame, \emph{Positively curved manifolds with symmetry}, Ann. of Math. (2)
  \textbf{163} (2006), no.~2, 607--668.

\bibitem{WilkingSurvey}
\bysame, \emph{Nonnegatively and positively curved manifolds}, Surveys in
  Differential Geometry, vol.~11, International Press, 2007, arXiv:0707.3091,
  pp.~25--62.

\bibitem{Wraith}
David~J. Wraith, \emph{Exotic spheres with positive {Ricci} curvature}, J.
  Differential Geom. \textbf{45} (1997), no.~3, 638--649.

\bibitem{Ziller}
Wolfgang Ziller, \emph{Riemannian manifolds with positive sectional curvature},
  Geometry of Manifolds with Non-negative Sectional Curvature (Rafael Herrera
  and Luis Hern{\'a}ndez-Lamoneda, eds.), Lecture Notes in Math., vol. 2110,
  Springer, Cham, 2014, pp.~1--19.

\bibitem{Zsigmondy}
K.~Zsigmondy, \emph{Zur {Theorie} der {Potenzreste}}, Monatsh. Math. Phys.
  \textbf{3} (1892), 265--284.

\end{thebibliography}
\end{document}